\documentclass[a4paper,11pt]{article}
\usepackage{amsmath,amssymb,latexsym,amsthm}
\usepackage{array,graphicx,epsfig,xcolor}
\usepackage{color}

\font\elevensf=cmss10 scaled\magstephalf

\newtheorem{theorem} {\bf THEOREM}[section]

\newtheorem{lemma} {\bf LEMMA}[section]
\newtheorem{corollary} {\bf COROLLARY}[section]

\newtheorem{counter example} {\bf Counter Example}[section]

\newtheorem{remark} {\bf REMARK}[section]

\newtheorem{definition} {\bf DEFINITION}[section]

\def\R{{I\!\!R}}

\def\CC{{\rm \kern.24em \vrule width.02em height1.4ex depth-.05ex \kern-.26emC}}

\def\TagOnRight

\def\AA{{it I} \hskip-3pt{\tt A}}

\def\QQ{\rlap {\raise 0.4ex \hbox{$\scriptscriptstyle |$}} {\hskip -0.1em Q}}

\newcommand{\be} {\begin{eqnarray}}
\newcommand{\ee} {\end{eqnarray}}
\newcommand{\Bea} {\begin{eqnarray*}}
\newcommand{\Eea} {\end{eqnarray*}}
\newcommand{\pa} {\partial}

\newcommand{\al} {\alpha}
\newcommand{\rr}{\rightarrow}
\newcommand{\ti}{\tilde}

\newcommand{\B} {\beta}

\newcommand{\g} {\gamma}

\newcommand{\T}  {\theta}

\newcommand{\p}  {\prime}
\newcommand{\e}  {\epsilon}

\newcommand{\f}{\infty}
\newcommand{\h}{\label}

\catcode`\@=11
\def\theequation{\@arabic{\c@section}.\@arabic{\c@equation}}

\catcode`\@=12
\catcode`\@=11
\title{Single  shock solution for   non convex scalar conservation laws}
\author{Adimurthi\footnote{aditi@math.tifrbng.res.in}\ \ and  \ Shyam Sundar
Ghoshal\footnote{ghoshal@math.tifrbng.res.in}\\\\
Centre for Applicable Mathematics,\\
Tata Institute of Fundamental Research,\\
Post Bag No 6503, Sharadanagar,\\
Bangalore - 560065, India.}
\date{}
\begin{document}
\maketitle

\begin{abstract}
In this paper we study the finite time emergence of one shock for the solution of scalar conservation laws in one space dimension with general flux $f$. We give a necessary and sufficient condition to the initial data connecting to  flux. The proof relies on the structure theorem for the linear degenerate flux and the finer analysis of characteristic curves. 


\end{abstract}

\section{Introduction} 
\setcounter{equation}{0}
Let $f:\R\rightarrow \R$ be a locally Lipschitz function and $u_0\in L^\f(\R)$. 
Consider the initial value problem 
\begin{eqnarray}
  u_t+f\left(u\right)_x&=&0,  x\in\R, \ t>0,\label{11}\\
  u(x,0)&=&u_0(x),  x\in\R.\label{12}
\end{eqnarray}
Here $f$ is  the flux function.

\par  The above equation has a  special importance, particularly in, mathematical physics and fluid dynamics, for example in river flow, flow of gas, oil recovery,  heterogeneous media,  petroleum industry,  modeling gravity, continuous sedimentation, modeling  car in  traffic flow on a highway, semiconductor industry, etc. Some of these applications finds from convex conservation laws, for example: very famous Burgers equation where $f(u)=\frac{u^2}{2}$, is an  important convex scalar conservation laws, which has several applications. On the other hand,  Buckley-Leverett equation is one of the most vital example of a non-convex scalar conservation laws,  mainly used to understand the dynamics of two-phase flow in porous  media.
Here the flux $f$ is given by $\frac{u^2}{u^2+r(1-u)^2}$, for some constant $r>0$, where the solution $u$ could represent as a water saturation, and the fraction flow function could be the flux $f$ and the constant $r$ denotes the viscosity ratio of water and oil. 

\par  (\ref{11}), (\ref{12}) admits a unique weak solution satisfying  Kruzkov  entropy condition \cite{3, 4,  8, 9, 5, 10, 6, Serre,  Whitham}.
Through out  this paper we mean the solution to (\ref{11}), (\ref{12}) in the sense of  Kruzkov.  In general, even when the data $u_0$ is smooth, the solution can have discontinuities (shock),   a shock curve is a locally Lipschitz curve. One of the fundamental  questions in this area  is to find the number of shock curves. When the flux $f\in C^2$, uniformly convex, there are many prior results to (\ref{11}).  Lax \cite{11} obtained an explicit formula (one can also see \cite{Oleinik} and for the boundary value problem see \cite{Joseph, Lefloch}) for the solution of (\ref{11}) in his seminal paper and also proved that  if the support of $u_0$
is compact,  the solution behaves like an  $N$-wave  as $t\rightarrow\f$. Also large time behaviour of the solution of (\ref{11}) has been studied in \cite{Da5, liu2}. Schaffer \cite{8} showed  later that there exists a  set $D\subset C^\f_0(\R)$ of first category such that 
if $u_0\in D^c$, then the solution $u$ admits finitely many shock curves. Tadmor and Tassa \cite{Tadmor} constructed explicitly a dense set $D^c$ for which the number of shocks are finite.
\par  Dafermos \cite{Da4, 2, Da2} introduced a landmark notion of generalized backward characteristics  to study the structure of the solution for scalar conservation laws, whereas we give a notion of forward  characteristics \cite{1, shyamcontrol, shyamop}, namely $R$ curves, has been introduced and a detailed study has been done without using the Filipov's theory \cite{Filippov}. Previously, by  using the  forward characteristics  as building blocks a detailed study of   structure Theorem for the strictly convex conservation laws has been obtained in \cite{1}. In the present paper, we  proved first  a structure Theorem for a degenerate convex flux (see definition \ref{definition41}).  This turns out to be an important step for proving finite time emergence 
of single shock solution for non convex scalar conservation laws. 

\par Another striking result  which has fundamental importance in this  paper  
 was  obtained by   Liu \cite{7} and  Dafermos-Shearer \cite{3}, in connection with the traffic flow problem. To describe it  succinctly, 
 let $A<B, \ u_0\in L^\f(\R)$ and $\bar{u}_0\in L^\f(A,B)$ such that 
\begin{eqnarray}\label{13}
u_0(x)=\left\{\begin{array}{lll}
u_- &\mbox{if}& x<A,\\
\bar{u}_0 &\mbox{if}& A<x<B,\\
u_+, &\mbox{if}& x>B,
\end{array}\right.
\end{eqnarray}
where $u_\pm$ are constants. Under the assumption 
\begin{equation}\label{14}
u_->u_+,
\end{equation}
 they showed that there exists a $T_0>0, x_0\in\R$ such that for $t>T_0$, \\ 
\begin{figure}[ht]
        \centering
        \def\svgwidth{0.8\textwidth}
        \begingroup
    \setlength{\unitlength}{\svgwidth}
  \begin{picture}(1.4,0.55363752)%
    \put(0,0.1){\includegraphics[width=0.8\textwidth]{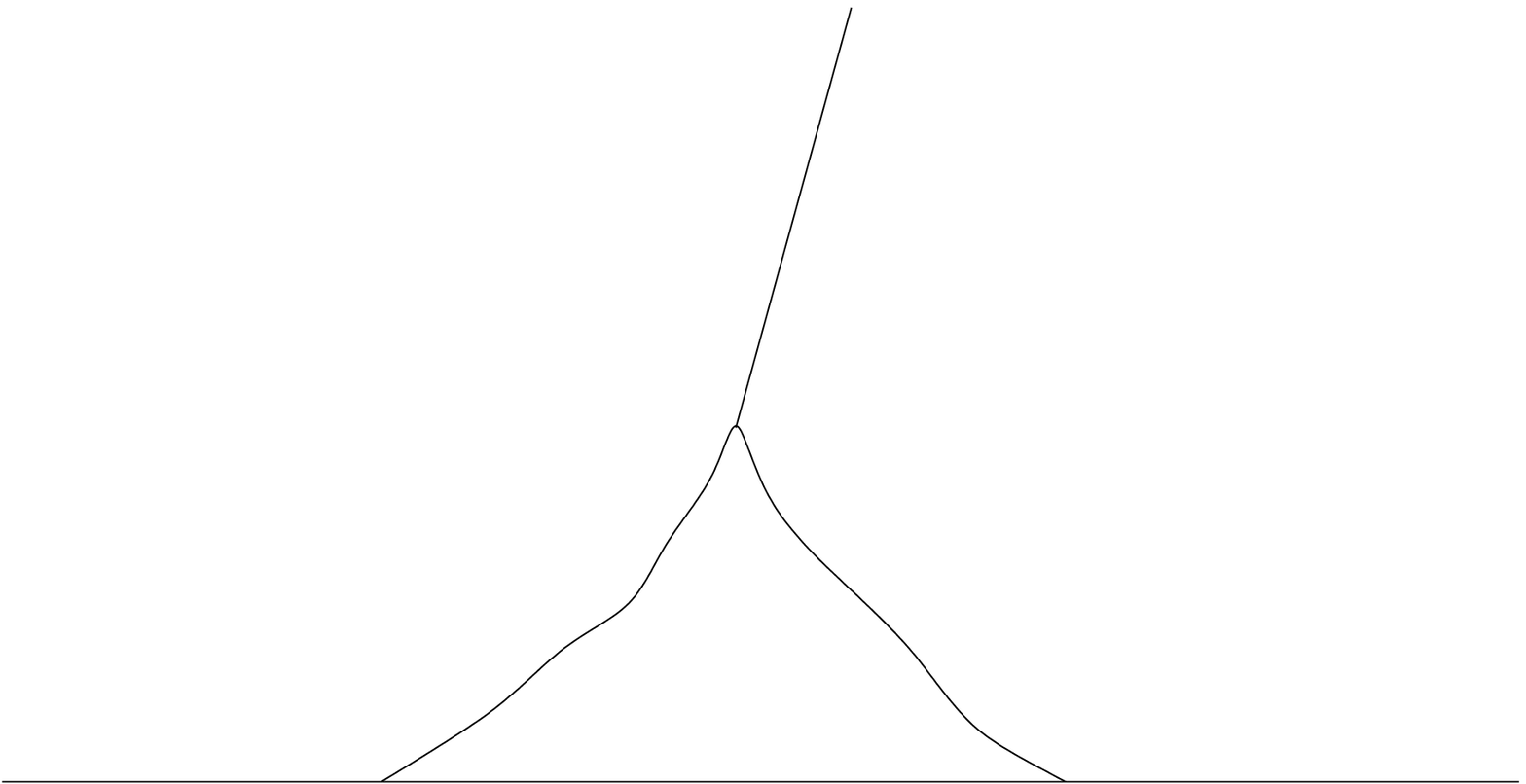}}%
    \put(0.529,0.47){\color[rgb]{0,0,0}\makebox(0,0)[lb]{\smash{$r(t):x-x_0 =\frac{f(u_-)-f(u+)}{u_--u_+}(t-T_0)$}}}%
        \put(0.3729,0.347){\color[rgb]{0,0,0}\makebox(0,0)[lb]{\smash{$(x_0,T_0)$}}}%
                \put(0.3729,0.07){\color[rgb]{0,0,0}\makebox(0,0)[lb]{\smash{$\bar{u}_0$}}}%
                                \put(0.0729,0.07){\color[rgb]{0,0,0}\makebox(0,0)[lb]{\smash{$u_-$}}}%
                                                                \put(0.7729,0.07){\color[rgb]{0,0,0}\makebox(0,0)[lb]{\smash{$u_+$}}}%
                             \put(0.2229,0.07){\color[rgb]{0,0,0}\makebox(0,0)[lb]{\smash{$A$}}}%
                                    \put(0.69229,0.07){\color[rgb]{0,0,0}\makebox(0,0)[lb]{\smash{$B$}}}%
                \put(0.090729,0.27){\color[rgb]{0,0,0}\makebox(0,0)[lb]{\smash{ $u(x,t)=u_-$}}}%
                                \put(0.690729,0.27){\color[rgb]{0,0,0}\makebox(0,0)[lb]{\smash{ $u(x,t)=u_+$}}}%

  \end{picture}
\endgroup
        \caption{Illustration for one shock solution from the point $(x_0,T_0)$, $T_0\leq \g|A-B|$}
       \label{Fig8}
\end{figure}
\begin{figure}[ht]
        \centering
        \def\svgwidth{0.8\textwidth}
        \begingroup
    \setlength{\unitlength}{\svgwidth}
  \begin{picture}(1,0.55363752)%
    \put(0,0.1){\includegraphics[width=0.8\textwidth]{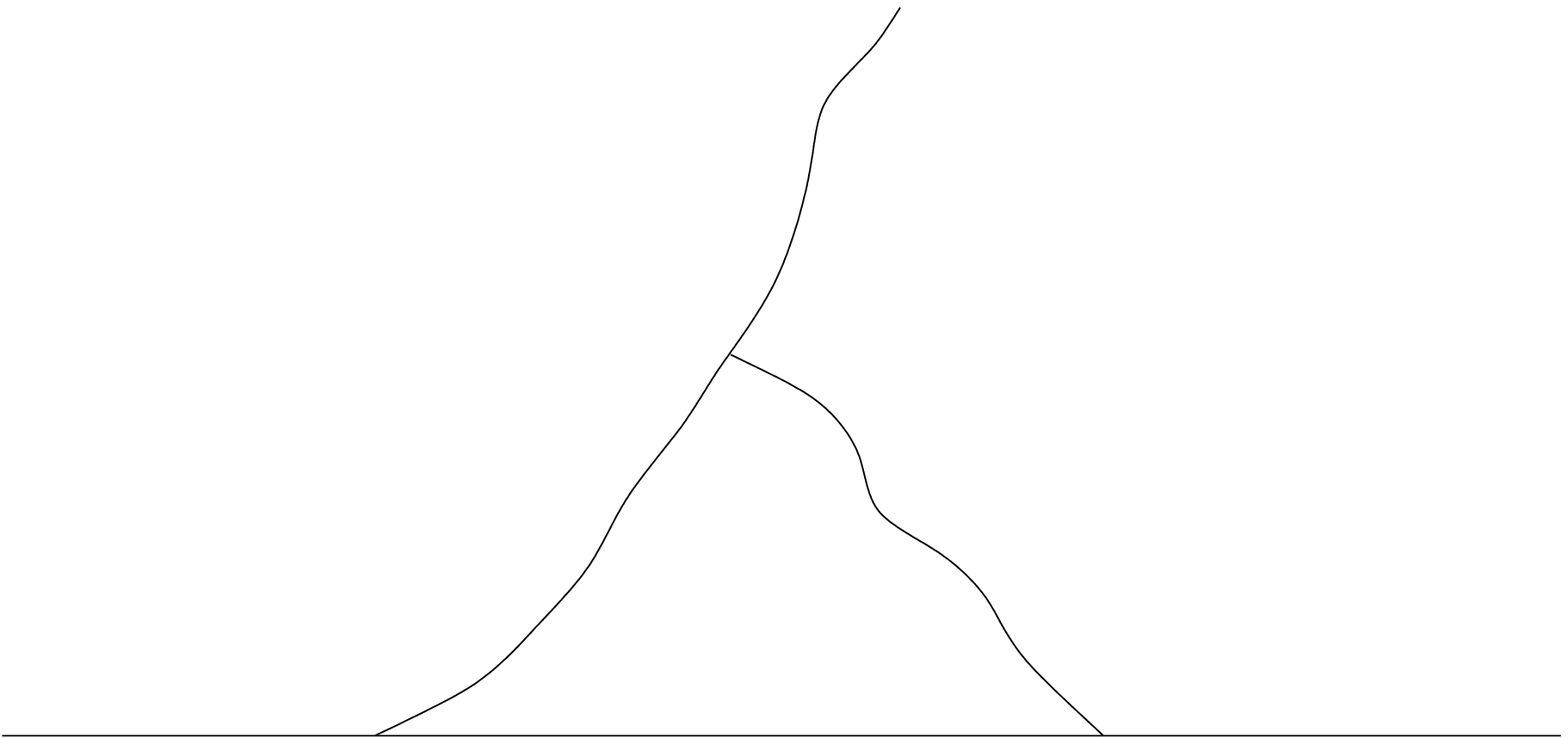}}%
    \put(0.529,0.47){\color[rgb]{0,0,0}\makebox(0,0)[lb]{\smash{$r(t)$}}}%
        \put(0.3729,0.37){\color[rgb]{0,0,0}\makebox(0,0)[lb]{\smash{$(x_0,T_0)$}}}%
                \put(0.3729,0.07){\color[rgb]{0,0,0}\makebox(0,0)[lb]{\smash{$\bar{u}_0$}}}%
                                \put(0.0729,0.07){\color[rgb]{0,0,0}\makebox(0,0)[lb]{\smash{$u_-(x)$}}}%
                                                                \put(0.7729,0.07){\color[rgb]{0,0,0}\makebox(0,0)[lb]{\smash{$u_+(x)$}}}%
                             \put(0.2229,0.07){\color[rgb]{0,0,0}\makebox(0,0)[lb]{\smash{$A$}}}%
                                    \put(0.69229,0.07){\color[rgb]{0,0,0}\makebox(0,0)[lb]{\smash{$B$}}}%
                \put(0.090729,0.27){\color[rgb]{0,0,0}\makebox(0,0)[lb]{\smash{ $u(x,t)=u_-$}}}%
                                \put(0.690729,0.27){\color[rgb]{0,0,0}\makebox(0,0)[lb]{\smash{ $u(x,t)\in \mbox{ Range }(u_+)$}}}%
  \end{picture}
\endgroup
        \caption{Illustration for one shock solution from the point $(x_0,T_0)$, $T_0\leq \g|A-B|$}
\label{Fig1}\end{figure}
\begin{eqnarray}\label{15}
 u(x,t)=\left\{\begin{array}{lll}
             u_- &\mbox{if}& x<x_0+\frac{f(u_+)-f(u_-)}{u_+-u_-}(t-T_0),\\
                          u_+ &\mbox{if}& x>x_0+\frac{f(u_+)-f(u_-)}{u_+-u_-}(t-T_0).
            \end{array}\right.
\end{eqnarray}
That is,  they proved that after  finite time there will be exactly one shock and the solution will become $u_-$ and $u_+$, this is called a single shock solution. 
 So if the left most characteristics  speed of the data is bigger than the right most characteristics speed then  these two characteristics will dominate the other waves in finite time. The main ingredient of this proof is the comparison principle of the conservation laws. 
 The case $u_-\leq u_+$ was left open in Dafermos-Shearer \cite{3}. In \cite{1}, using the finer analysis of forward characteristic curves, a unified approach has been developed to tackle both the cases, namely $u_->u_+$ and $u_-\leq u_+$,  for any $C^1$ strictly convex flux.
  
 \par Very few results are known when one consider  genuinely non linear and linear degeneracy together. When the flux function is non convex not much literature is known in this subject. Even solving a Riemann problem for flux having finitely many inflection points, one has to consider either a convex hull or a concave hull and it purely depends on the data. Therefore even for  a piecewise constant data it is difficult to keep  into account the interaction between waves.  Thus it is one of the  reasons, the structure of the solution for  non  scalar convex conservation laws are less  known.
 
 \par  In this paper we have resolved the  following  important  questions for nonconvex flux:
\begin{itemize}

\item [1.] Suppose $f$ is not convex, what is the  condition
 such that $u$ becomes single shock solution (see definition \ref{definition11} and figures \ref{Fig1}, \ref{Fig8}) in finite time?

\item [2.] Under which condition(s) $u$ admits a single shock solution when one consider $u_\pm$ to be a function in $L^\f(\R\setminus(A,B))$ instead of constant? \end{itemize}

 \par To  answer the above questions,  we first  proved the structure Theorem for the flux $f\in  C^1$ but need not be strictly convex.  Also we give a proof of a generalized version of the one shock solution, namely we not only consider constant data away from a compact set, also we consider general data $u_-
(x)$ and $u_+(x)$ and give a necessary and sufficient criteria that in a finite time $T_0$, this will be separated by a Lipschitz curve.   We have used heavily front tracking analysis, convex analysis for the degenerate convex flux and $L^1$ contraction.

\par One related topic in this direction is to understand the stability estimates. For the uniformly convex flux, $L^2$ stability result has been obtained by \cite{leger} for the shock situation. He showed $L^2$ norm of a perturbed solution can be bounded by the $L^2$ norm of initial perturbation. Later in \cite{shyam3},  $L^p$ stability for the shock and non-shock situation  has been studied using the structure Theorem \cite{1} for more general convex flux. Interested reader  can also see \cite{leger2} for the system case.
\par On a slightly different direction, structure, regularity of the entropy solution has been studied in \cite{ shyam4, Ambrosio, Bianchini, cheng1, cheng2, Delellis, Jenssen, Marconi} and the references therein for the related work. For the  flux with one inflection point, a class of  attainable sets has been obtained in \cite{Boris}. 

\subsection{Main results}

\par Before stating the main result, we need the following definitions and hypothesis. 
\begin{definition}[Single shock solution]\h{definition11}\label{sss} Let $u_\pm\in L^\f(\R)$ and $A\leq B, \bar{u}_0\in L^\f(A,B)$ and $f$ be a  locally Lipschitz function. 
\begin{eqnarray}\label{16}
u_0(x)=\left\{\begin{array}{lll}
u_-(x) &\mbox{if}& x<A,\\
\bar{u}_0(x) &\mbox{if}& A<x<B,\\
u_+(x), &\mbox{if}& x>B,
\end{array}\right.
\end{eqnarray}
let $u$ be the solution of (\ref{11}) with the initial data (\ref{16}). Then the pair $(u_-(x),u_+(x))$ gives rise to a single shock solution if for every $\bar{u}_0\in L^\f(A,B),$
there exist $(x_0,t_0)\in \R\times(0,\f)$, $\g>0,$ a Lipschitz curve $r:[T_0,\f)\rightarrow \R$ depending on $u_\pm(x),$ $||u_0||_{\f}$ such that for $t>T_0$,
\begin{eqnarray}
T_0&\leq& \g|A-B|,\label{17}\\
r(T_0)&=&x_0, \label{18}
\end{eqnarray}
\begin{eqnarray}
 u(x,r(t))\in\mbox{Range of }\ u_- \ &\mbox{for}&\ x<r(t),\label{19}\\
 u(x,r(t))\in \mbox{Range of }\ u_+\ &\mbox{for}&\ x>r(t). \label{110}
\end{eqnarray}
\end{definition}
Observe that if $f$ is uniformly convex, then by Liu's and Dafermos-Shearer's result, $(u_+,u_-)$ gives rise to single shock solution provided it satisfies 
(\ref{14}). 
\begin{figure}[ht]        \centering
        \def\svgwidth{0.8\textwidth}
        \begingroup
    \setlength{\unitlength}{\svgwidth}
  \begin{picture}(1.4,0.655363752)%
    \put(0,0.1){\includegraphics[width=0.95\textwidth]{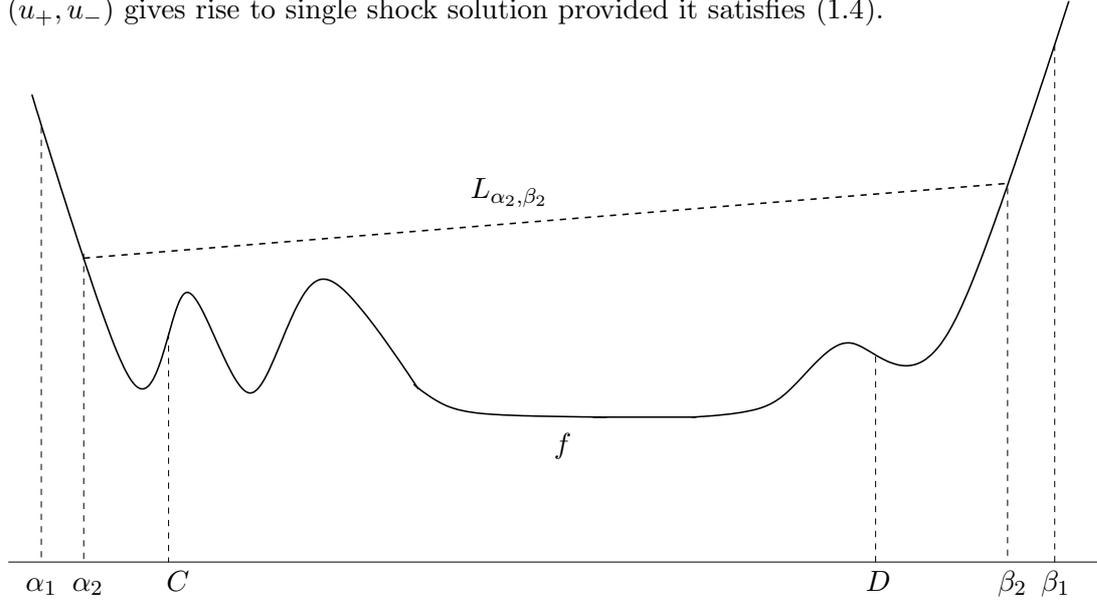}}%
    \put(0.02,0.07){\color[rgb]{0,0,0}\makebox(0,0)[lb]{\smash{$\al_1$}}}%
        \put(0.07,0.07){\color[rgb]{0,0,0}\makebox(0,0)[lb]{\smash{$\al_2$}}}%
            \put(0.172,0.07){\color[rgb]{0,0,0}\makebox(0,0)[lb]{\smash{$C$}}}%
                \put(0.9263,0.07){\color[rgb]{0,0,0}\makebox(0,0)[lb]{\smash{$D$}}}%
                    \put(1.068751,0.07){\color[rgb]{0,0,0}\makebox(0,0)[lb]{\smash{$\B_2$}}}%
    \put(1.1154,0.07){\color[rgb]{0,0,0}\makebox(0,0)[lb]{\smash{$\B_1$}}}%
        \put(0.5,0.497){\color[rgb]{0,0,0}\makebox(0,0)[lb]{\smash{$L_{\al_2,\B_2}$}}}%
            \put(0.59,0.22){\color[rgb]{0,0,0}\makebox(0,0)[lb]{\smash{$f$}}}%
              \end{picture}%
\endgroup
        \caption{Illustration of the condition to be one shock for the convex-convex case}
\label{Fig2}
\end{figure} 
\begin{definition} [Convex-convex type] (see figure \ref{Fig2}) Let $f\in C^1(\R)$ and $C\leq D$. Then $(f,C,D)$ is 
said to be a convex-convex type triplet if 
\begin{eqnarray}\label{112}
f|_{(-\f,C]} \mbox{ and } f|_{[D,\f)}
\end{eqnarray}
are convex functions.
\end{definition}
\begin{definition} [Convex-concave type] (see figure \ref{Fig3}) Let $f\in C^1(\R)$ and $C\leq D$. Then $(f,C,D)$ is said to be  a convex-concave type triplet if 
\begin{eqnarray}\label{113}
\begin{array}{lll}
f|_{(-\f,C]} \ \mbox{is a convex function}\\
 f|_{[D,\f)} \ \mbox{is a concave function}.
 \end{array}
\end{eqnarray}
\end{definition}
\begin{figure}[ht]        \centering
        \def\svgwidth{0.7\textwidth}
        \begingroup
    \setlength{\unitlength}{\svgwidth}
  \begin{picture}(1.4,0.95363752)%
    \put(0,0.1){\includegraphics[width=0.95\textwidth]{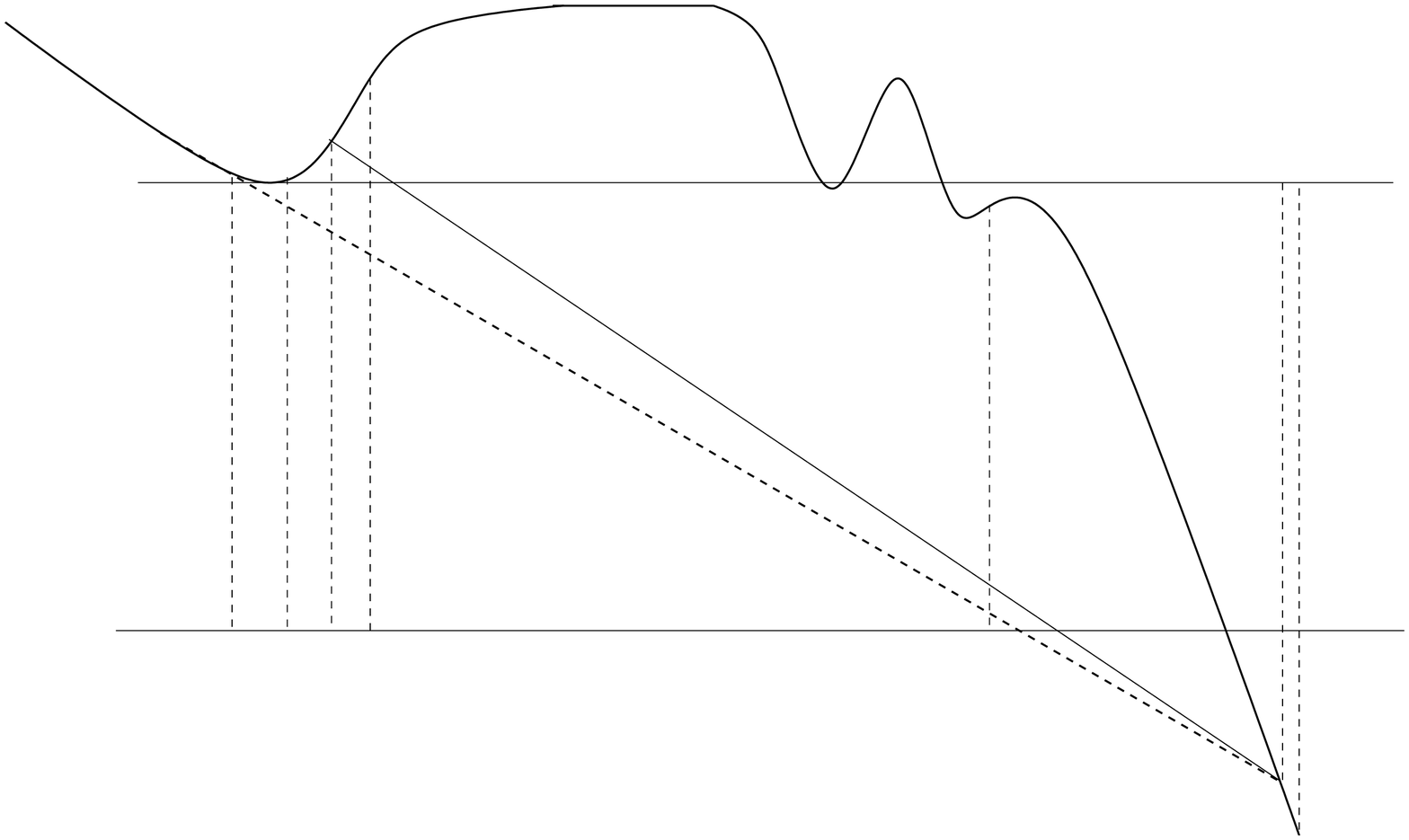}}%
           \put(0.19,0.27){\color[rgb]{0,0,0}\makebox(0,0)[lb]{\smash{$\al_0$}}}%
       \put(0.25,0.27){\color[rgb]{0,0,0}\makebox(0,0)[lb]{\smash{$\al_1$}}}%
        \put(0.3,0.27){\color[rgb]{0,0,0}\makebox(0,0)[lb]{\smash{$\al_2$}}}%
            \put(0.36,0.27){\color[rgb]{0,0,0}\makebox(0,0)[lb]{\smash{$C$}}}%
                \put(0.9263,0.27){\color[rgb]{0,0,0}\makebox(0,0)[lb]{\smash{$D$}}}%
                    \put(1.2568751,0.27){\color[rgb]{0,0,0}\makebox(0,0)[lb]{\smash{$\B_1$}}}%
    \put(1.199124,0.27){\color[rgb]{0,0,0}\makebox(0,0)[lb]{\smash{$\B_2$}}}%
        \put(0.5,0.517){\color[rgb]{0,0,0}\makebox(0,0)[lb]{\smash{$L_{\al_0}$}}}%
                \put(0.5,0.7517){\color[rgb]{0,0,0}\makebox(0,0)[lb]{\smash{$L_{\al_1}$}}}%
                                \put(1.15,0.7517){\color[rgb]{0,0,0}\makebox(0,0)[lb]{\smash{$L_{\al_1}(\B_2)$}}}%
                \put(0.639,0.586097){\color[rgb]{0,0,0}\makebox(0,0)[lb]{\smash{$L_{\al_2,\B_2}$}}}%
            \put(0.59,0.922){\color[rgb]{0,0,0}\makebox(0,0)[lb]{\smash{$f$}}}%
  \end{picture}%
\endgroup
        \caption{Illustration of the condition to be one shock for the convex-concave case}
\label{Fig3}
\end{figure}
\noindent{\it Notation:} For $a,b\in \R$, let $L_{a,b}$ be the line joining $(a,f(a))$ and $(b,f(b))$ given by 
\begin{eqnarray}\label{114}
L_{a,b}(\T)=f(a)+\frac{f(b)-f(a)}{b-a}(\T-a).
\end{eqnarray}  
Let $f^\p(a-)$ be the left derivative of $f$ at $a$ if it exists. Then define the tangent line $L_a$ at $(a,f(a))$ by 
\begin{eqnarray}\label{115}
L_a(\T)=f(a)+f^\p(a-)(\T-a).
\end{eqnarray}
\noindent{\it Hypothesis (\textbf{H}):} Let $f$ be a $C^1$ function and $\al_1\leq \al_2<C \leq D <\beta_2 \leq \beta_1$. Assume that they satisfy 
\begin{eqnarray}
[\al_1,\al_2]=\{x<C: f^\p(x)\in [f^\p(\al_1),f^\p(\al_2)]\}\label{116}\\
\ [\beta_1,\beta_2]=\{x>D: f^\p(x)\in [f^\p(\beta_1),f^\p(\beta_2)]\}\label{117}\\
\al_2+(\beta_1-\beta_2)<C\leq D<\beta_2-(\al_2-\al_1)\label{118}.
\end{eqnarray}
\begin{remark}From the maximum principle, without loss of generality, we can assume that 
\begin{eqnarray}\label{111}
\lim\limits_{|p|\rightarrow \f}\frac{|f(p)|}{|p|}=\f.
\end{eqnarray}
\end{remark}
\noindent Therefore, through out   this paper we assume that $f$ satisfies (\ref{111}).
 We have the following 
\label{maintheorem}\begin{theorem}[Main Theorem] Let $f,\al_1,\al_2,\beta_2,\beta_1, C$, $D$ satisfies (\textbf{H}) and  $u_\pm(\cdot)\in L^\f(\R).$
Assume that they satisfy any one of the following conditions:
\begin{itemize}
\item [I.] Let $(f,C,D)$ be of convex-convex type triplet and $u_\pm$ satisfies
\begin{eqnarray}
&&u_+(x)\in[\al_1,\al_2],\label{119}\\
&&u_-(x)\in  [\B_2,\B_1].\label{120}
\end{eqnarray}
For $\T\in [C,D]$
\begin{eqnarray}\label{121}
f(\T)<\min\{L_{\al_1,\B_1}(\T),L_{\al_2,\B_2}(\T), L_{\al_2+\B_1-\B_2, \B_2-(\al_2-\al_1)}(\T)\}
\end{eqnarray}
\item [II.] Let $(f,C,D)$ be of convex-concave type triplet and $u_\pm$ satisfies one of the following conditions:

\noindent Condition 1:
\begin{eqnarray}
u_+(x)\in [\B_2,\B_1],\label{123}\\
 u_-(x)\in [\al_1,\al_2].\label{124}
\end{eqnarray}
For $\T\in [C,D]$
\begin{eqnarray}
&&f(\T)>\max\{L_{\al_1,\B_1}(\T),L_{\al_2,\B_2}(\T), L_{\al_2+\B_1-\B_2, \B_2-(\al_2-\al_1)}(\T)\},\label{125}\\
&&L_{\al_1}(\B_2)>f(\B_2)\label{126}.
\end{eqnarray}
\noindent Condition 2:
\begin{eqnarray}
 u_+(x)\in [\al_1,\al_2].\label{n124}\\
 u_-(x)\in [\B_2,\B_1].\label{n123}
\end{eqnarray}
For $\T\in [C,D]$
\begin{eqnarray}
&&f(\T)<\min\{L_{\al_1,\B_1}(\T),L_{\al_2,\B_2}(\T), L_{\al_2+\B_1-\B_2, \B_2-(\al_2-\al_1)}(\T)\},\label{n125}\\
&&L_{\B_1}(\al_2)<f(\al_2)\label{n126}.
\end{eqnarray}
\end{itemize}
Then $(u_-(\cdot), u_+(\cdot))$ gives rise to a single shock solution (see figure \ref{Fig1}, \ref{Fig2}).
\end{theorem}
\begin{remark}
In the main Theorem, for the case when the $(f,C,D)$ is a convex-convex triplet, one cannot interchange the role of $u_+(x)$ and $u_-(x)$,  see second part of  the Theorem \ref{theorem28}. Whereas  for the convex-concave situation, we can interchange the role of $u_+(x)$ and $u_-(x)$,
 due to the  two different polarity. \end{remark}
\begin{remark}
In the  proof of the main Theorem we have not used the $C^1$ regularity  of the flux in $(C,D)$ as   Lipschitz regularity  is good enough in  $(C,D)$.
\end{remark}

As an immediate corollary to the main Theorem, we have the following generalization of Liu \cite{7} and Dafermos-Shearer \cite{3}.
\begin{corollary}
Let $\al_1=\al_2<C\leq D <\B_2=\B_1$. Assume  one of the following conditions hold 
\begin{itemize}
\item [I$^\p$.] Let $(f,C,D)$ be of convex-convex type such that $f$ is strictly convex on $(-\f,C)\cup (D,\f)$. Assume that  $(u_-,u_+)=(\B_1, \al_1)$ such that 
for $\T\in[C,D]$
\begin{eqnarray}\label{127}
f(\T)<L_{u_-,u_+}(\T).
\end{eqnarray}
\item [II$^{\p}$.]  Let $(f,C,D)$ be of convex-concave type triplet and $u_\pm$ satisfies one of the following conditions:

\noindent Condition 1:
\begin{eqnarray}
\begin{array}{llllll}
u_+(x)=\B_1,\  u_-(x)=\al_1,\\
f(\T)> L_{u_-,u_+}(\T), \ for \ \T\in [C,D], \label{128}\\
L_{u_-}(u_+)>f(u_+).
\end{array}
\end{eqnarray}
\noindent Condition 2:
\begin{eqnarray}
\begin{array}{llllll}
u_+(x)=\al_1,\  u_-(x)=\B_1,\\
f(\T)< L_{u_-,u_+}(\T),\ for \ \T\in [C,D],\\
L_{u_-}(u_+)<f(u_+).
\end{array}
\end{eqnarray}
\end{itemize}
Then $(u_-,u_+)$ gives rise to a single shock solution.
\begin{proof} From the strict convexity and concavity, it follows that (\ref{116}) and (\ref{117}) hold and hence (\textbf{H}). (\ref{121}), (\ref{125}), (\ref{126}) follows
from (\ref{127}) and (\ref{128}). Hence from the main Theorem, $(u_-,u_+)$ gives rise to a single shock solution. This proves the corollary. 
\end{proof}
\end{corollary}
\subsection{Outline of the paper}
The paper is organized as follows: 

\noindent Section 2  deals with  Lax-Oleinik formula for degenerate convex scalar conservation laws  via Hamilton-Jacobi equations. There we introduce the notion of characteristics curves (forward characteristic curve), shock packets from \cite{1} and using them we  prove the structure Theorem for convex flux with bound. Getting this bounds is very crucial to obtain the final result for non convex flux.

\noindent Section 3  concerns the proof of  the main Theorems with the following main steps:
\begin{itemize}
\item [Step 1.] Since the Riemann problem solution involves the contact discontinuities, the convex hull or the concave hull of $f$ admits degenerate 
parts. That is the hulls need not be strictly convex or concave. Therefore, first we prove the Lax-Oleinik explicit formula and structure Theorem for $C^1$ convex 
flux with bounds using a blow up analysis. This version of Structure Theorem is new to the literature. Using this, we first prove the main Theorem for the convex flux.
\item [Step 2.] Using the front tracking and $L^1$-contraction, we complete the proof of the  main Theorem,  which contents two parts with the  different polarity and therefore the proofs are of different nature. The main ingredient for the first part of  the main Theorem is  structure Theorem with bound. First we prove the convex-convex situation and then we prove the case convex-concave situation by using the front tracking.
 
\item [Step 3.] We need to use some elementary properties of  convex and concave functions and for the sake of completeness we are presenting their proof 
in the Appendix. 
\end{itemize}
Finally we give counter examples to show that the conditions (\ref{121}), (\ref{125}) and (\ref{126}) are optimal.

\noindent In order  to make the paper self contained,  we give the proof of some important Lemmas to get the stability results  using techniques  from convex analysis in Section 4 (appendix). 

\section{Structure Theorem for $C^1$-convex flux}

\subsection{Preliminaries}
\setcounter{equation}{0}
Let $f$ be a locally Lipschitz function on $\R$. Let $K\subset \R$ be a compact set and define 
\begin{eqnarray}\label{21}
Lip (f,K)=\sup_{\substack{x,y\in K\\ x\neq y}}\frac{|f(x)-f(y)|}{|x-y|}.
\end{eqnarray}
Recall the facts from Kruzkov Theorem \cite{5, 6}. Let $u_0,v_0\in L^\f(\R)$ and $u,v$ be the respective solutions of (\ref{11}) with initial data $u_0$
and $v_0$. Let $M=\mbox{Max}\{||u_0||_\f,||v_0||_\f\}$ and $K=[-M,M]$. Then 
\begin{itemize}
\item [(i).] Comparison principle: Assume that $u_0(x)\leq v_0(x)$ a.e. $x\in\R$, then for a.e. $(x,t)\in \R\times (0,\f)$, $u(x,t)\leq v(x,t).$
\item [(ii).] $L^1_{\mbox{loc}}$ contraction: Let $a\leq b,$ then for $t>0$
\begin{eqnarray*}
\int_a^b|u(x,t)-v(x,t)|dx \leq \int_{a-Mt}^{b+Mt}|u_0(x)-v_0(x)|dx.
\end{eqnarray*}
\item [(iii).] Let $u_0\in BV(\R),$ then for $0\leq s< t,$
\begin{eqnarray*}
\int_\R |u(x,t)-u(x,s)|dx \leq Lip(f,K) |s-t| TV(u_0),
\end{eqnarray*}
where $TV(u_0)$ denotes the total variation semi norm of $u_0$.
\end{itemize}
As a consequence of this we are stating the following well known approximation Lemma \cite{3, 5, 9}. For the sake of completeness, we are presenting the proof in the appendix. 
\begin{lemma}\label{lemma21} Let ${f_k}$ be a sequence of Lipschitz functions such that for any compact set $K\subset \R$
\begin{eqnarray}\label{22}
\sup\limits_{n}Lip(f_n,K)<\f.
\end{eqnarray}  
Assume that $f_k\rightarrow f$ in $C^0_{\mbox{loc}}(\R).$ Let $u_k,u$ be the solutions of (\ref{11}) with the  corresponding fluxes $f_k$, $f$ and the initial data $u_0$. 
Then $u_k\rightarrow u$ in $L^1_{\mbox{loc}}(\R\times(0,\f)).$
\end{lemma}
Next we state the following front tracking Lemma of Dafermos \cite{Da} without proof [See chapter 6 in \cite{9},  Lemma (2.6) in \cite{10}].
\begin{lemma}\label{lemma22}
Let $f$ be a piecewise affine continuous function. Let $u_0$ be a piecewise constant function with finite number of jumps. Then for all $t>0, x\rightarrow u(x,t)$ is a 
piecewise constant function with uniformly bounded number of jumps. 
\end{lemma}

\subsection{Structure Theorem for $C^1$-convex flux with bounds}

The structure Theorem is the main ingredient in the proof of the main Theorem. Since the flux $f$ need not be strictly convex,  one must prove the structure 
Theorem with bounds. The main ingredients are Hopf and Lax-Oleinik formula for the solution of Hamilton-Jacobi equation and the corresponding conservation laws. 
For the strictly convex $C^1$ functions, structure Theorem has been proved in \cite{1} when $u_\pm$ are constants  and without bounds. Here we prove the structure Theorem with bounds for $C^1$ convex flux which exhibits finite number of degeneracies.

\noindent{Hmilton-Jacobi equation}: In order to prove the structure Theorem, we need to prove the Lax-Oleinik type explicit formula for the $C^1$ convex flux. As in 
Lax-Oleinik, we establish this via  the Hamilton-Jacobi equations.
\par Let $f$ be a convex function and $u_0\in L^\f(\R)$ with $M=||u_0||_\f.$ For $0\leq s \leq t, x\in \R, p\in \R$, define 
\begin{eqnarray}
f^*(p)&=&\sup\limits_{q}\{pq-f(q)\},\label{23}\\
v_0(x)&=&\int_0^x u_0(\T)d\T,\label{24}\\
v(x,t,f)&=&\inf\limits_{y\in\R}\left\{v_0(y)+tf^*\left(\frac{x-y}{t}\right)\right\},\label{25}\\
w(x,s,t,f)&=&\inf\limits_{y\in\R}\left\{v(x,s,f)+(t-s)f^*\left(\frac{x-y}{t-s}\right)\right\},\label{26}\\
ch(x,t,f)&=&\{\mbox{minimizers in (\ref{25})}\},\label{27}\\
ch(x,s,t,f)&=&\{\mbox{minimizers in (\ref{26})}\},\label{28}\\
y_+(x,t,f)&=&\max\{y: y\in ch(x,t,f)\},\label{29}\\
y_-(x,t,f)&=&\min\{y: y\in ch(x,t,f)\},\label{210}\\
y_+(x,s,t,f)&=&\max\{y: y\in ch(x,s,t,f)\},\label{211}\\
y_-(x,s,t,f)&=&\min\{y: y\in ch(x,s,t,f)\}.\label{212}
\end{eqnarray}
Points in $ch(x,t,f), ch(x,s,t,f)$ are called the characteristic points and the corresponding  sets are called the characteristic sets. $y_\pm$ are called the extreme characteristic points.
\par Then we have the following stability Lemma. Most of the results here are known and for the sake 
of completeness, we are sketching the proofs in the appendix.
\begin{lemma}[Stability Lemma]
Let $f$ be a convex function. Then
\begin{itemize}
\item [1.] 
\begin{eqnarray}\label{213}
\lim\limits_{|p|\rr \f} \frac{f^*(p)}{|p|}=\f.
\end{eqnarray}
\item [2.] Let $\{f_n\}$ be a sequence of convex functions such that $f_n\rightarrow f$ in $C_{\mbox{loc}}^0(\R)$ and \\ $\lim\limits_{|p|\rr \f}\inf\limits_n \frac{f_n(p)}{|p|}=\f$. Then 
$f_n^{*}\rr f^*$ in $C_{\mbox{loc}}^0(\R).$ 
\item [3.] For $0\leq s <t, x\in \R,$ $v$ is a Lipschitz function with Lipschitz constant  bounded 
by $M=||u_0||_\f.$ Let $p_0>1$ such that for $|p|>p_0$, 
\be\label{214}
f^*(p)-M|p|>f^*(0),
\ee
then $ch(x,s,t,f)\neq \phi$ and 
\be
v(x,t,f)&=&\inf\limits_{\left|\frac{x-y}{t-s}\right|\leq p_0}\left\{v(x,s,t)+(t-s)f^*\left(\frac{x-y}{t-s}\right)\right\},\label{215}\\
ch(x,t,f)&=&ch(x,0,t,f), \label{216}\\
y_\pm(x,t,f)&=& y_\pm(x,0,t,f)\label{217}.
\ee
\item [4.] For $0\le s<t, y\in \R$, $\g$ denotes  the line joining $(x,t)$ and $(y,s)$ and is given by 
\be\h{218}
\g(\T,x,s,t,y)=x+\left(\frac{x-y}{t-s}\right)(\T-t).
\ee
\noindent  If $y\in ch(x,s,t,f)$, then $\g$ is called a characteristic line segment.
Let $x_1\neq x_2$, for $i=1,2$, $\xi_i\in ch(x_i, s,t,f), y_i\in ch(\xi_i,s,f),$ then 
\be
&&y_i\in ch(x_i,t,f)\h{219}\\
&&x\mapsto y_\pm(x,s,t,f) \ \mbox{are non decreasing functions.}\h{220}
\ee
Furthermore if $f^*$ is a strictly convex function, then no two different characteristic line segments intersect in the interior. That is  for $\T\in (s,t)$
\be
\g(\T,x_1,s,t,y_1)&\neq& \g(\T,x_2,s, t,y_2),\h{221}\\
\lim\limits_{\xi\uparrow x_0} y_+(\xi,s,t,f)&=&y_-(x_0,s,t,f),\h{222}\\
\lim\limits_{\xi\downarrow x_0} y_-(\xi,s,t,f)&=&y_+(x_0,s,t,f).\h{223}
\ee
\item [5.] For a sequence of sets $E_n\subset\R$, let us denote the set of 
all  cluster points of sequences $\{\rho_n\in E_n\}$ by  $\lim E_n$. Let $\{f_n\}$ be a sequence of convex functions such that $f_n\rr f$ in $C^0_{\mbox{loc}}(\R)$ and $\lim\limits_{|p|\rr \f} \inf\limits_n \frac{f_n(p)}{|p|}=\f$. Let $C>0,$ then there exists $p_1\geq 1$ depending only on $C$ such that 
for all $n, |p|>p_1$
\be
&&\frac{f_n^*(p)}{|p|}\geq C+1, \h{224}\\
 &&v(\cdot, \cdot, f_n)\rr v(\cdot, \cdot, f) \ \mbox{in}\  C^0_{\mbox{loc}}(\R\times(0,\f))\h{225},\\
&&\lim ch(x,s,t,f_n)\subset ch(x,s,t,f)\h{226}.
\ee
\item [6.] As $k\rr \f$, let  $(x_k,t_k)\rr(x,t)$. Let $f, f_k$  be as in (5). Let $y_k\in ch(x_k,s,t_k,f_k)$ such that 
$y_k\rr y.$ Then $y\in ch(x,s,t,f)$.
\item [7.] Let $u_{0,k}\rightharpoonup u_0$ in $L^\f$ weak $^*$ topology and let $f, f_k$  be as in (5). Let $v_{0,k},v_0,v_k,v$, $ch(x,t,f_k),ch(x,t,f)$ as in (\ref{24}), (\ref{25}), (\ref{27}) associated to
$u_{0,k}$ and $u_0$ respectively. Let $y_k\in ch(x,t,f_k)$ such that $y_k\rr y$, then $y\in ch(x,t,f)$. 
\end{itemize} 
\label{lemma23}\end{lemma}
Then we have the following Lax-Oleinik type of explicit formula.
\begin{theorem}\h{theorem24} Let $f$ be a $C^1$ convex function and $u_0\in L^\f(\R).$
Let $v(x,t)=v(x,t,f)$ be the associated value function as in (\ref{25}). Let $u=\frac{\partial v}{\partial x}$, then for $t>0$ and a.e. $x\in\R$,
\be\h{227}
f^\p(u(x,t))=\frac{x-y(x,t)}{t}.
\ee
If $u_0\in C^0(\R)\cap L^\f(\R),$ then for a.e. $x\in\R,$
\be\h{228}
u(x,t)=u_0(y_+(x,t)).
\ee
\begin{proof} From (\ref{42}), choose a sequence $\{f_n\}\subset C^2(\R)$ of uniformly convex function such that $f_n\rr f$ in $C^1_{\mbox{loc}}(\R)$ and $\lim\limits_{|p|\rr \f}\inf_{n}\frac{f_n(p)}{|p|}=\f.$ Let $v_n(x,t)=v(x,t,f_n)$, $v(x,t)=v(x,t,f)$, $y_{+,n}(x,t)=y_+(x,t,f_n)$, $y_+(x,t)=y_+(x,t,f)$ and $u_n=\frac{\partial v_n}{\partial x},u=\frac{\partial v}{\partial x}.$ Let $D(t)$ be the points of discontinuities of $y_+.$ Then $ch(x,t)=\{y_+(x,t)\}$ if $x\notin D(t)$ and thus from (\ref{226}) $\lim\limits_{n\rr \f} y_{+,n}(x,t)=y_+(x,t)$ for $x\notin D(t).$  From Lemma \ref{lemma21}, let $u_n\rr w$ in $L^1_{\mbox{loc}}(\R\times(0,\f))$ and hence 
a.e. $(x,t)$. Therefore from Lax-Oleinik \cite{4} explicit formula, for a.e. $t$, a.e. $x\notin D(t),$
\be\h{229}
\begin{array}{lll}
f^\p(w(x,t))&=&\displaystyle\lim\limits_{n\rr \f} f^\p_n(u_n(x,t))\\
&=&\displaystyle\lim\limits_{n\rr \f} \frac{x-y_{+,n}(x,t)}{t}\\
&=&\displaystyle\frac{x-y_+(x,t)}{t}
\end{array}
\ee
\noindent Claim: $u=w$.\\
From (\ref{225}), $v_n\rr v$ in $C^0_{\mbox{loc}}(\R\times[0,\f)),$ hence for $\phi\in C^\f_{c}(\R\times[0,\f))$
\begin{eqnarray*}
\int\limits_{-\f}^{\f}\int\limits_0^\f v\frac{\partial \phi}{\partial x}dxdt&=&\lim\limits_{n\rr\f}\int\limits_{-\f}^{\f}\int\limits_0^\f  v_n \frac{\pa \phi}{\pa x}dxdt\\
&=&-\lim\limits_{n\rr\f}\int\limits_{-\f}^{\f}\int\limits_0^\f \frac{\pa v_n}{\pa x}\phi dxdt\\
&=&-\lim\limits_{n\rr\f}\int\limits_{-\f}^{\f}\int\limits_0^\f u_n\phi dxdt\\
&=&-\int\limits_{-\f}^{\f}\int\limits_0^\f w \phi dxdt.
\end{eqnarray*}
Hence $w=\frac{\pa v}{\pa x}=u$. This proves the claim.
\par First assume that $u_0\in BV(\R).$ Then for any $s,t\geq 0$, we have 
\be\h{230}
\int\limits_\R|u(x,s)-u(x,t)|dx \leq Lip (f,K)|s-t|TV(u_0).
\ee
From (\ref{229}), (\ref{230}), choose a sequence $s_k\rr t$, a null set $N \supset D(t) $ such that for all $x\notin N$, $u(x,s_k)\rr u(x,t)$ and $f^\p(u(x,s_k))=\frac{x-y_+(x,s_k)}{t}.$
 Since $x\notin D(t)$, 
$ch(x,t)=\{y_+(x,t)\}$, therefore from (6) of Lemma \ref{lemma23} $y_+(x,s_k)\rr y_+(x,t)$
and 
\be\h{231}
\begin{array}{llll}
f^\p(u(x,t))&=&\displaystyle \lim\limits_{k\rr \f} f^{\p}(u(x,s_k))\\
&=&\displaystyle\lim\limits_{k\rr \f}\frac{x-y_+(x,s_k)}{t}\\
&=&\displaystyle\frac{x-y_+(x,t)}{t}.
\end{array}
\ee
Let $u_0\in L^\f(\R)$ and $u_{0,k}\in BV(\R)$ such that $u_{0,k}\rr u_0$ in $L^1_{\mbox{loc}}(\R)$ and almost everywhere. Let $u_k$ be the solution of (\ref{11}) with initial data $u_{0,k}$,  then from $L^1_{\mbox{loc}}$ contraction $u_k(x,t)\rr u(x,t)$ for a.e. $x\in\R, t>0$. Let $y_{+,k},y_+$
be as in (\ref{29}) for $u_{0,k}$ and $u_0$ respectively. Then from (\ref{231}) and (7) of Lemma
\ref{lemma23}, there exists a null set $N\supset D(t)$ such that for $x\notin N$, $y_{+,k}(x,t)\rr y_+(x,t)$ and 
\begin{eqnarray*}
f^\p(u(x,t))&=&\lim\limits_{k\rr \f}f^\p(u_k(x,t))\\
&=&\lim\limits_{k\rr \f}\frac{x-y_{+,k}(x,t)}{t}\\
&=& \frac{x-y_+(x,t)}{t}.
\end{eqnarray*} 
This proves (\ref{227}).
\par Let $u_0\in C^0(\R),$ then from the  monotonicity of $y_{+,n}$ and uniform convexity, it follows that 
$$u_n(x,t)=u_0(y_{+,n}(x,t)).$$
Hence as in the previous case for a.e. $s$, a.e. $x,$ $u_n(x,s)\rr u(x,s)$ and $y_{+,n}(x,s)\rr y_+(x,s) $. Consequently  
\begin{eqnarray*}
u(x,s)=\lim\limits_{n\rr \f} u_n(x,s)=\lim\limits_{n\rr \f} u_0(y_{+,n}(x,s))=u_0(y_+(x,s)).
\end{eqnarray*}
Next assume that $u_0\in BV(\R)$, then $u(x,s_k)\rr u(x,t)$ as $s_k\rr t$, a.e. $x$. Therefore letting $s_k\rr t$, $x\notin N\supset D(t)$ such that $y_+(x,s_k)\rr y_+(x,t)$, we have 
\begin{eqnarray*}
u(x,t)&=& \lim\limits_{k\rr \f} u(x,s_k)=\lim\limits_{k\rr \f} u_0(y_{+}(x,s_k))\\
&=& u_0(y_+(x,s)).
\end{eqnarray*}
Let $u_0\in C^0$, then approximate $u_0$ by $u_{0,k}\in C^0(\R)\cap BV(\R)$ in $L^1_{\mbox{loc}}$ norm. Then from $L^1_{\mbox{loc}}$ contraction, for a.e. $x,$ $u_k(x,t)\rr u(x,t),$
$y_{+,k}(x,t)\rr y_+(x,t)$. Thus
$$u(x,t)=\lim\limits_{k\rr \f} u_k(x,t)=\lim\limits_{k\rr \f}u_0(y_{+,k}(x,t))=u_0(y_+(x,t)).$$
This proves (\ref{228}) and hence the Theorem.
\end{proof}
\end{theorem}
\par Next we pass onto quantitative version of the structure Theorem \cite{1}. For this, let us recall the definition of $R_\pm$ curves called the characteristic curves. 
\begin{definition}
Let $f\in C^1(\R)$ be a convex function and $y_\pm(x,t)=y_\pm(x,t,f)$, $y_\pm(x,s,t)=y_\pm(x,s,t,f)$, $ch(x,s,t)=ch(x,s,t,sf)$, $ch(x,t)=ch(x,t,f)$ as in (\ref{29}) to (\ref{212}). For $\al\in \R$, $0\leq s<t$, define 
\be
R_+(t,s,\al,u_0)&=&\sup\{x:y_+(x,s,t)\leq \al\},\h{2311}\\
R_-(t,s,\al,u_0)&=&\inf\{x:y_-(x,s,t)\geq \al\},\h{232}\\
R_{\pm}(t,\al,u_0)&=&R_\pm(t,0,\al,u_0).\h{233}
\ee
\end{definition}
From the comparison principle and (\ref{215}), $R_\pm$ satisfies the following 
\begin{lemma}\h{lemma25} Let $f$ be $C^1$ convex function having finite number of degeneracies. Also let  $u_0\in L^\f(\R), M=||u_0||_\f$ and $p_0>1$ as in (\ref{214}). Then 
\begin{itemize}
\item [1.] $t\mapsto R_\pm(t,\al,u_0)$ are uniformly Lipschitz continuous functions with 
\be
&& R_\pm(0,\al,u_0)=\al,\h{234}\\
&& \left\|\frac{dR_\pm}{dt}\right\|\leq 1+p_0,\h{235}\\
&& R_-(t,\al,u_0)\leq R_+(t,\al, u_0),\h{236}\\
&&y_-(R_\pm(t,\al,u_0),t)\leq \al \leq y_+(R_\pm(t,\al,u_0),t).\h{237}
\ee
\item [2.] Let $u_0\leq w_0$ and $y_{1,\pm}$, $y_{2,\pm}$ be the respective optimal characteristic points of $u_0$ and $w_0$. Then 
\be
&&y_{2,\pm}(x,t)\leq y_{1,\pm}(x,t),\h{238}\\
&& R_\pm(t,\al,u_0)\leq R_\pm(t,\al,w_0)\h{239}.
\ee
\item [3.] Let $u_{0,n}\rr u_0$ in $L^1_{\mbox{loc}}(\R),$ then $\lim\limits_{n\rr \f}R_\pm(t,\al,u_{0,n})$ exist in $C^0_{\mbox{loc}}(\R)$ and satisfies
\begin{itemize}
\item [i.] If for all $n$, $R_-(t,\al,u_{0,n})\leq R_-(t,\al,u_0)$, then
$$\lim\limits_{n\rr \f}R_-(t,\al,u_{0,n})=R_-(t,\al,u_0).$$ 
\item [ii.] If for all $n$, $R_+(t,\al,u_{0,n})\geq R_+(t,\al,u_0)$, then
$$\lim\limits_{n\rr \f}R_+(t,\al,u_{0,n})=R_+(t,\al,u_0).$$ 
\end{itemize}
\item [4.] Let $0<s<t$, then 
\be
R_+(t,s,\al,u_0)&=&R_-(t,s,\al,u_0),\h{240}\\
R_\pm(t,\al,u_0)&=&R_\pm(t,s,R_\pm(s,\al,u_0),u_0)\h{241}.
\ee
If for some $\al,\B$ and  $T>0$, $R_+(T,\al,u_0)=R_+(T,\B,u_0)$ or $R_+(T,\al,u_0)=R_-(T,\B,u_0)$, then for $t>T, R_+(t,\al,u_0)=R_+(t,\B,u_0)$ or $R_+(t,\al,u_0)=R_-(t,\B,u_0)$ respectively.
\item [5.] Suppose for some $T>0$, 
\be\h{242}
R_-(T,\al,u_0)<R_+(T,\al,u_0),
\ee
then for $R_-(T,\al,u_0)<x<R_+(T,\al,u_0), \ \al\in ch(x,T)$ and 
\be\h{243}
f^\p(u(x,T))=\frac{x-\al}{T}.
\ee
\item [6.] Let $\{u_{0,n}\}$ be a bounded sequence in $L^\f(\R),$ $\{f_n\}$ be a sequence of $C^1$ convex  functions  and $p_0>0$ such that for $|p|>p_0$, for all $n$
\be
&&f_n^*(p)-M|p|>\sup\limits_mf^*_m(0),\h{244}\\
&&u_{0,n}\rightharpoonup u_0 \ \mbox{in}\ L^\f\ \mbox{weak}^*,\h{245}\\
&&f_n\rr f\ \mbox{in}\ C^0_{\mbox{loc}}(\R)\h{246}.
\ee
Then for $t>0$, a.e. $x\in \R$, 
\be
\lim\limits_{n\rr \f} f_n^\p(u_n(x,t))=f^\p(u(x,t))\h{247},
\ee
where $u_n$ is the solution of (\ref{11}), (\ref{12}) with flux $f_n$ and initial data $u_{0,n}.$
\item [7.] Let $T>0$, $\{y_0\}=ch(x_0,T)$, $f^\p(p)=\frac{x_0-y_0}{T}.$ Let $\g$ be the characteristic line segment defined by 
\be
\g(\T)=x_0+\frac{x_0-y_0}{T}(\T-T)\h{248}.
\ee
Let $\tilde{u}_0\in L^\f(\R)$ defined by 
\be\h{249}
\ti{u}_0(x)=\left\{\begin{array}{lll}u_0(x) &\mbox{if}& x<y_0,\\
p &\mbox{if}& x>y_0
\end{array}\right.
\ee
and $\ti{u}$ be the corresponding solution of (\ref{11}), (\ref{12}). Then for $0<t<T$, $\ti{u}$ is given by 
\be\h{250}
\ti{u}(x,t)=\left\{\begin{array}{lll}u(x,t) &\mbox{if}& x<\g(t),\\
p &\mbox{if}& x>\g(t).
\end{array}\right.
\ee
\item [8.] Let $\al\in\R$, $a_1\leq a_2$, $b_2\leq b_1$ and $u_0$ be such that 
\be\h{251}
u_0(x)\in \begin{cases} [a_1,a_2] \mbox{ if } x>\al,\\
 [b_2,b_1] \mbox{ if } x<\al.
\end{cases}
\ee
Then for a.e. $x$,
\be\h{252}
u(x,t)\in\begin{cases}[a_1,a_2] \mbox{ if } x>R_+(t,\al,u_0),\\
[b_2,b_1] \mbox{ if } x<R_-(t,\al,u_0).
\end{cases}
\ee
Furthermore if $b_2>a_2$, then for all $t>0, s>0$
$$R_-(t,\al,u_0)=R_+(t,\al,u_0).$$
\be\h{253}
\begin{array}{lll}
\displaystyle\min_{\substack{p\in[b_2,b_1]\\ q\in[a_1,a_2]}}\left(\frac{f(p)-f(q)}{p-q}\right) \leq\displaystyle
 \frac{dR_1}{dt}(t,\al,u_0)\leq\displaystyle \max_{\substack{p\in[b_2,b_1]\\ q\in [a_1,a_2]}} \frac{f(p)-f(q)}{p-q}.
\end{array}\ee
\end{itemize}
\begin{proof}
(1) to (5) follows as in Lemma 4.2 in \cite{1}. Let  $y_{+,n}=y_+(x,t,f_n)$, then from  (7) of Lemma \ref{lemma23}, $\lim\limits_{n\rr \f}y_{+,n}(x,t)=y_+(x,t)$ for all $x\notin D(t)$. Hence from (\ref{227}), for a.e. $x\in\R, x\notin D(t),$
\begin{eqnarray*}
\lim\limits_{n\rr \f} f^\p_n(u_n(x,t))&=& \lim\limits_{n\rr \f} \frac{x-y_{+,n}(x,t)}{t}\\
&=& \frac{x-y_+(x,t)}{t}\\
&=& f^\p(u(x,t)). 
\end{eqnarray*}
This proves (6).
\par From (4) of Lemma \ref{lemma41}, $f^*$ is strictly convex and thus from (\ref{223}) and (\ref{224}) for $s=0,$ we have 
\be
\lim\limits_{\xi\uparrow \g(t)} y_+(\xi,t)&=& y_-(x_0,t)=y_0,\h{254}\\
\lim\limits_{\xi\downarrow \g(t)} y_-(\xi,t)&=& y_+(x_0,t)=y_0.\h{255}
\ee
First assume that $f^\p$ is a strictly increasing function. Let $w$ denote the right hand side  of (\ref{250}).
From (\ref{227}) and (\ref{254}), we have for $0<t<T$,
\begin{eqnarray*}
\lim\limits_{\xi\uparrow \g(t)}f^\p(w(\xi,t))&=&\lim\limits_{\xi\uparrow \g(t)} f^\p(u(\xi,t))\\
&=&\lim\limits_{\xi\uparrow\g(t)} \frac{\xi-y_+(\xi,t)}{t}\\
&=&\frac{\g(t)-y_0}{t}=\frac{x_0-y_0}{t}=f^\p(p).
\end{eqnarray*}
$$\lim\limits_{\xi\downarrow \g(t)} f^\p(w(\xi,t))=f^\p(p).$$
Since $f^\p$ is strictly increasing, $w(\g(t)-,t)$ and $w(\g(t)+,t)$ exist and $w(\g(t)-,t)=w(\g(t)+,t)=p$. Hence $w$ is continuous across $\g(\cdot)$ and is a solution for $x\neq \g(t).$ Therefore $w$ is the solution of (\ref{11}) in $\R\times(0,T)$. Whence $w=\ti{u}$ in $\R\times (0,T).$
\par For general $f$, let $\{f_n\}\subset C^1(\R)$ be a sequence of convex function converging to $f$ in $C^1_{\mbox{loc}}(\R)$ with $\lim\limits_{|p|\rr \f}\inf\limits_n \frac{f_n(p)}{|p|}=\f$. Since 
\be\h{256}
\lim ch(x_0,T,f_n)\subset ch(x_0,T,f)=\{y_0\}
\ee
and therefore  from (\ref{254}), (\ref{256}) choose $\{x_n\}, \{y_n\}$ such that 
$$ch(x_n, T,f_n)=\{y_n\}, \ \lim\limits_{n\rr \f}(x_n,y_n)=(x_0,y_0).$$
Define for $0<t<T$
$$\g_n(t)=x_n+\left(\frac{x_n-y_n}{T}\right)(t-T), \ f^\p_n(p_n)=\frac{x_n-y_n}{T}.$$
\begin{eqnarray*}
\ti{u}_{0,n}(x)=\left\{\begin{array}{lll}
u_0(x) &\mbox{if}& x<y_n,\\
p_n &\mbox{if}& x>y_n,
\end{array}\right.
\end{eqnarray*}
 \begin{eqnarray*}
\ti{u}_{n}(x,t)=\left\{\begin{array}{lll}
u_n(x,t) &\mbox{if}& x<\g_n(t),\\
p_n &\mbox{if}& x>\g_n(t),
\end{array}\right.
\end{eqnarray*}
 where $u_n$  is the solution of (\ref{11}), (\ref{12}) with flux $f_n$ and  data $u_0$.  Hence by previous analysis, for $0<t<T$, $\ti{u}_n$ is the solution of (\ref{11}), (\ref{12}) with flux $f_n$ and initial data $\ti{u}_{0,n}$. From Lemma \ref{lemma21}, $u_n\rr u$ a.e. $(x,t)\in \R\times (0,\f)$. Let $w$ denotes the right hand side of (\ref{250}). Then $\ti{u}_n\rr w$ a.e. $(x,t)\in \R\times(0,T)$. Since  $\ti{u}_{0,n}\rr \ti{u}_0$, by dominated convergence Theorem, letting $n\rr \f$ in the entropy inequality we obtain $w$ is the solution of (\ref{11}) and (\ref{12}) with flux $f$ and initial data $\ti{u}_0$. Hence $w=\ti{u}$ in $\R\times (0,T)$. This proves (7).
 \par In order to prove (8), first assume that $u_0$ is continuous for $x<\al.$ Let $x<R_-(t,\al,u_0)$, then from (\ref{237}), $y_+(x,t)<\al$ and from (\ref{228}), for a.e. $x<R_-(t,\al,u_0)$
 \be\h{257} u(x,t)=u_0(y_+(x,t))\in [b_2,b_1].\ee
 Similarly, if $u_0$ is continuous for $x>\al$, then for  $x>R_+(t,\al,u_0).$
 \be\h{258} u(x,t)=u_0(y_+(x,t))\in [a_1,a_2].\ee
 Let $u_{0,n}\in L^\f$ such that $u_{0,n}$ is continuous in $(-\f,\al)\cup (\al,\f)$ and $u_{0,n}\rr u_0$ in $L^1_{\mbox{loc}}(\R)$. Let $u_n$ be the solution of (\ref{11}), (\ref{12}) with the initial data $u_{0,n}$. Then for $x<$$R_-(t,\al,u_{0,n}),$\\$ u_n(x,t)\in [b_1,b_2]$, for $x>R_+(t,\al, u_{0,n}), u_n(x,t)\in [a_1,a_2]$ and $u_n(x,t)\rr u(x,t)$ a.e. $(x,t)\in \R\times(0,\f)$. Now from (3) of Lemma \ref{lemma25}, we have 
 $$R_-(t,\al,u_0)\leq \lim R_-(t,\al,u_{0,n})\leq \lim R_+(t,\al,u_{0,n})\leq R_+(t,\al,u_0).$$
 Hence for $x<R_-(t,\al,u_0), u(x,t)\in [b_2,b_1]$ and $u(x,t)\in [a_1,a_2]$ if $x>R_+(t,\al,u_0)$. This proves (\ref{252}). 
 \par Proof of (\ref{253}) follows from the approximation procedure. First assume that $u_0$ is  a piecewise constant function with finite number of jumps. Let $\{f_n\}$ be a sequence of piecewise affine convex  functions such that $f_n\rr f$ in $C^0_{\mbox{loc}}(\R)$ and $\sup\limits_nLip (f_n,K)<\f$ for any compact set $K\subset \R$. Let $u_n$ be the solution of (\ref{11}), (\ref{12}) with flux $f_n$ and initial data $u_0$. Then from Lemma \ref{lemma21}, $u_n\rr u$ a.e. in $\R\times (0,\f)$.
 \par Let $T_1^n$ be the first time of interaction of waves for $u_n$. Then for $0<t<T_1^n$, let 
 \be\h{259}
 m_1^n=\frac{f_n(u_0(\al+))-f_n(u_0(\al-))}{u_0(\al+)-u_0(\al-)},
 \ee
$l^n_1(t)=\al+m_1^nt$, $y_{\pm,n}(x,t)=y_{\pm}(x,t,f_n),$ then 
 \be\h{260}
 u_n(x,t)\in\begin{cases}
\ [a_1,a_2] \mbox{ if } x>l_1^n(t),\\
\ [b_2,b_1] \mbox{if } x<l_1^n(t), 
\end{cases} \ee
 and 
 \be\h{261}
 y_{-,n}(l_1^n(t),t)\leq \al \leq y_{+,n}(l_1^n(t),t).
 \ee
 Let $T_2^n>T_1^n$ be the second time of interaction waves. For $t\in(T_1^n, T_2^n)$, let 
 \begin{eqnarray*}
 m_2^n&=&\frac{f_n(u_n(l_1^n(T_1^n)+,T_1^n))-f_n(u_n(l_1^n(T_1^n)-,T_1^n))}{u_n(l_1^n(T_1^n)+,T_1^n))-u_n(l_1^n(T_1^n)-,T_1^n))}.\\
  l_2^n(t)&=&l_1^n(T_1^n)+m_2^n(t-T_1^n).
 \end{eqnarray*}
 Then from (\ref{260}), $u_n(l_1^n(T_1^n)+,T_1^n)\in[b_2,b_1]$ and $u_n(l_1^n(T_1^n)-,T_1^n)\in[a_1,a_2]$, hence for $T_1^n<t<T_2^n$,
  \begin{eqnarray*}
 u_n(x,t)\in \begin{cases}
 [a_1,a_2], \mbox{ if } x>l_1^n(t),\\
[b_2,b_1], \mbox{ if } x<l_1^n(t), 
\end{cases} 
\end{eqnarray*}
\begin{eqnarray*}
 y_{-,n}(l_2^n(t),T_1^n,t)\leq l^n_1(T_1^n) \leq y_{+,n}(l_2^n(t),T_1^n,t).
\end{eqnarray*}
Therefore  from (\ref{219}) we have 
\begin{eqnarray*}
 y_{n,-}(l_2^n(t),t)\leq \al \leq y_{+,n}(l_2^n(t),t).
\end{eqnarray*}
Define 
\be\h{262}
r^n(t)=\left\{\begin{array}{lll}
l_1^n(t) &\mbox{if}& t\leq T_1^n,\\
l_2^n(t) &\mbox{if}& T_1^n\leq t \leq T_2^n.
\end{array}\right.
\ee
Then $r^n$ is continuous and for $t\in [0,T_2^n), t\neq T_1^n$, we have 
\be\h{263}
\begin{array}{lll}
\displaystyle\min_{\substack{p\in[b_2,b_1]\\ q\in[a_1,a_2]}}\left(\frac{f_n(p)-f_n(q)}{p-q}\right) \leq\displaystyle
 \frac{dr^n}{dt}(t)\leq\displaystyle \max_{\substack{p\in[b_2,b_1]\\ q\in [a_1,a_2]}} \frac{f_n(p)-f_n(q)}{p-q},
\end{array}\ee
 \begin{eqnarray}\h{264}
 u_n(x,t)\in\begin{cases}[a_1,a_2] \mbox{ if } x>r^n(t),\\
[b_2,b_1] \mbox{ if } x<r^n(t), 
\end{cases}\end{eqnarray}
\be\h{265}
 y_{-,n}(r^n(t),t)\leq \al=r^n(0) \leq y_{+,n}(r^n(t),t).
\ee
Hence by induction and from front tracking Lemma \ref{lemma22}, $r^n(\cdot)$ is well defined for all $t>0$ and 
(\ref{263}) to (\ref{265}) holds. Let for a subsequence, $r^n(\cdot)\rr r(\cdot)$ in $C^0_{\mbox{loc}}([0,\f))$.
From (\ref{265}), (\ref{226}), and (\ref{264}) we have for all $t>0$,
 \begin{eqnarray}\h{266}
 u(x,t)\in\begin{cases}
 [a_1,a_2] \mbox{ if } x>r(t),\\
 [b_2,b_1] \mbox{ if } x<r(t), 
\end{cases}\end{eqnarray}
\be\h{267}
 y_-(r(t),t)\leq \al=r(0) \leq y_+(r(t),t).
\ee
Since $f^*$ is strictly increasing, from (\ref{221}), from the definition of $R_\pm$, we have 
\be\h{268}
R_-(t,\al,u_0)\leq r(t)\leq R_+(t,\al,u_0)
\ee
and for a.e. $t$, 
\be\h{270}
\begin{array}{lll}
\displaystyle\min_{\substack{p\in[b_2,b_1]\\ q\in[a_1,a_2]}}\left(\frac{f(p)-f(q)}{p-q}\right) 
\leq\displaystyle
 \frac{dr}{dt}(t)\leq\displaystyle \max_{\substack{p\in[b_2,b_1] \\q\in [a_1,a_2]}} \frac{f(p)-f(q)}{p-q}.
\end{array}\ee
This proves (\ref{253}) when $u_0$ is piecewise constant. Let $u_0\in L^\f(\R)$ and $\{u_{0,k}\}$ be a sequence of piecewise constant function converging to $u_0$ in $L^1_{\mbox{loc}}$ and a.e. $x$. Denote 
$u^k$ be the corresponding solution of (\ref{11}), (\ref{12}) with initial data $u_{0,k}$ and  $y_\pm^k$ be the corresponding extreme characteristics. Let  $r^k(t)=R_-(t,\al,u_{0,k}),$ then from (\ref{267}) we have 
\be\h{n271} y_-^k(r^k(t),t)\leq \al \leq y_+^k(r^k(t),t)\ee
and satisfies (\ref{266}), (\ref{270}). Hence for a subsequence let $r^k\rr r$ in $C^0_{\mbox{loc}}(\R)$.
Then $u$ satisfies (\ref{266}) and from (7) of Lemma \ref{lemma23}, (\ref{n271}), we have 
\begin{eqnarray*}
y_-(r(t),t)\leq \al \leq y_+(r(t),t),\\
R_-(t,\al,u_0)\leq r(t) \leq R_+(t,\al,u_0).
\end{eqnarray*}
Next we claim that for all $t>0$, $R_-(t,\al,u_0)=R_+(t,\al,u_0)$. If not, then there exists a $T>0$ such that for $0<t<T$, $R_-(t,\al,u_0)<R_+(t,\al,u_0)$. Then from (\ref{243}), $x\mapsto u(x,t)$ is a strictly increasing continuous  function for $x\in (R_-(t,\al,u_0), R_+(t,\al,u_0))$. Suppose for some $0<t_0<T$, $R_-(t_0,\al,u_0)<r(t_0)<R_+(t_0,\al,u_0)$, then from the hypothesis $b_2>a_2$ and (\ref{266}) $u$ has a jump discontinuity at $x=r(t_0)$ which is a contradiction. Hence by continuity, for all $t<T$ either  $r(t)=R_-(t,\al,u_0)$ or $r(t)=R_+(t,\al,u_0)$. 
Assume  for $0<t\leq T_0$, $r(t)=R_-(t,\al,u_0)<R_+(t,\al,u_0)$, then from (\ref{243}) and (\ref{266})
\begin{eqnarray*}
&&f^\p(u(x,t))\in [a_1,a_2], \ \mbox{for} \ x>r(t)\\
&&f^\p(u(x,t))=\frac{x-\al}{t}.
\end{eqnarray*}
Hence define $f^\p(p)\in [a_1,a_2]$ by 
\begin{eqnarray*}
f^\p(p)&=&\lim\limits_{x\uparrow r(T_0)} f^\p(u(x,t))\\
&=& \frac{r(T_0)-\al}{T_0}.
\end{eqnarray*}
Then $\al\in ch(r(T_0),T_0)$ and $s(t)=\al+tf^\p(p)$ is a characteristic line segment for $0\leq t \leq T_0$. Let 
\begin{eqnarray*}
u_{1,0}=\left\{\begin{array}{llll}
p &\mbox{if}& x<\al,\\
u_0(x) &\mbox{if}& x>\al.
\end{array}\right.
\end{eqnarray*}
For $0<t<T_0$, 
\begin{eqnarray*}
u_1(x,t)=\left\{\begin{array}{llll}
p &\mbox{if}& x<s(t),\\
u(x,t) &\mbox{if}& x>s(t),
\end{array}\right.
\end{eqnarray*}
then from (\ref{250}), for $0<t<T_0$, $u_1$ is the solution of (\ref{11}), (\ref{12}) with initial data $u_{1,0}$. Since $b_2>a_2$, thus $u_{1,0}\leq u_0$. Consequently  for $0<t<T_0$, $u_1(x,t)\leq u(x,t)$
and $s(t)=R_-(t,\al,u_{1,0})\leq R_-(t,\al,u_0)=r(t).$ As $r(t)\leq s(t)$, therefore $s(t)=r(t)=R_-(t,\al,u_0)$ and then $r(\cdot)$ is a characteristic line segment for $u$ for $0<t\leq T_0$. Hence $y_+(x,t)\rr \al$ as $x\uparrow r(t)$ and 
\begin{eqnarray*}
\lim\limits_{x\uparrow r(t)} f^\p(u(x,t)) &=& \lim\limits_{x\uparrow r(t)} \frac{x-y_+(x,t)}{t}\\
&=& \frac{r(t)-\al}{t}\\
&=& f^\p(p)\in [f^\p(a_1),f^\p(a_2)].
\end{eqnarray*}
On the other hand, from (\ref{266}), $\lim\limits_{x\uparrow r(t)} f^\p(u(x,t))\in [f^\p(b_1),f^\p(b_2)]$. This implies that $f^\p(p)\in [f^\p(a_1), f^\p(a_2)]\cap [f^\p(b_2), f^\p(b_1)].$ Since for $r(t)<x<R_+(t,\al,u_0)$, $f^\p(u(x,t))=\frac{x-\al}{t}$ is an increasing function and therefore from (\ref{266}), $p\in [a_1,a_2]$ and $f^\p$ is  strictly increasing  function in  $(p,p+\e)\subset [a_1,a_2] $ for some $\e>0$. Since $b_2>a_2$, hence $f^\p(b_2)>f^\p(a_2)$ and therefore $f^\p(p)\notin [f^\p(a_1), f^\p(a_2)]\cap [f^\p(b_1), f^\p(b_2)],$ which is a contradiction. Similarly if $r(t)=R_+(t,\al,u_0)$ for some $0\leq t \leq T_0$. This proves (8) and hence the Lemma.
\end{proof}
\end{lemma}
\par In order to prove the Structure Theorem, let us recall the characteristic line and ASSP from \cite{1}.
\noindent Let $f$ be $C^1$ convex, $u_0\in L^\f$ and $u$ be the solution of (\ref{11}) and (\ref{12}). Let $a<b$, $p\in \R$, 
\be\h{271}
\g(t,a,p)&=&a+tf^\p(p),\\
D(a,b,p)&=& \{(x,t)\in \R\times (0,\f): \g(t,a,p)<x<\g(t,b,p)\}\h{272}.
\ee        
Let $y_\pm(x,t)=y_\pm(x,t,f)$ and define 
\begin{itemize}
\item [1.] Characteristic line: $\g(\cdot, a, p)$ is called a characteristic line if for all $t>0$,
\be\h{273}
a=\g(0,a,p)\in ch(\g(t,a,p),t).
\ee  
\item [2.] Asymptotically single shock packet (ASSP): $D(a,b,p)$ is called an  ASSP if 
\begin{itemize}
\item [i.] $\g(\cdot,a,p), \g(\cdot, b, p)$ are characteristic lines.
\item [ii.] $D(a,b,p)$ does not contain a characteristic line. 
\item [iii.] For $\al\in (a,b)$, $R_\pm(t,\al,u_0)\in D(a,b,p)$.
\end{itemize}
\end{itemize}
\begin{lemma}\label{lemma26} Let $f\in C^1(\R)$ with finite number of degeneracies.
\begin{itemize}
\item [1.] Let $\g(\cdot, a, p)$ be a characteristic line. Then 
\be\h{274}y_\pm(\g(t,a,p),t)=a.\ee
\item [2.] Let $A,\al_1\leq \al_2$, $\B_2\leq \B_1$, $u_0\in L^\f(\R)$ are given. Assume that $\bar{u}_{1,0}, \bar{u}_{2,0}\in L^\f(\R)$ such that $\bar{u}_{1,0}(x)\in [\B_2,\B_1], \bar{u}_{2,0}(x)\in [\al_1,\al_2]$ and define
\be\h{275}
u_0^1(x)=\left\{\begin{array}{lll}
\bar{u}_{1,0}(x) &\mbox{if}& x<A,\\
u_0(x) &\mbox{if}& x>A,
\end{array}\right.
\ee
\be\h{276}
u_0^2(x)=\left\{\begin{array}{lll}
u_0(x) &\mbox{if}& x<A,\\
\bar{u}_{2,0}(x) &\mbox{if}& x>A.
\end{array}\right.
\ee
Let $u_1$, $u_2$ be the corresponding solutions of (\ref{11}), (\ref{12}) with initial data $u_0^1$ and $u_0^2$. Then 
\begin{itemize}
\item [$(a_1)$.] Assume that 
\be\h{277}
\sup\limits_{t>0}y_+(R_-(t,A,u_0^1),t)<\f,
\ee
then there exist an $A_1\geq A$, $p_-\in\R$ such that 
\be
&&\lim\limits_{t\rr\f}y_+(R_-(t,A,u_0^1),t)=A_1.\h{278}\\
&&f^\p(p_-)=\lim\limits_{t\rr \f}\frac{R_-(t,A,u_0^1)-y_+(R_-(t,A,u_0^1),t)}{t}.\h{279}\\
&&f^\p(p_-)\in [f^\p(\B_2),f^\p(\B_1)].\h{280}\\
&&\g(\cdot,A_1,p_-)\ \mbox{is a characteristic line.}\h{281}
\ee
\item [$(a_2)$.] Assume that 
\be\h{282}
\inf\limits_{t>0}y_-(R_+(t,A,u_0^2),t)>-\f,
\ee
then there exist an $A_2\leq A$, $p_+\in\R$ such that 
\be
&&\lim\limits_{t\rr\f}y_-(R_+(t,A,u_0^2),t)=A_2.\h{283}\\
&&f^\p(p_+)=\lim\limits_{t\rr \f}\frac{R_+(t,A,u_0^2)-y_-(R_+(t,A,u_0^2),t)}{t}.\h{284}\\
&&f^\p(p_+)\in [f^\p(\al_1),f^\p(\al_2)].\h{285}\\
&&\g(\cdot,A_2,p_+)\ \mbox{is a characteristic line.}\h{286}
\ee
\end{itemize}
\end{itemize}
\begin{proof}
Note that from (4) of Lemma \ref{lemma41}, $f^*$ is a strictly convex function.
\begin{itemize}
\item [1.] Let $\g(s)=\g(s,a,p)$ and $v$ be the value function as in (\ref{25}). Suppose for some $s>0$, $y_+(\g(s),s)>a$, then for $t>s$, we have 
\begin{eqnarray*}
v(\g(t),t)&=& v(\g(s),s)+(t-s)f^*\left(\frac{\g(t)-\g(s)}{t-s}\right)\\
&=& v_0(y_+(\g(s),s))+sf^*\left(\frac{\g(s)-y_+(\g(s),s)}{s}\right)\\&+&(t-s)f^*\left(\frac{\g(t)-\g(s)}{t-s}\right)\\
&>& v_0(y_+(\g(s),s))+tf^*\left(\frac{\g(t)-y_+(\g(s),s)}{t}\right)\\
&\geq& v(\g(t),t),
\end{eqnarray*}
 which is a contradiction. Hence $y_+(\g(s),s)=a$. Similarly $y_-(\g(s),s)=a$ and this proves (1).
 \item [2.] Assume (\ref{274}) holds. Let $R(t)=R_-(t,A,u_0^1), y(t)=y_+(R(t),t)$ and 
 \be\h{287}
 \g(\T,t)=y(t)+\left(\frac{R(t)-y(t)}{t}\right)\T.
 \ee
 Let $t_1>t_2$. Then from (\ref{221}), for all $\T\in(0,t_2)$, either $\g(\T,t_1)=\g(\T,t_2)$ or  $\g(\T,t_1)\neq \g(\T,t_2)$. From (\ref{237}) $y(t)\geq A$, hence $\g(\T,t_1)\geq \g(\T,t_2)$ for $\T\in [0,t_2]$. Hence $y(t_1)\geq y(t_2).$ Therefore from (\ref{277}), $A_1$ exists satisfying (\ref{278}). Let $\T_0>0$ be fixed, then from the above analysis, for  $t>\T_0$, $t\rr \g(\T_0,t)$ is a non decreasing function. Also from (\ref{215}), $\left\{\frac{R(t)-y(t)}{t}\right\}$ is bounded. Hence $p_-$ exist and satisfies (\ref{279}), $\lim\limits_{t\rr \f} \g(\T_0,t)=A_1+f^\p(p_-)\T_0$. Thus from (7) of Lemma \ref{lemma23}, $A_1\in ch(A_1+f^\p(p_-)\T_0, \T_0)$ for all $\T_0>0,$ therefore $A_1\in ch(\g(\T, A_1, p_-), \T)$ for $\T>0$. Hence $\g(\cdot, A_1,p_-)$ is a characteristic line. This proves (\ref{281}).
 \par Next we prove (\ref{280}). From the strict convexity of $f^*$ as in the proof of (\ref{274}), we have for all  $0\le \T \leq t$, 
 $$\{y(t)\}=ch(\g(\T,t),\T).$$
 Therefore $R(\T)\leq \g(\T,t)$ for all $0<\T\leq t$. Letting $t\rr \f$ to obtain $R(\T)\leq \g(\T,A_1, p_-)$ for all $\T>0$. Define 
 \begin{eqnarray*}
 \ti{u}_0(x)=\left\{\begin{array}{lll}
 u_0(x) &\mbox{if}& x<A_1,\\
 p_-&\mbox{if}& x>A_1,\end{array}\right.
 \end{eqnarray*}
 \begin{eqnarray*}
 \ti{u}(x,t)=\left\{\begin{array}{lll}
 u(x,t) &\mbox{if}& x<\g(t,A_1,p_-),\\
 p_-&\mbox{if}& x>\g(t,A_1,p_-),\end{array}\right.
 \end{eqnarray*}
 then from (\ref{250}), $\ti{u}$ is the solution of (\ref{11}), (\ref{12}) with initial data $\ti{u}_0$.
 
 \noindent Case (i): Assume that there exists $t_0>0$ such that 
 $$R(t_0)=\g(t_0, A_1, p_-).$$
 From (4) of Lemma \ref{lemma25}, for $t>t_0$ it follows that $R_\pm(\g(t,A_1,p_-),t)=\g(t, A_1,p_-)$ and hence  $R(t)=\g(t,A_1,p_-)$.  From (\ref{252}), $\ti{u}(x,t)=u(x,t)\in [\B_2,\B_1]$ for $x<R(t)=\g(t,A_1,p_-)$ and $t>t_0$. Consequently  from (\ref{227}) for a.e. $x<\g(t,A_1,p_-)$, $t>t_0$, 
 \begin{eqnarray*}
 \frac{x-y_+(x,t)}{t}&=&f^\p(\ti{u}(x,t))\\
 &=&f^\p(u(x,t))\in [f^\p(\B_2),f^\p(\B_1)].
 \end{eqnarray*}
Letting $x\uparrow\g(t,A_1,p_-)$ to obtain 
$$f^\p(p_-)=\frac{\g(t,A_1,p_-)-A_1}{t}\in [f^\p(\B_2),f^\p(\B_1)].$$
\noindent Case (ii): Assume that for all $t>0$, 
$$R(\T)<\g(\T,A_1,p_-).$$
Suppose $f^\p(p_-)\notin [f^\p(\B_2),f^\p(\B_1)].$ Then choose an $\e>0,t_0>0$ such that for $t
\geq t_0,$
\be\h{288}
\left|\frac{R(t)-y(t)}{t}-f^\p(p_-)\right|<\e/2\\
\ [f^\p(p_-)-\e, f^\p(p_-)+\e]\cap [f^\p(\B_2),f^\p(\B_1)]=\phi\h{289}.
\ee
Then for $R(t_0)<x<\g(t_0,A_1,p_-), y_+(x,t_0)\leq A_1$ and 
\begin{eqnarray*}
\frac{R(t_0)-y(t_0)}{t_0}&\leq& \frac{x-y_+(x,t_0)}{t_0}+\frac{y_+(x,t_0)-y(t_0)}{t_0}\\
&\leq & \frac{x-y_+(x,t_0)}{t_0}+\frac{A_1-y(t_0)}{t_0}.
\end{eqnarray*}
Hence choose $t_0$ large such that for $R(t_0)<x<\g(t_o,A_1,p_-)$ with $y(x,t)=y_+(x,t)$ and 
\begin{eqnarray*}
\left|\frac{R(t_0)-y(t_0)}{t_0}-\frac{x-y(x,t_0)}{t_0}\right|<\e/2.
\end{eqnarray*}
Thus
\begin{eqnarray}\h{290}
\left|\frac{x-y(x,t_0)}{t_0}-f^\p(p_-)\right|<\e
\end{eqnarray}
and from (\ref{227}) for a.e. $x\in (R(t_0),\g(t_0,A_1,p_-))$
\begin{eqnarray*}
|f^\p(u(x,t_0))-f^\p(p_-)|<\e.
\end{eqnarray*}
Let 
\begin{eqnarray*}
\ [a,b]&=&(f^\p)^{-1}[f^\p(\B_2),f^\p(\B_1)]\\
\ [c(\e),d(\e)]&=&(f^\p)^{-1}[f^\p(p_-)-\e,f^\p(p_-)+\e].
\end{eqnarray*}
Then from (\ref{289}), $[a,b]\cap [c(\e),d(\e)]=\phi$. We have to consider two sub cases.

\noindent  (i): Suppose $f^\p(p_-)+\e<f^\p(\B_2).$ Then $d(\e)<a$ and by convexity  $(d(\e),a)$ contains open sets on which $f^\p$ is strictly increasing. Hence 
$$\min_{\substack{p\in[a,b]\\ q\in [c(0),d(0)]}} \frac{f(p)-f(q)}{p-q}> f^\p(p_-).$$ Therefore by continuity, there exists an $\e_0>0$ such that for $0<\e\leq \e_0$, 
$$m=\min_{\substack{p\in[a,b]\\ q\in [c(\e),d(\e)]}} \frac{f(p)-f(q)}{p-q}> f^\p(p_-)+\e.$$ Let $0<\e<\e_0$ and choose $t_0>0$ large such that  (\ref{288}), (\ref{289}), (\ref{290}) and the above inequality hold. 
By taking the initial data $\ti{u}(x,t_0)$ at $t=t_0$, then from (\ref{270}), there exists a Lipschitz curve $r(\cdot)$ for $t>t_0$, such that for $t>t_0$, 
\begin{eqnarray*}
r(t_0)&=& R(t_0),\\
\ti{u}(x,t)&\in& [a,b], \ \mbox{if}\ x<r(t),\\
\ti{u}(x,t)&\in& [c(\e),d(\e)], \ \mbox{if}\ x>r(t),\\
\frac{d r}{dt }(t)&\geq& m \geq f^\p(p_-)+\e.
\end{eqnarray*}
Hence for $t>t_0$, 
$$r(t)\geq R(t_0)+(t-t_0)(f^\p(p_-)+\e)=s(t)$$
and $s(t)$, $\g(t,A_1,p_-)$ intersect at $T$ given by 
\begin{eqnarray*}
R(t_0)+(T-t_0)(f^\p(p_-)+\e)&=&\g(T,A_1,p_-)\\
&=& \g(t_0, A_1,p_-) + (T-t_0)f^\p(p_-)\\
T&=& t_0+\frac{\g(t_0, A_1,p_-)-R(t_0)}{\e}>t_0.
\end{eqnarray*}
Thus the curve $r(\cdot), \g(\cdot, A_1,p_-)$ intersect at $\ti{T}\leq T.$ From (8) and (4) of Lemma \ref{lemma25}, $R(t)=r(t)$ for $t>t_0$. Hence $R(\ti{T})=r(\ti{T})=\g(\ti{T}, A_1,p_-)$. Therefore from case (i), $f^\p(p_-)\in [f^\p(\B_2), f^\p(\B_1)]$ which is a contradiction.

\noindent (ii): Suppose $f^\p(p_-)-\e>f^\p(\B_1),$ then $b<c(\e)$ and from convexity of $f$ as in case (i), we can choose $\e$ sufficiently small and $t_0$ large such that   
$$M=\max_{\substack{p\in[a,b]\\ q\in [c(0),d(0)]}} \frac{f(p)-f(q)}{p-q}\leq f^\p(p_-)-\e,$$
and for $t>t_0$, 
$$\frac{d r}{dt }(t)\leq M \leq f^\p(p_-)-\e.$$
Hence $$r(t)\leq R(t_0)+(t-t_0)(f^\p(p_-)-\e).$$
Since for $t>t_0$, $r(t)=R(t),$ we have 
\begin{eqnarray*}
f^\p(p_-)=\lim\limits_{t\rr \f}\frac{R(t)-y(t)}{t}&=&\lim\limits_{t\rr \f}\frac{r(t)-y(t)}{t}\\
&\leq & \lim\limits_{t\rr \f}\left[\frac{R(t_0)-y(t)}{t}+\left(1-\frac{t_0}{t}\right)(f^\p(p_-)-\e)\right]\\
&=&f^\p(p_-)-\e,
\end{eqnarray*}
which is a contradiction. This proves (\ref{280}).
\par Proof of $(a_2)$ follows in a similar way and this proves the lemma.
\end{itemize}
\end{proof}
\end{lemma}
Next we prove  a quantitative result regarding the single shock situation  by blow up analysis. 
\begin{lemma}\h{lemma27}
Let $A<B, u_\pm, \bar{u}_0,u_0$ as in $(\ref{16})$. Let $u$ be the solution of (\ref{11}), (\ref{12}). Also let $\al_1\leq \al_2$, $\B_1\leq \B_2$ such that 
\be\h{2900}
\left\{\begin{array}{lll}
u_-(x)\in [\B_2,\B_1] &\mbox{if}& x<A,\\
u_+(x)\in[\al_1,\al_2] &\mbox{if}& x>B.
 \end{array}\right.
\ee
Assume that $$f^\p(\B_2)>f^\p(\al_2).$$
Then there exist $(x_0,T_0)\in \R\times (0,\f), \g_0>0$ and Lipschitz curve $r(\cdot)$ depending only on 
$||u_0||_\f, f^\p(\B_2)-f^\p(\al_2)$ such that for $t>T_0$
\be
&&r(T_0)=x_0,\h{291}\\
&&T_0\leq \g_0|A-B|,\h{292}
\ee
\be\h{293}
\left\{\begin{array}{lll}
u(x,t)\in [\B_2,\B_1] &\mbox{if}& x<r(t),\\
u(x,t)\in[\al_1,\al_2] &\mbox{if}& x>r(t).
 \end{array}\right.
\ee
\begin{proof}
We prove this by blow up argument. Let
\be
(L(t),R(t))&=&(R_-(t,A,u_0), R_+(t,B,u_0)),\h{294}\\
(y_+(t),y_-(t))&=&(y_+(L(t),t),y_-(R(t),t))\h{295}.
\ee 
Clearly $L(t)\leq R(t)$. Suppose $L(t)<R(t)$ for all $t>0$. Then $A\leq y_+(t)\leq y_-(t)\leq B$ for all $t.$
Hence from (2) of Lemma \ref{lemma26}, there exist $p_-$ and $p_+$ such that 
\begin{eqnarray*}
f^\p(p_-)&=&\lim\limits_{t\rr \f}\frac{L(t)-y_+(t)}{t} \in [f^\p(\B_2),f^\p(\B_1)]\\
f^\p(p_+)&=&\lim\limits_{t\rr \f}\frac{R(t)-y_-(t)}{t} \in [f^\p(\al_1),f^\p(\al_2)]\\
\lim\limits_{t\rr \f}(y_+(t),y_-(t))&=&(A_1,B_1),
\end{eqnarray*}
and $\g(\cdot, A_1,p_-)$ and $\g(\cdot, B_1,p_+)$ are characteristic lines. 
Since $A_1\leq B_1$ and $f^\p(p_-)>f^\p(p_+)$, therefore $\g(\cdot, A_1,p_-)$ and $\g(\cdot, B_1,p_+)$ meet at $\ti{T}>0$ which contradicts the fact that they do not intersect. Hence let $L$ and $R$ meet at some $T_0>0$. Then from (4) of Lemma \ref{lemma25}, $L(t)=R(t)$ for $t>T_0$. Let $r(t)=L(t)$, then from (\ref{252}) and (\ref{253}), we have for $t>T_0$
\be\h{2944}
\left\{\begin{array}{lll}
u(x,t)\in [\B_2,\B_1] &\mbox{if}& x<r(t),\\
u(x,t)\in[\al_1,\al_2] &\mbox{if}& x>r(t).
 \end{array}\right.
\ee
We prove the uniform bounds on $T_0$ by assuming  $[A,B]=[0,1]$. Then  by blow up argument, we prove it for general $[A,B]$. Suppose not, then there exist $\bar{u}_{0,k}\in L^\f(0,1),$ $\al_{1,k}\leq \al_{2,k}, \B_{2,k}\leq \B_{1,k}$, $\delta>0$, $r_k(t),$ $T_k,$ convex $C^1$ functions $f_k$ and $f$ such that 
\begin{eqnarray*}
&&\sup\limits_k||u_{0,k}||_\f<\f, \lim\limits_{|p|\rr \f}\inf\limits_k\frac{f_k(p)}{|p|}= \f,  \ f_k\rr f\ \mbox{in}\ C^0_{\mbox{loc}}(\R),\\
&&\lim\limits_{k\rr \f}(\al_{1,k},\al_{2,k}, \B_{2,k},\B_{1,k}) =(\al_1,\al_2,\B_2,\B_1),\\
&&f^\p(\B_{2,k})-f^\p(\al_{2,k})\geq \delta
\end{eqnarray*}
and the solution $u_k$ of (\ref{11}), (\ref{12}) with flux $f_k$ and initial data $u_{0,k}$ satisfying (\ref{2944}) for $t>T_k$.
\par Let $(u_{-,k}, u_{+,k}, \bar{u}_{0,k})\rightharpoonup (u_-,u_+,\bar{u}_0)$ in $L^\f$ $\mbox{ weak }^*$ topology and 
\begin{eqnarray*}
u_0(x)=\left\{\begin{array}{lll}
u_-(x) &\mbox{if}& x<0,\\
\bar{u}_0(x) &\mbox{if}& x\in (0,1),\\
u_+(x) &\mbox{if}& x>1.
\end{array}\right.
\end{eqnarray*}
Then $u_-(x)\in [\B_2,\B_1]$ for $x<0$, $u_+(x)\in [\al_1,\al_2]$ if $x>1$. Let $u$ be the solution of (\ref{11})
with initial data $u_0$. Let $(L_k, R_k,y_{\pm,k})$ be as in (\ref{294}), (\ref{295}) for the solution $u_k$ and $T_k>0$ be the first point of interaction of $L_k$ and $R_k$. That is 
\begin{eqnarray*}
&&L_k(t)<R_k(t)\ \mbox{for}\ 0<t<T_k,\\
&&L_k(T_k)=R_k(T_k)=x_{0,k},\\
&&L_k(t)=R_k(t)=r_k(t)\ \mbox{for} \ t>T_k,\\
&&\lim\limits_{k\rr \f}T_k=\f.
\end{eqnarray*}
Hence $$0\leq y_{+,k}(T_k)\leq y_{-,k}(T_k)\leq 1.$$
Let \begin{eqnarray*}
\g_{1,k}(\T)&=& y_{+,k}(T_k)+\left(\frac{x_{0,k}-y_{+,k}(T_k)}{T_k}\right)\T,\\
\g_{2,k}(\T)&=& y_{-,k}(T_k)+\left(\frac{x_{0,k}-y_{-,k}(T_k)}{T_k}\right)\T,
\end{eqnarray*}
be the characteristic line segments. From (\ref{215}), $\left\{\frac{x_{0,k}-y_{\pm,k}(T_k)}{T_k}\right\}$, $\left\{\frac{dR_k}{dt}\right\}$, $\left\{\frac{dL_k}{dt}\right\}$ are uniformly bounded sequences, hence for  subsequences as $k\rr \f$, let 
\begin{eqnarray*}
(R_k,L_k)\rr (R_0,L_0)\ \mbox{in}\ C^0_{\mbox{loc}}(\R),\\
(y_{+,k}(T_k), y_{-,k}(T_k))\rr (y_+,y_-),\\
\frac{x_{0,k}-y_{\pm,k}(T_k)}{T_k} \rr f^\p(p_\pm).
\end{eqnarray*}
Let 
\begin{eqnarray*}
(\g_1(\T), \g_2(\T))&=&\lim\limits_{k\rr \f}(\g_{1,k}(\T), \g_{2,k}(\T))\\
&=&(y_++f^\p(p_-)\T, y_-+f^\p(p_+)\T).
\end{eqnarray*}
Therefore  from (7) of Lemma \ref{lemma23}, $y_+\in ch(\g_1(\T),\T), y_-\in ch(\g_2(\T),\T)$ for all $\T>0$. Hence $\g_1$ and $\g_2$ are characteristic lines with respect to the solution $u$.
\par Imitating the proof of (\ref{280}) and (\ref{285}), it follows that  $f^\p(p_-)\in [f^\p(\B_2), f^\p(\B_1)]$ and $f^\p(p_+)\in [f^\p(\al_1), f^\p(\al_2)].$
Since $f^\p(\B_2)>f^\p(\al_2)$, therefore $\g_1$ and $\g_2$ must necessarily intersect, contradicting that $\g_1$ and $\g_2$ are characteristic lines. This proves (\ref{291}) and (\ref{292}) for $[0,1]$. 
\par For general $[A,B]$, define 
$$w(x,t)=u(A+x(B-A), (B-A)t),$$
then from the previous analysis, there exist a $(\ti{x}_0, \ti{T}_0)\in\R\times(0,\f), \g_0>0$, $\ti{r}(\cdot)$ a Lipschitz curve such that (\ref{291}) to (\ref{293}) hold for $w$. Let $x_0=A+\ti{x}_0(B-A),$ $T_0=(B-A)\ti{T}_0$, $r(t)=\ti{r}\left(\frac{t}{B-a}\right)$, then (\ref{291}) to (\ref{293}) holds for $u$. This proves the Lemma.
\end{proof}
\end{lemma}
\begin{theorem}[Structure Theorem]\label{theorem28}
Let $f\in C^1(\R)$ and convex with finite number of degeneracies. Assume that $\al_1\leq \al_2$, $\B_2\leq \B_1$, $u_\pm\in L^\f(\R)$, $\bar{u}_0\in L^\f(\R),$ $A<B$ such that 
\be\h{296}
\left\{\begin{array}{lll}
u_-(x)\in [\B_2,\B_1] &\mbox{if}& x<A,\\
u_+(x)\in [\al_1,\al_2] &\mbox{if}& x>B, 
\end{array}\right.
\ee 
\be\h{297}
u_0(x)=\left\{\begin{array}{lll}
u_-(x) &\mbox{if}& x<A,\\
\bar{u}_0(x) &\mbox{if}& x\in(A,B),\\
u_+(x) &\mbox{if}& x>B 
\end{array}\right.
\ee 
 and let $u$ be the solution of (\ref{11}), (\ref{12}). Then 
 \begin{itemize}
 \item [1.] Shock case: $\{u_-,u_+\}$ gives rise to a single shock solution if 
 \be\h{298}
 f^\p(\B_2)>f^\p(\al_2),
 \ee
 and $(x_0,T_0)$, $\g_0, r(\cdot)$ exist as in (\ref{17}) to (\ref{110}) which depends on $f^\p(\B_2)-f^\p(\al_2)$, $||u_0||_\f$, $Lip (f,[-||u_0||_\f, ||u_0||_\f])$ and uniform growth of $\frac{f^*(p)}{|p|}$ as $|p|\rr \f.$
 \item [2.] Rarefaction case: Assume that 
 \be\h{299}
 f^\p(\B_1)\leq f^\p(\al_2),
 \ee
 then there exist $A\leq A_1 \leq B_1 \leq B,\ p_\pm\in \R$ with  $f^\p(p_-)\in [f^\p(\B_2),f^\p(\B_1)]$, $f^\p(p_+)\in [f^\p(\al_1), f^\p(\al_2)]$ and a countable number of ASSP $D_i=D(a_i,b_i,p_i)$ such that 
 \begin{itemize}
 \item [i.] $\g_-(t)=A_1+tf^\p(p_-), \g_+(t)=B_1+tf^\p(p_+)$ are characteristic lines. 
 \item [ii.] $D(a_i,b_i,p_i)\subset\{(x,t)\in\R\times(0,\f): \g_-(t)<x<\g_+(t)\}$.
 \item [iii.] Define $F_\pm, D_\pm, R$ by 
 \begin{eqnarray*}
 F_-&=&\{(x,t): x<R_-(t,A,u_0)\},\\
  F_+&=&\{(x,t): x>R_+(t,B,u_0)\},\\
 D_-&=&\{(x,t): R_-(t,A,u_0)<x<\g_-(t)\},\\
 D_+&=&\{(x,t): \g_+(t)<x <R_+(t,B,u_0)\},\\
R&=&\{(x,t)\notin D_-\cup D_-\cup_i D(a_i,b_i,p_i): \g_-(t)<x<\g_+(t)\},
 \end{eqnarray*}
 then 
 \begin{eqnarray*}
 f^\p(u(x,t))\in [f^\p(\B_2), f^\p(\B_1)] &\mbox{if}& (x,t)\in F_-\\
 f^\p(u(x,t))\in [f^\p(\al_1), f^\p(\al_2)] &\mbox{if}& (x,t)\in F_+
 \end{eqnarray*}
 \item [iv.] $(x,t)\in R$, then $(x,t)$ lies on a characteristic line.
 \item [v.] If for some $\eta>0$, $u_0$ is continuous in $[a,a+\eta)$ and $(b-\eta,b]$, then 
 $$u_0(a_i)=u_0(b_i)=p_i.$$
 \item [vi.] If $u_0$ is monotone in $[a,a+\eta)$  and $(b-\eta,b]$, then $u_0$ is increasing in $[a,a+\eta)$ and $(b-\eta,b]$.
 \item [vii.] $\displaystyle\frac{1}{b_i-a_i}\int\limits_{a_i}^{b_i} u_0(x)dx=p_i.$
 \end{itemize}
 \end{itemize}
 \begin{proof}
 Let $L(t)=R_-(t,A,u_0), R(t)=R_+(t,B,u_0), y_+(t)=y_+(R_-(t,A,u_0),t)$, $y_-(t)=y_-(R_+(t,B,u_0),t).$
 \begin{itemize}
 \item [1.] Shock case follows from Lemma \ref{lemma27}.
 \item [2.] Assume (\ref{299}). 
 
\noindent If for all  $t>0$, $L(t)<R(t)$, then $A\leq y_+(t)\leq y_-(t)\leq B$, hence existence of $A_1, B_1, p_-, p_+$ follows from Lemma \ref{lemma26}. Assume that there exists $T>0$ such that $L(T)=R(T)$. Then from (4) of \ref{lemma25}, $L(t)=R(t)$ for all $t\geq T$. From (\ref{252}), for $t>T$,
 \begin{eqnarray*}
 u(x,t)\in\begin{cases}
  [\B_2,\B_1] \mbox{ if } x<L(t),\\
  [\al_1,\al_2] \mbox{ if } x>L(t).
 \end{cases}
 \end{eqnarray*}
 Let $x\uparrow L(t), \xi\downarrow R(t)$, then $y_+(x,T,t)\rr y_-(L(t),T,t)$ and $y_-(\xi, T,t)\rr y_+(R(t), T,t)$.
 
 \noindent Claim: $y_\pm(L(t),T,t)=L(T)$ and $L(t)$ is a line segment for $t>T$.
 
 \par Suppose not, say $y_-(L(t),T,t)<L(T)\leq y_+(L(t), T,t)$, then 
 \begin{eqnarray}\h{2100}
 \begin{array}{lll}
\displaystyle \frac{L(t)-y_-(L(t),T,t)}{t-T}&=&\displaystyle \lim\limits_{x\uparrow L(t)} \frac{x-y_+(x,T,t)}{t-T}\\
 &=&\displaystyle \lim\limits_{x\uparrow L(t)} f^\p(u(x,t))\in [f^\p(\B_2),f^\p(\B_1)],\end{array}
 \end{eqnarray}
  \begin{eqnarray}\h{2101}
  \begin{array}{lll}
 \displaystyle \frac{L(t)-y_+(L(t),T,t)}{t-T}&=&\displaystyle \lim\limits_{\xi\downarrow L(t)} \frac{\xi-y_-(\xi,T,t)}{t-T}\\
 &=&\displaystyle \lim\limits_{\xi\downarrow L(t)} f^\p(u(\xi,t))\in [f^\p(\al_1),f^\p(\al_2)].\end{array}
 \end{eqnarray}
 But $$\frac{L(t)-y_-(L(t),T,t)}{t-T}>\frac{L(t)-y_+(L(t),T,t)}{t-T},$$ which contradicts (\ref{299}).
 Hence $\g(\T)=L(T)+\left(\frac{L(t)-L(T)}{t-T}\right)(\T-T)$ is the characteristic line segment joining $(L(t),t)$ and $(L(T),T)$. Thus $L(\T)\leq \g(\T)\leq R(\T)$ for $\T\in [T,t]$. Since $L(\T)=R(\T)$, for $\T>T$, which implies that $L(\T)=\g(\T)$. Since $t$ is arbitrary, this implies that $L(t)$ is a straight line from $(L(T),T)$.
 \par Let $x_0=L(T)=R(T)$ and $\g_\pm(\T)=y_\pm(x_0,T)+\left(\frac{x_0-y_\pm(x_0)}{T}\right)\T$ the characteristic line segments at $(x_0,T)$. Then  by the strict convexity of $f^*$, the curves 
 \begin{eqnarray*}
 \g_1(t)=\left\{\begin{array}{lll}
 \g_+(t) &\mbox{if}& 0<t<T,\\
 L(t) &\mbox{if}& t>T,\end{array}\right.
 \end{eqnarray*}
  \begin{eqnarray*}
 \g_2(t)=\left\{\begin{array}{lll}
 \g_-(t) &\mbox{if}& 0<t<T,\\
 L(t) &\mbox{if}& t>T,\end{array}\right.
 \end{eqnarray*}
 are characteristic lines and hence  $\g_1=\g_2=L=R.$ Therefore $A_1=B_1=y_\pm(x_0,T)$ and $f^\p(p_-)=f^\p(p_+)$$=f^\p(\B_2)=f^\p(\al_1)$. Hence either $L(t)<R(t)$ for all $t>0$ or $L(t)=R(t)$ for all $t>0$ and is a characteristic line. (ii), (iii) and (iv) follow exactly as in \cite{1}.
 
Observe that (v), (vi) and (vii) give information about $u_0$ in an ASSP. Since $f$ is not strictly convex,  the  proof does not follow  from \cite{1} and it is quite delicate, therefore we adopt a different procedure. In order to do this, we concentrated only on  $D(a_i,b_i,p_i)$ and by approximation procedure we prove (v), (vi) and (vii).  For the sake of simplicity denote an ASSP $D(a_i,b_i,p_i)$ by $D(a,b,p)$. Define 
\begin{eqnarray*}
\ti{u}_0(x)=\left\{\begin{array}{llll}
u_0(x) &\mbox{if}& x\in (a,b),\\
p &\mbox{if}& x\notin(a,b).
\end{array}\right.
\end{eqnarray*}
\begin{eqnarray*}
\ti{u}(x,t)=\left\{\begin{array}{llll}
u(x,t) &\mbox{if}& (x,t)\in D(a,b,p),\\
p &\mbox{if}& (x,t)\notin D(a,b,p).
\end{array}\right.
\end{eqnarray*}
Then from (\ref{249}) and (\ref{251}) $\ti{u}$ is the solution of (\ref{11}) with initial data $\ti{u}_0$. Moreover $\g(\cdot, a, p)$ and $\g(\cdot, b, p)$ are the characteristic lines and $D(a,b,p)$ does not contain any other characteristic line. 
\par Let $f_\e\in C^2(\R)$ be a uniformly convex flux converging to $f$ in $C^1_{\mbox{loc}}(\R)$ (see (\ref{42}), (\ref{44}) and (\ref{410})) with $\lim\limits_{|p|\rr\f}\inf\limits_{\e}\frac{f_\e(p)}{|p|}=\f$. Then $f_\e^*\rr f^*$ in  $C^0_{\mbox{loc}}(\R)$. Let $u_\e$ be the solution of (\ref{11}) with flux $f_\e$ and initial data $\ti{u}_0$. Let $L_\e(t)=R_-(t,a,f_\e)$ and $R_\e(t)=R_+(t,b,f_\e)$ be the extreme characteristic curves. Let $y_\e(t)=y_+(L_\e(t),t)$, $Y_\e(t)=y_-(R_\e(t),t)$. Then from the structure Theorem \cite{1}, there exist $a\le a_\e \leq b_\e\leq b$ such that for $i=1,2$, $\g_{1,\e}, \g_{2,\e}$ are characteristic lines where 
\begin{eqnarray*}
\lim\limits_{t\rr\f}(y_\e(t), Y_\e(t))&=&(a_\e,b_\e)\\
\g_{1,\e}(t)&=&\g(t,a_\e, p)=a_\e+tf^\p_\e(p)\\
\g_{2,\e}(t)&=& \g(t,b_\e,p)=b_\e+tf^\p_\e(p).
\end{eqnarray*}
Let for a subsequence $\{a_\e,b_\e\}\rr \{\ti{a}, \ti{b}\}$ as $\e\rr 0$. Since the limits of characteristics are characteristics, we obtain $\g_1$ and $\g_2$ are characteristics lines corresponding to $\ti{u}$ where $$(\g_1(t),\g_2(t))=\lim\limits_{\e\rr 0}(\g_{1,\e}(t),\g_{2,\e}(t))=(\ti{a}+tf^\p(p), \ti{b}+tf^\p(p)).$$
\noindent Claim 1: For a subsequence 
$$\lim\limits_{\e\rr0}(L_\e(t), R_\e(t))=(\g(t,a,p), \g(t,b,p)).$$
Let for a subsequence $\lim\limits_{\e\rr 0}(L_\e(t), R_\e(t))=(L(t), R(t))$ in  $C^0_{\mbox{loc}}(\R)$. Since 
$$y_-(L_\e(t),t)\leq a \leq y_+(L_\e(t),t),$$
therefore by letting $\e\rr \f$ to obtain 
$$y_-(L(t),t)\leq a \leq y_+(L(t),t),$$
and then
$$\g(t, a, p) =R_-(t,a,f)\leq L(t)\leq R_+(t,a,f)=\g(t,a,p),$$ because  $\g(t,a,p)$ is a characteristic line. Hence $L(t)=\g(t,a,p)$ and similarly $R(t)=\g(t,b,p)$ and this proves the claim 1. \\
\noindent Claim 2: Let for a subsequence $\lim\limits_{\e\rr 0} (a_\e, b_\e)=(\ti{a}, \ti{b}),$ then $\ti{a}=a, \ti{b}=b.$\\
\par If not,   then let $0<\ti{a}\leq b$. If $\ti{a}<b$, then  $\g_1(t)=\lim\limits_{\e\rr 0}\g_{1,\e}(t)=\ti{a}+tf^\p(t)$ is a characteristic line with  respect to $\ti{u}$, which is a contradiction since $\g_1(t)\in D(a,b,p)$ is an ASSP. 
\par Hence $\ti{a}=b$. Let $\e_0>0$ such that for $0<\e<\e_0$, $a_\e>a+\frac{3}{4}(b-a).$ Then $s_\e(t)=R_-(t,\frac{a+b}{2}, f_\e)$ meet $L_\e$ at $t=t_\e$ and 
\be\h{nn}
y_-(s_\e(t_\e),t_\e)\leq \frac{a+b}{2}\leq y_+(s_\e(t_\e),t_\e).
\ee
Suppose $\lim\limits_{\e\rr 0} t_\e=\f$, then choose $\ti{t}_\e<t_\e$ such that $\lim\limits_{\e\rr 0} y_+(L(\ti{t}_\e), \ti{t}_\e)=b$. Hence the characteristic line segments $L(\ti{t}_\e)+\frac{y_+(L(\ti{t}_\e), \ti{t}_\e)-L(\ti{t}_\e)}{\ti{t}_\e}(t-\ti{t}_\e)$ and $L(t_\e)+\frac{y_-(L(t_\e),t_\e)-L(t_\e)}{t_\e}(t-t_\e)$ meet for some $0<t_0<\min(\ti{t}_\e,t_\e)$ which contradicts the non intersection of characteristics lines. Hence $\{t_\e\}$ is bounded. Let $t_1=\lim\limits_{\e\rr 0} t_\e$ and $s(t)=\lim\limits_{\e\\r 0} s_\e(t)$ for $0\leq t \leq t_1$. Then from (\ref{nn}), we have for $0\leq t \leq t_1$ and from claim 1
\begin{eqnarray*}
&&R_-\left(t,\frac{a+b}{2}, f\right)\leq s(t) \leq R_+\left(t,\frac{a+b}{2}, f\right),\\
&&s(t_1)=\g_1(t_1, a, p).
\end{eqnarray*}
This implies that $R_-(t,\frac{a+b}{2},f)$ meet $\g(t,a, p)$ which contradicts that $D(a,b,p)$ is an ASSP, which proves claim 2. Let 
$$E_\e(t_0)=\{(x,t): a_\e+tf^\p_\e(p)<x<b_\e+tf^\p_\e(p), 0<t<t_0\},$$
then integrating the equation (\ref{11}) with flux $f_\e$ in $E_\e(t_0)$ to obtain 
\begin{eqnarray*}
\int\limits_{a_\e}^{b_\e}u_0(x)dx &=& \int\limits_{a_\e+t_0f_\e^\p(p)}^{b_\e+t_0f_\e^\p(p)}u_\e(x,t_0)dx\\
&=&  \int\limits_{a_\e+t_0f_\e^\p(p)}^{b_\e+t_0f_\e^\p(p)}(f^\p_\e)^{-1}\left(\frac{x-y_+(x,t_0,f_\e)}{t_0}\right)\\
&=&  \int\limits_{a_\e}^{b_\e}(f^\p_\e)^{-1}\left(f^\p_\e(p)+\frac{\xi-y_+(\xi+t_0f^\p(p),t_0,f_\e)}{t_0}\right),
\end{eqnarray*}
letting $t_0\rr \f$ to obtain 
$$\int\limits_{a_\e}^{b_\e}u_0(x)dx=p(b_\e-a_\e).$$ Now letting $\e\rr 0$ and from claim 2, we obtain 
$$\frac{1}{b-a}\int\limits_{a}^{b}u_0(x)dx=p.$$
This proves (vii). If $u_0$ is continuous in $[a,a+\eta)$ and $(b-\eta,b]$ for some $\eta>0$, then for $\e$ small and $\eta$ small, $u_0$ is continuous in $[a_\e,a_\e+\eta), (b_\e-\eta,b_\e]$ and hence from structure Theorem \cite{1}, we have $u_0(a_\e)=u_0(b_\e)=p$. Now letting $\e\rr 0$ to obtain $u_0(a)=u_0(b)=p$ and this proves (v). Furthermore if  $u_0$ is monotone in $[a,a+\eta), (b-\eta,b]$ then  for $\e$ small, $u_0$ is monotone in  $[a_\e,a_\e+\eta), (b_\e-\eta,b_\e]$. Hence letting $\e\rr 0$ to obtain that $u_0$ is increasing in  $[a,a+\eta), (b-\eta,b]$ for some $\eta>0$. This proves (v).  Hence the Theorem.
  \end{itemize}
 \end{proof}
\end{theorem} 
 
 \begin{remark}
 Decay estimates and $N$-wave follow exactly as in \cite{1, shyamcontrol} and therefore we omit here.
 \end{remark}

\section{Proof of the main theorem}
\setcounter{equation}{0}
First we give the proof of part I of the main Theorem using the convex modification (\ref{428}) and the structure  Theorem \ref{theorem28}. 
\begin{lemma}\h{lemma46}
Let $(f,C,D)$ be a convex-convex type triplet and $\al_1\leq \al_2 \leq \ti{\al}<C\leq D< \ti{\B}\leq \B_2\leq \B_1 $. Assume that 
\be\h{n428}
f^\p(\ti{\al})<f^\p(\ti{\B}),\ f(\T)<L_{\ti{\al},\ti{\B}}(\T), \ \mbox{for } \T\in[C,D].
\ee
Let $A<B, \bar{u}_0\in L^\f(A,B)$ such that 
\begin{eqnarray}
\mbox{either}\ \bar{u}_0(x)\in (-\f,\ti{\al}] &\mbox{for}& x\in[A,B], \h{430}\\
\mbox{or} \ \ \ \ \ \bar{u}_0(x)\in [\ti{\B},\f) &\mbox{for}& x\in [A,B].\h{431}
\end{eqnarray}
Let $u_\pm\in L^\f(\R)$ satisfies (\ref{119}) and (\ref{120}) and $u_0$ be defined as in (\ref{16}). Let $u$ be the solution of (\ref{11}), (\ref{12}). 
Then there exist a $\g>0$, $(x_0,T_0)\in \R\times(0,\f)$  and a Lipschitz curve $r(\cdot)$ in $[T_0,\f)$ such that for $t>T_0$,
\be\h{430}
r(T_0)=x_0, T_0\leq \g |A-B|,
\ee
\be\h{431}
u(x,t)\in \begin{cases}
 [\al_1,\al_2] \mbox{ if } x>r(t)\\
 [\B_2,\B_1] \mbox{ if } x<r(t).
\end{cases}
\ee
Furthermore if $(g,\ti{C},\ti{D})$ be another convex-convex type triplet such that $g=f$ in a neighbourhood of $(-\f,\ti{\al})\cup (\ti{\B},\f)$, then $u$ is also the solution of (\ref{11}), (\ref{12}) with flux $g$.
\begin{proof}
Let $\ti{\al}<\al<C\leq D<\B<\ti{\B}$ such that for $\T\in [C,D],$
\be\h{432}
f^\p(\al)<f^\p(\B),\ f(\T)<L_{\al,\B}(\T)
\ee
and let $\ti{f}$ be a convex modification of $f$ as in (\ref{428}).
 First assume that 
\be\h{433}
\bar{u}_0(x)\in (-\f,\ti{\al}], \ \mbox{for}\ x\in [A,B].
\ee
Now consider a Riemann problem 
\be\h{434}
w_0(x)=\left\{\begin{array}{lll}
a &\mbox{if}& x<z_0,\\
b &\mbox{if}& x>z_0,
\end{array}\right.
\ee
where $a\leq \ti{\al}$ and $b\geq \ti{\B}$. Then from Oleinik entropy condition and from (\ref{432}), (\ref{427}), the solution to (\ref{11}) with flux $f$ and initial data  (\ref{434}) is same as the solution to (\ref{11}) with flux $\ti{f}$ and data (\ref{434}). Hence by front tracking Lemma \ref{lemma22}, the solution to (\ref{11}) with flux $f$ and the data given by the hypothesis is same as that of (\ref{11}) with flux $\ti{f}$. Furthermore range of $u$ is contained in $(-\f,\ti{\al}]\cup [\ti{\B},\f)$ and hence this is also the solution of (\ref{11}) with flux $g$.
\par From (\ref{427}), $f^\p(\ti{\al})\leq f^\p(\al)<f^\p(\B)\leq f^\p(\ti{\B}).$ Hence $\ti{f}^\p(\ti{\B})>\ti{f}^\p(\ti{\al})$. Therefore (\ref{430}), (\ref{431}) follows from (1) of Theorem \ref{theorem28} with flux $\ti{f}$. Similarly the result follows if $\bar{u}_0(x)\in [\ti{\B},\f)$ for $x\in (A,B)$. This proves the lemma.
\end{proof}
\end{lemma}

\begin{proof}[Proof of the Theorem for convex-convex type flux]  Let $(f,C,D)$ be a convex-convex type triplet and $\al_1\leq \al_2<C<D<\B_2\leq \B_1$ satisfying (\ref{116}) to (\ref{118}) and (\ref{121}). Let $u_0,u_\pm, \bar{u}_0$ be as in (\ref{16}), (\ref{119}) and (\ref{120}).  Let 
$$m_1=\min\{\al_1,\inf \bar{u}_0\}, m_2=\max\{\B_1, \sup\bar{u}_0\},$$
\begin{eqnarray*}
u_{0,m_i}(x)=\left\{\begin{array}{lll}
u_0(x) &\mbox{if}&x\notin (A,B),\\
m_i &\mbox{if}& x\in (A,B),\end{array}\right.
\end{eqnarray*}
and $u_i$ for $i=1,2$ be the solution of (\ref{11}) and (\ref{12}) with initial data $u_{0,m_i}$. Since $u_{0,m_1}\leq u_0 \leq u_{0,m_2}$, therefore $u_1\leq u \leq u_2$. From Lemma \ref{lemma46}, for $i=1,2$, there exist $\g_i, (x_i,T_i)\in \R\times(0,\f)$, Lipschitz curves $r_i(\cdot)$ such that for $t>T_i$
\begin{eqnarray*}
r_i(T_i)=x_i,\ T_i\leq \g_i|A-B|,\end{eqnarray*}
\begin{eqnarray*}
u_i(x,t)\in \begin{cases} [\al_1,\al_2]\mbox{ if } x>r_i(t)\\
  [\B_2,\B_1]\mbox{ if } x<r_i(t).\end{cases}
\end{eqnarray*}
Let $T=\max\{T_1,T_2\}$, $\g=\max\{\g_1,\g_2\}$, $A_1=r_1(T), B_1=r_2(T),$ then from $u_1\leq u \leq u_2$, to obtain 
\begin{eqnarray*}
u(x,t)\in \begin{cases} [\al_1,\al_2]\mbox{ if } x>B_1,\\
 [\al_1,\B_1]\mbox{ if }A_1<x<B_1,\\
  [\B_2,\B_1]\mbox{ if } x<A_1.\end{cases}
\end{eqnarray*}
\be\h{31}
T\leq \g |A-B|,
\ee
where $\g,T$ depending only on $||u_0||_\f$ and $f^\p(\B_2)-f^\p(\al_2).$
\par Hence without loss of generality we can assume that 
\be\h{32}
\bar{u}_0(x)\in [\al_1,\B_1]\ \mbox{for}\ x\in (A,B).
\ee
From (\ref{118}) and (\ref{121}), choose an $\e>0, (\eta_1,\xi_1), (\eta_2,\xi_2)$ in $\R^2$ such that 
\be
\al_1+\e<C<D<\B_1-\e,\h{33}\\
(\B_1-\B_2)+(\al_2-\al_1)<\e,\h{34}\\
f(\T)<L_{\al_1+\e, \B_1-\e}(\T)\ \mbox{for}\ \T\in [C,D],\h{35}\\
\eta_1<\al_1<D<\xi_1<\B_1-\e,\h{36}\\
\al_1+\e<\eta_2<C<\B_1<\xi_2,\h{37}\\
\mbox{ For } \T\in [C,D],\  f(\T)<\min\{L_{\eta_1,\xi_1}(\T),L_{\eta_2,\xi_2}(\T)\}\h{38}.
\ee
Let 
\begin{eqnarray*}
u_{0,A}=\left\{\begin{array}{lll}
\eta_1 &\mbox{if}& x>B,\\
\xi_1 &\mbox{if}& x<B.
\end{array}\right.
\end{eqnarray*}
\begin{eqnarray*}
u_{0,B}=\left\{\begin{array}{lll}
\eta_2 &\mbox{if}& x>A,\\
\xi_2 &\mbox{if}& x<A.
\end{array}\right.
\end{eqnarray*}
\begin{eqnarray*}
l_{1,B}(\T)&=& B+\left(\frac{f(\eta_1)-f(\xi_1)}{\eta_1-\xi_1}\right)\T, \\
l_{2,A}(\T)&=& A+\left(\frac{f(\eta_2)-f(\xi_2)}{\eta_2-\xi_2}\right)\T .
\end{eqnarray*}
Then the solutions $u_A$ and $u_B$ of (\ref{11}) with respective initial data $u_{0,A}$ and $u_{0,B}$ are given by 
\begin{eqnarray*}
u_{A}(x,t)=\left\{\begin{array}{lll}
\eta_1 &\mbox{if}& x>l_{1,B}(t),\\
\xi_1 &\mbox{if}& x<l_{1,B}(t).
\end{array}\right.
\end{eqnarray*}
\begin{eqnarray*}
u_{B}(x,t)=\left\{\begin{array}{lll}
\eta_2 &\mbox{if}& x>l_{2,A}(t),\\
\xi_2 &\mbox{if}& x<l_{2,A}(t).
\end{array}\right.
\end{eqnarray*}
\noindent Claim: There exist $\g>0, T_1>0, A_1\leq B_1$, depending only on $||u_0||_\f$ and $f^\p(\B_2)-f^\p(\al_2)$ such that 
\be
l_{1,B}(T_1)\leq A_1 \leq B_1 \leq l_{2,A}(T_1),\h{39}\\
T_1\leq \g|A-B|,\h{310}\\
|A_1-B_1|\leq \frac{\B_1-\al_1-\e}{\B_2-\al_2}|A-B|\h{311},
\ee
\be\h{312}
u(x,t)\in \begin{cases} [\al_1,\al_2] \mbox{ if } x>B_1,\\
 [\al_1,\B_1]\mbox{ if } A_1<x<B_1,\\
  [\B_2,\B_1]\mbox{ if } x<A_1.\end{cases}
\ee
In order to prove the claim, define 
\be\h{313}
u_{0,1}(x)=\left\{\begin{array}{lll}
u_0(x) &\mbox{if}& x\notin(A,B)\ \mbox{or}\ u_0(x)<\al_1+\frac{\e}{2},\\
\al_1+\frac{\e}{2} &\mbox{if}& x\in(A,B),\ u_0(x)\geq \al_1+\frac{\e}{2},\\
\end{array}\right.
\ee
\be\h{314}
u_{0,2}(x)=\left\{\begin{array}{lll}
u_0(x) &\mbox{if}& x\notin(A,B)\ \mbox{or}\ u_0(x)\geq \B_1-\frac{\e}{2},\\
\B_1-\frac{\e}{2} &\mbox{if}& x\in(A,B), u_0(x)\leq \B_1-\frac{\e}{2}\\
\end{array}\right.
\ee
and $u_1,u_2$ be the corresponding solutions of (\ref{11}),(\ref{12}) with respect to the initial data $u_{0,1}$ and $u_{0,2}$. From (\ref{33}), (\ref{38})  and Lemma \ref{lemma46}, there exist  $\g>0, T_1>0, A_1, B_1$ such that 
$$T_1\leq \g |A-B|,$$
\begin{eqnarray*}
u_1(x,T_1)\in \begin{cases} [\al_1,\al_2]\mbox{ if } x>A_1,\\
  [\B_2,\B_1]\mbox{ if } x<A_1.\end{cases}
\end{eqnarray*}
\begin{eqnarray*}
u_2(x,T_1)\in \begin{cases} [\al_1,\al_2] \mbox{ if } x>B_1,\\
  [\B_2,\B_1]\mbox{ if } x<B_1.\end{cases}\end{eqnarray*}
Since $u_{0,1}\leq u_0 \leq u_{0,2}$, therefore $u_1\leq u \leq u_2$. This implies that $A_1\leq B_1$ and 
\begin{eqnarray}\h{315}
u(x,T_1)\in \begin{cases} [\al_1,\al_2]\mbox{ if } x>B_1,\\
  [\al_1,\B_1]\mbox{ if } A_1<x<B_1,\\
   [\B_2,\B_1]\mbox{ if } x<A_1.\end{cases}
\end{eqnarray}
Also $u_{0,A}\leq u_{0,1}\leq u_{0,2}\leq u_{0,B}$ and hence $u_A\leq u_1\leq u_2 \leq u_B$. Therefore 
\be\h{316}
l_{1,B}(T_1)\leq A_1\leq B_1\leq l_{2,A}(T_1)
\ee
and from $L^1_{\mbox{loc}}$ contraction, we have from (\ref{313}), (\ref{314})
\begin{eqnarray*}
(\B_2-\al_2)|A_1-B_1| &\leq& \int\limits_{A_1}^{B_1}|u_1(x,T_1)-u_2(x,T_1)|dx\\
 &\leq& \int\limits_{A}^{B}|u_{0,1}(x)-u_{0,2}(x)|dx\\
 &\leq& (\B_1-\al_1-\e)|A-B|.
\end{eqnarray*}
This gives 
\be\h{317}
|A_1-B_1|\leq \left(\frac{\B_1-\al_1-\e}{\B_2-\al_2}\right)|A-B|.
\ee
This proves the claim. 
\par Repeating the above procedure for $t>T_1$, by induction we can find $\g>0$, sequence $T_n>T_{n+1}$, $A_n\leq B_n$, with $A_0=A, B_0=B$ such that for $n\geq 1$,
\be
l_{1,B_{n-1}}(T_n)\leq A_n \leq B_n\leq l_{2,A_{n-1}}(T_n),\h{318}\\
T_n\leq T_{n-1}+\g|A_{n-1}-B_{n-1}|,\h{319}\\
|A_n-B_n|\leq \left(\frac{\B_1-\al_1-\e}{\B_2-\al_2}\right)|A_{n-1}-B_{n-1}|,\h{320}
\ee
\be\h{321}
u(x,T_n)\in \begin{cases} [\al_1,\al_2]\mbox{ if } x>B_n,\\
  [\al_1,\B_1]\mbox{ if } A_n<x<B_n,\\
   [\B_2,\B_1]\mbox{ if } x<A_n.\end{cases}
\ee
From (\ref{34}), (\ref{319}) and (\ref{320}) we have 
\begin{eqnarray*}
&&\delta=\left(\frac{\B_1-\al_1-\e}{\B_2-\al_2}\right)<1,\\
&&T_n\leq \g(1+\delta+\delta^2+\cdots+\delta^{n-1})|A-B|\leq \frac{\g}{1-\delta}|A-B|,\\
&&|A_n-B_n| \leq \delta^n|A-B|.
\end{eqnarray*}
Hence $\{T_n\}$ is bounded, from (\ref{318}), $A_n$, $B_n$ are bounded and $|A_n-B_n|\rr 0$ as $n\rr \f.$ Let 
$$(x_0,T_0)=\lim\limits_{n\rr \f} (A_n,T_n).$$ 
Then from (\ref{321}), we have 
$$T_0\leq \frac{\g}{1-\delta}|A-B|,$$
\begin{eqnarray*}
u(x,T_0)\in\begin{cases}
[\al_1,\al_2] \mbox{ if } x<x_0,\\
 [\B_2,\B_1] \mbox{ if } x>x_0.
\end{cases}
\end{eqnarray*}
From Lemma \ref{lemma46}, there exists a Lipschitz curve $r(\cdot)$ such that $r(T_0)=x_0$ and for $t>T_0$, 
\begin{eqnarray*}
u(x,t)\in\begin{cases}
 [\al_1,\al_2] \mbox{ if } x>r(t),\\
[\B_2,\B_1] \mbox{ if } x<r(t).
\end{cases}
\end{eqnarray*}
This concludes the proof of the Theorem for the convex-convex type flux.\end{proof}
\par Next we consider the convex-concave type flux and the method of the proof is different from that of convex-convex type flux as they have different polarity. 
 In order to prove the second part of the Theorem  with condition 1 ((\ref{123}) to (\ref{126})) we need to prove the following Lemmas whose proof depends on the front tracking and first part of the structure Theorem. Basically this Lemma reduces the problem to having $\bar{u}_0\in [\al_1,\B_1].$
\par Second part of the Theorem  with condition 2 follows in a similar way where one has to use structure Theorem for concave flux instead of convex flux.
\par From Lemma \ref{lemma23} we will prove the following Lemma which is essential to prove the second part of the main Theorem.
\par Let $(f,C,D)$ be a convex-concave type triplet, $ \al_1,\al_2,\beta_1, \beta_2$ and $f$ satisfies 
\be
&&\al_1\leq \al_2 <C<D<\beta_2\leq \beta_1, \h{455}\\
&& f(\T)>L_{\al_2,\beta_2}(\T),\ \T\in [C,D],\h{456}\\
&& f(\beta_2)<L_{\al_1}(\beta_2)\h{457}.
\ee
From (\ref{457}), choose $\al_0<\al_1<D$ (see figure \ref{Fig3}) such that 
\be\h{458}
f(\beta_2)=L_{\al_0}(\beta_2).
\ee
Let $u_\pm, \bar{u}_0\in L^\f(\R)$ such that 
\be\h{459}
u_+(x)\in [\beta_2,\beta_1], \ u_-(x)\in [\al_1,\al_2],
\ee
\be\h{460}
u(x,t)=\left\{\begin{array}{lll}
u_+(x) &\mbox{if}& x>B,\\
\bar{u}_0(x) &\mbox{if}& x\in (A,B),\\
u_-(x) &\mbox{if}& x<A,\end{array}\right.
\ee
and $u$ be the solution of (\ref{11}) and (\ref{12}). Then 
\begin{lemma}\h{lemma47}
Let $\B_0\geq \B_2$ and assume that $\bar{u}_0$ satisfies one of the following conditions 
\begin{itemize}
\item [1.] $\bar{u}_0(x)\in [\al_0,\al_2]$ if $x\in (A,B)$.
\item [2.] $\bar{u}_0(x)\in [\B_2,\B_0]$ if $x\in (A,B)$.
\end{itemize}
Then there exist $(x_0,T_0)\in \R\times (0,\f)$, $\g>0$ and a Lipschitz curve $r(\cdot)$ depending only on $||u_0||_\f$, $L_{\al_1}(\B_2)-f(\B_2)$ such that 
\be\h{461}
T_0\leq \g|A-B|, r(T_0)=x_0,
\ee
\be\h{462}
u(x,t)\in \begin{cases}
 [\al_1,\al_2] \mbox{ if } x<r(t),\\
 [\B_2,\B_1] \mbox{ if } x>r(t).
\end{cases}
\ee
\begin{proof}
Without loss of generality, let $\bar{u}_0(x)\in [\al_0,\al_2]$ for $x\in (A,B)$. Other case follows similarly. First assume that $u_0$ is piecewise constant with finite number of discontinuities. Define 
\begin{eqnarray*}
\xi_0&=&\min\left\{\frac{f(p)-f(q)}{p-q}: p\in [\al_0,\al_2], q\in [\B_2,\B_1]\right\},\\
\xi_1&=&\max\left\{\frac{f(p)-f(q)}{p-q}: p\in [\al_0,\al_2], q\in [\B_2,\B_1]\right\},\\
m_0&=&\min\left\{\frac{f(p)-f(q)}{p-q}: p,q \in [\al_0,\al_2]\right\},\\
m_1&=&\max\left\{\frac{f(p)-f(q)}{p-q}: p,q\in [\al_0,\al_2]\right\},\\
I&=& [-||u_0||_\f, ||u_0||_\f],\\
E&=& \{\al_1,\al_2,\B_1,\B_2,\al_0, \B_0, C,D\}\cup \{\mbox{jumps of } u_0\}. 
\end{eqnarray*}
Let $\{f_n\}$ be a sequence of piecewise affine functions such that (see Lemma \ref{lemma41})
\begin{itemize}
\item [i.] $f_n\rr f$ in $C^0_{\mbox{loc}}(\R).$
\item [ii.] For some $C>0$, for all $n$, $Lip (f_n, I)\leq C Lip(f,I)$.
\item [iii.] $E\subset$ corner points of $f_n$ (see definition \ref{definition41}), for all $n$.
\item [iv.] $(f_n,C,D)$ is a convex-concave type triplet with 
\begin{eqnarray*}
&&f_n(\T)>L_{\al_2,\B_2}(\T) \ \mbox{for}\ \T\in [C,D]\\
&& f_n(\B_2)<L_{\al_1}(\B_2)\\
&& f_n(\B_2)=L_{\al_0}(\B_2)=f(\al_0)+f^\p_{n_-}(\al_0)(\B_2-\al_0), \mbox{for all}\ n.
\end{eqnarray*}
\end{itemize}
Let $u_n$ be the solution of (\ref{11}), (\ref{12}) with flux $f_n$ and initial data $u_0$. Define 
\be
\eta_{n_1}&=&\frac{f_n(u_0(A-))-f_n(u_0(A+))}{u_0(A-)-u_0(A+)},\h{463}\\
\T_{n_1}&=&\frac{f_n(u_0(B-))-f_n(u_0(B+))}{u_0(B-)-u_0(B+)},\h{464}\\
l_{n_1}(t)&=&A+\eta_{n_1}t,\ L_{n_1}(t)=B+\T_{n_1}t\h{465}.
\ee
Now from (\ref{458}), $m_0>\xi_1$ and hence the lines $A+m_0t$ and $B+\xi_1 t$ meet at 
$\ti{T}_0>0$ given by 
\be\h{467}
\ti{T}_0=\frac{B-A}{m_0-\xi_1}.
\ee
Furthermore 
\be\h{466}
m_0\leq \eta_{n_1}=\frac{d l_{n_1}}{dt} \leq m_1, \xi_0\leq \T_{n_1}=\frac{d L_{n_1}}{dt}\leq \xi_1.
\ee
Let $T_{n_1}$ be the first point of interaction of waves of $u_n$. Then for $0<t<T_{n_1}$
and from Oleinik entropy condition  and (\ref{458}),
\begin{eqnarray*}
u_n(x,t)\in \begin{cases}
[\al_1,\al_2] \mbox{ if } x<l_{n_1}(t),\\
 [\al_0,\al_2] \mbox{ if } l_{n_1}(t)<x<L_{n_1}(t)\\
 [\B_2,\B_1] \mbox{ if } x>L_{n_1}(t).
\end{cases}\end{eqnarray*}
Suppose $l_{n_1}(T_{n_1})=A_{n_1}<B_{n_1}=L_{n_1}(T_{n_1}),$ then let $T_{n_2}>T_{n_1}$ be the second time of interaction of waves of $u_n$. Define
\begin{eqnarray*}
\eta_{n_2}&=&\frac{f_n(u_n(A_{n_1}-,T_{n_1}))-f_n(u_n(A_{n_1}+,T_{n_1}))}{u_n(A_{n_1}-, T_{n_1})-u_n(A_{n_1}+,T_{n_1})},\\
\T_{n_2}&=&\frac{f_n(u_n(B_{n_1}-,T_{n_1}))-f_n(u_n(B_{n_1}+,T_{n_1}))}{u_n(B_{n_1}-, T_{n_1})-u_n(B_{n_1}+,T_{n_1})},\\
l_{n_{2}}(t)&=& A_{n_1}+\eta_{n_2}(t-T_{n_1}),\ L_{n_2}(t)=B_{n_1}+\T_{n_2}(t-T_{n_1}),
\end{eqnarray*}
\begin{eqnarray*}
l_n(t)=\left\{\begin{array}{lll}
l_{n_1}(t) &\mbox{if}& t<T_{n_1},\\
l_{n_2}(t) &\mbox{if}& T_{n_1}\leq t \leq T_{n_2},
\end{array}\right.
L_n(t)=\left\{\begin{array}{lll}
L_{n_1}(t) &\mbox{if}& t<T_{n_1},\\
L_{n_2}(t) &\mbox{if}& T_{n_1}\leq t \leq T_{n_2}.
\end{array}\right.
\end{eqnarray*}
Then for $0<t<T_{n_2}$,
\begin{eqnarray}\h{468}
u_n(x,t)\in \begin{cases}
[\al_1,\al_2] \mbox{ if } x<l_{n}(t),\\
 [\al_0,\al_2] \mbox{ if } l_{n}(t)<x<L_{n}(t),\\
 [\B_2,\B_1] \mbox{ if } x>L_{n}(t).
\end{cases}
\end{eqnarray}
\be\h{469}
m_0\leq \frac{dl_{n}}{dt}\leq m_1, \ \xi_0\leq \frac{d L_n}{dt}\leq \xi_1,\ l_n(0)=A, L_n(0)=B.
\ee
If $A_{n_2}=l_n(T_{n_2})<B_{n_2}=L_n(T_{n_2}),$ then repeat the above to obtain the curves $l_n, L_n$ satisfying (\ref{468}) and (\ref{469}) for $t\leq T_{n_k}$, where $T_{n_{k}}$ is the $k$ th time interaction of waves of $u_n$. From (\ref{469}) and (\ref{467}), $l_n$ and $L_n$ meet at $T_{n_{k_0}}\leq \ti{T}_0$ with $x_n=l_n(T_{n_{k_0}})=L_n(T_{n_{k_0}}).$ Then 
\begin{eqnarray*}
u_n(x,T_{n_{k_0}})\in \begin{cases}
 [\al_1,\al_2] \mbox{ if } x<x_n,\\
 [\B_2,\B_1] \mbox{ if } x>x_n.
\end{cases}
\end{eqnarray*}
and $\{x_n\}$ is bounded from (\ref{469}). Furthermore for $t>T_{n_{k_0}},$
\begin{eqnarray}\h{470}
u_n(x,t)\in \begin{cases}
 [\al_1,\al_2] \mbox{ if } x<L_n(t),\\
 [\B_2,\B_1] \mbox{ if } x>L_n(t).
\end{cases}
\end{eqnarray}
Hence from Lemma \ref{lemma21} and Arzela Ascoli's Theorem, for a subsequence $x_n\rr x_0, T_{n_{k_0}}\rr T_0,$ $u_n\rr u$ in $L^1_{\mbox{loc}}(\R\times(0,\f)).$
$L_n\rr r(\cdot)$ as $n\rr \f$. Since $T_0\leq \ti{T}_0,$ hence with $\g=\frac{1}{m_0-\xi_1}$, and from (\ref{470}), (\ref{461}) and (\ref{462}) follows.  If $u_0$ is arbitrary, approximate $u_0$ by piecewise constant functions and from $L^1$ contraction, the Lemma follows, since $\g, x_0, T_0$ depends only $||u_0||_\f$ and $L_{\al_1}(\B_2)-f(\B_2)$. This proves the Lemma.
\end{proof}
\end{lemma}
\begin{lemma}\h{lemma48}
Let $\al_2\leq \ti{\al}<C\leq D<\ti{\B}\leq \B_2$ such that 
\be\h{470}
\begin{array}{lll}
f(\T)>L_{\ti{\al},\ti{\B}}(\T), \ \mbox{for }\T\in[C,D],\\
f(\ti{\B})<L_{\al_1}(\ti{\B}).
\end{array}
\ee
Assume $\bar{u}_0$ satisfies one of the following hypothesis 
\begin{itemize}
\item [i.] range of $\bar{u}_0\subset [\al_1,\ti{\al}]$
\item [ii.] range of $\bar{u}_0\subset [\ti{\B},\B_1]$
\end{itemize}
then there exist $(x_0,T_0)\in \R\times(0,\f), \g>0$ and a Lipschitz curve $r(\cdot)$ in $[T_0,\f)$ depending only on $||u_0||_\f, f^\p(\al_0)-f^\p(\al_1)$ such that for $t>T_0$
\begin{eqnarray}
T_0\leq \g |A-B|, r(T_0)=x_0, \h{471}
\end{eqnarray}
\begin{eqnarray}\h{472}
u(x,t)\in \begin{cases}
[\al_1,\al_2] \mbox{ if } x<r(t),\\
[\B_2, \B_1] \mbox{ if } x>r(t).
\end{cases}
\end{eqnarray}
\begin{proof}
We can assume that $\bar{u}_0\in [\al_1,\ti{\al}]$ for all $x\in (A,B)$. Other case follows in a similar way. Again we use front tracking. Let $u_0$ is piecewise constant with finite number of jumps. Define 
\begin{eqnarray*}
E&=&\{\al_1, \al_2, \B_1, \B_2, \ti{\al}, C, D\}\cup \{\mbox{Jumps of } u_0\},\\
I&=& [-||u_0||_\f,||u_0||_\f],\\
\xi_0&=& \min\left\{\frac{f(p)-f(q)}{p-q},\ p\in [\al_1,\ti{\al}], q\in [\B_1,\B_2]\right\},\\
\xi_1&=& \max\left\{\frac{f(p)-f(q)}{p-q},\ p\in [\al_1,\ti{\al}], q\in [\B_1,\B_2]\right\},\\
m_0&=&f^\p(\al_1), m_1=f^\p(\ti{\al}).
\end{eqnarray*}
Let $\{f_n\}$ be a sequence of piecewise affine functions such that  (see Lemma \ref{lemma41})

\begin{itemize}
\item [i.] $f_n\rr f$ in $C^0_{\mbox{loc}}$ with $f^{\p}_{n,-}(\al_1)=f^\p(\al_1)$,$\ Lip (f_n, I)\leq C Lip (f,I)$, for some constant $C>0$.
\item [ii.] For all $n$, $E\subset\mbox{ corner points of }f_n$.
\item [iii.] $(f_n,C,D)$ is a convex-concave type triplet with 
$$f_n(\T)>L_{\ti{\al},\ti{\B}}(\T), \forall\ \T\in [C,D].$$
\end{itemize}
Let $u_n$ be the solution to (\ref{11}), (\ref{12}) with flux $f_n$ and initial data $u_0$. Define $\eta_{n_1}, \T_{n_1}, l_{n_1}, L_{n_2}$ as in (\ref{462}), (\ref{463}) and (\ref{464}). Since $m_0>\xi_1$, let $\ti{T}_0$ 
 be the intersection point of $A+m_0t$ and $B+\xi_1t$. Furthermore (\ref{467}) holds.
\par Let $T_{n_1}$ be the first time of interaction of waves of $u_ n$, then for $0<t<T_{n_1}$, we have 
\begin{eqnarray*}
u(x,t)\in \begin{cases}
[\al_1,\al_2] \mbox{ if } x<l_{n_1}(t),\\
[\al_1, \ti{\al}] \mbox{ \ if } l_{n_1}(t)<x<L_{n_1}(t),\\
[\B_1,\B_2] \mbox{ if } x>L_{n_1}(t).
\end{cases}
\end{eqnarray*}
As in \ref{lemma47}, suppose $l_{n_1}(T_{n_1})=A_{n_1}<B_{n_1}=L_{n_1}(T_{n_1}),$ then by induction there exists a $T_{n_k}$ and a piecewise affine curves $l_n(t), L_n(t)$ for $0<t<T_{n_k}$
such that $l_n$ and $L_n$ satisfies (\ref{469}) and 
\begin{eqnarray*}
u(x,t)\in \begin{cases}
[\al_1,\al_2] \mbox{ if } x<l_{n}(t),\\
[\al_1, \ti{\al}] \mbox{ \ if } l_{n}(t)<x<L_{n}(t),\\
[\B_1,\B_2] \mbox{ if } x>L_{n}(t).
\end{cases}
\end{eqnarray*}
From (\ref{469}) and (\ref{465}), there exists $k_0$ such that $x_n=l_n(T_{n_{k_0}})=L_n(T_{n_{k_0}})$ and $T_{n_{k_0}}\leq \ti{T}_0=\frac{B-A}{m_0-\xi_1}$. Hence from $L^1$-contraction and Arzela-Ascoli Theorem, for a subsequence, still denoted by $n$ such that $u_n\rr u$ in $L^1_{\mbox{loc}}$, $x_n\rr x_0$, $T_{n_{k_0}}\rr T_0, l_n(t)\rr r(t)$, a Lipschitz curve with $r(T)=x_0$, $T_0\leq \g |B-A|,$
\begin{eqnarray*}
u(x,t)\in \begin{cases}
[\al_1,\al_2] \mbox{ if } x<r(t),\\
[\B_1,\B_2] \mbox{ if } x>r(t).
\end{cases}
\end{eqnarray*}
This proves the Lemma.
\end{proof}
\end{lemma}
\noindent{\it Convex modification of $f$:} Define 
\be\h{322}
\ti{f}(p)=\left\{\begin{array}{lll}
f(p) &\mbox{if}& p\leq C,\\
(p-C)^2+f^\p(C)(p-C)+f(C) &\mbox{if}& p>C.
\end{array}\right.
\ee
Then $\ti{f}$ is a $C^1$ convex function with $\ti{f}=f$ in $(-\f,C]$ and 
\be\h{323}
\lim\limits_{|p|\rr \f}\frac{\ti{f}(p)}{|p|}=\f.
\ee
\begin{lemma}\h{lemma31}
Let $\al_0$ be as in (\ref{458}), $m\leq \al_0$ and assume that $\bar{u}_0$ satisfies 
\be\h{325}
\bar{u}_0(x)=m,
\ee
then there exist $(x_0,T_0)\in \R\times(0,\f), \g>0$ and a Lipschitz $r(\cdot)$ depending only on $||u_0||_\f$ and hypothesis (\ref{123}), (\ref{126}) such that 
\be\h{326}
r(T_0)=x_0, T_0\leq \g|A-B|, \h{326}\ee
\be\h{327}
u(x,t)\in \begin{cases}
 [\al_1,\al_2] \mbox{ if } x<r(t),\\
 [\B_2,\B_1] \mbox{ if } x>r(t).
\end{cases}
\ee
\begin{proof}
Proof is lengthy and we use the structure Theorem. Let $g(p)=\ti{f}(p)$ be the convex modification of $f$ as in (\ref{322}).
\par First assume that $u_0$ is piecewise constant with finite number of discontinuities. Let 
$$E=\{\al_1,\al_2,\B_1,\B_2,C,D\}\cup\{\mbox{points of discontinuities of}\ u_0\}.$$
Let $\{f_n\}$ and $\{g_n\}$ be sequences of continuous piecewise affine functions such that $g_n$ is convex satisfying  (see Lemma \ref{lemma41})

\begin{itemize}
\item [i.] $(f_n,g_n)\rr (f,g)$ in $C^0_{\mbox{loc}}(\R)$ as $n\rr\f.$
\item [ii.] $f_n=g_n$ in $(-\f,C]$.
\item [iii.] $\displaystyle\lim\limits_{|p|\rr \f}\inf\limits_{n}\frac{g_n(p)}{|p|}=\f.$
\item [iv.] For all $x$, $E\subset \{\mbox{corner points of }\ f_n\}$.
\item [v.] $\al_0,m$ lies in the interior of degenerate points of $f_n$ for all $n$.  
\item [vi.] $f_{n}(\T)<L_{\al_2, \B_2}(\T)$ for $\T\in[C,D]$.
\item [vii.] $f_n(\B_2)=L_{\al_0}(\B_2).$
\item [viii.] $I=[-||u_0||_\f,||u_0||_\f]$ and for all $n$,
$$\max(Lip(f_n,I), Lip(g_n,I))\leq 2 \max(Lip(f,I), Lip(g,I))=J.$$
\end{itemize}
Let 
\begin{eqnarray*}
w_0=\left\{\begin{array}{lll}
u_0(x) &\mbox{if}& x<B,\\
\al_0 &\mbox{if}& x>B,
 \end{array}\right.
\end{eqnarray*}
and $u_n,w_n,w$ be the solutions of (\ref{11}), (\ref{12}) with respective  fluxes $f_n,g_n,g$ and initial data $u_0,w_0,w_0$.
Define
\begin{eqnarray*}
\xi_0^n&=&\min\left\{\frac{f_n(p)-f_n(q)}{p-q}: p\in [m,\al_0], q\in [\al_1,\al_2]\right\}.\\
\xi_1^n&=&\max\left\{\frac{f_n(p)-f_n(q)}{p-q}: p\in [m,\al_0], q\in [\al_1,\al_2]\right\}.\\
\eta_{n_1}&=& \frac{g_n(u_0(A-))-g_n(m)}{u_0(A-)-m}.\\
l_n(t)&=&B+\eta_{n_1}t, \  r(t)=B+f^\p(m)t, \ R(t)=B+f^\p(\al_0)t.
\end{eqnarray*}
Then from (viii) we have for $i=0,1$, $|\xi_i^n|,  |\eta_{n_1}|$ are bounded by $J$ and from (v) to (vii), 
$f^\p(m)<\xi_0=\inf\limits_n \xi_0^n\leq \eta_{n_1}.$ Hence the line $r(\cdot)$ and $A+\xi_0t$ meet at $\ti{T}>0$ given by 
\be\h{329}
\ti{T}=\frac{B-A}{\xi_0-f^\p(m)}.
\ee
 Let $T_{n_1}$ be the first time of interaction of waves of $w_n$. Let $y_\pm$ be the extreme characteristic points with respect to $w_0$ and $g_n$. Then for $0<t<T_{n_1}$, $w_n$ satisfies 
 \begin{eqnarray}\h{330}
 \left\{\begin{array}{lll}
 w_n(x,t)\in [\al_1,\al_2] \ \mbox{if}\ x<l_n(t),\\
 w_n(x,t) = u_n(x,t) \ \mbox{if}\ x<r(t),\\
 w_n(x,t)=\left\{\begin{array}{llll}
 m \ \mbox{if}\ l_n(t)<x<r(t),\\
 \mbox{rarefaction from}\ m \ \mbox{to}\ \al_0\ \mbox{if}\ r(t)<x<R(t),\\
 \al_0 \ \mbox{if}\ x>R(t),
 \end{array}\right.
 \end{array}\right.
 \end{eqnarray}
From (\ref{441}) to (\ref{454})
\be\h{331}
\left\{\begin{array}{lllll}
y_-(R(t),t)=A=y_+(R(t),t),\\
y_-(l_n(t),t)\leq A\leq y_+(l_n(t),t),\\
y_-(r(t),t)\leq B \leq y_+(r(t),t),\\
\displaystyle\left|\frac{d l_n}{dt}\right|\leq J, \displaystyle\ \frac{d l_n}{dt}\geq \xi_0>f^\p(m).
\end{array}\right.
\ee
Let $(A_{n_1}, B_{n_1})=(l_{n_1}(T_{n_1}),r(T_{n_1})).$ Clearly $A_{n_1}\leq B_{n_1}$. If $A_{n_1}<B_{n_1}$, again starting the front tracking from $T_{n_1}$, let $T_{n_2}>T_{n_1}$ be the first time of interaction of waves for $w_n$. Let 
\be
&&\eta_{n_2}=\frac{g(u(A_{n_1}-,T_{n_1}))-g_n(m)}{u(A_{n_1}-,T_{n_1})-m},\h{332}\\
&&l_n(t)=A_{n_1}+\eta_{n_2}(t-T_{n_2})\ \mbox{for}\ T_{n_1}\leq t \leq T_{n_2}.
\ee
Then $w_n$ satisfies (\ref{330}) and for  $T_{n_1}\leq t \leq T_{n_2},$
\be\h{333}
\left\{\begin{array}{llll}
y_-(l_n(t),T_{n_1},t)\leq A_{n_1}\leq y_+(l_n(t), T_{n_1},t),\\
y_-(r(t),T_{n_1},t)\leq B_{n_1}\leq y_+(l_n(t), T_{n_1},t),\\
f^\p(m)<\xi_0\leq \frac{d l_n}{dt}.
\end{array}\right.
\ee 
Hence from (\ref{219}) for $T_{n_1}\leq T_{n_2}$, 
\be\h{334}
y_-(l_n(t),t)\leq A \leq y_+(l_n(t),t),\\
y_-(r(t),t)\leq B \leq y_+(r(t),t),\\
y_-(R(t),t)=B=y_-(R(t),t),
\ee
and from (\ref{333}), $l_n(t)\geq A+t\xi_0$ for $t\in [0,T_{n_2}],$ continuing the front tracking one can get $T_{n_k}>T_{n_{k-1}}>\cdots>T_{n_1}$ and $A_{n_k}<B_{n_k}$, $l_n(\cdot)$ such that $w_n$ satisfies (\ref{330}), (\ref{331}) for $0<t<T_{n_k}$. Since $l_n(t)\geq A+t\xi_0$, from (\ref{329}), there exists a $k_0$ such that $x_n=A_{n_{k_0}}=B_{n_{k_0}}$ and $l_n(T_{n_k})=r(T_{n_{k_0}}),$ $T_{n_{k_0}}\leq \ti{T}=\frac{B-A}{\xi_0-f^\p(m)}$. Let $\{m<v_0^n<v_1^n<\cdots<v_p^n<\al_0\}$ be the corner points of $f_n$ in $[m,\al_0].$ Then $w_n$ satisfies 
\begin{eqnarray*}
\left\{\begin{array}{llll}
w_n(x,T_{n_{k_0}})&\in& [\al_1,\al_2], \ \mbox{if}\ x<x_n,\\
w_n(x,T_{n_{k_0}})&=&\left\{\begin{array}{llll}
\al_0 \ \mbox{if}\ x>R(T_{n_{k_0}}),\\
\mbox{rarefaction from}\ v_1^n \ \mbox{to}\ v_k^n\ \mbox{if}\ x_n<x<R(T_{n_{k_0}}).
\end{array}\right.\end{array}\right.
\end{eqnarray*}
Hence from the front tracking Lemma and (\ref{219}), we can extend the function $l_n(t)$ to a maximal time $T_{n_{k_0}}\leq \ti{T}_n\leq \f$ such that 
\be\h{335}
\left\{\begin{array}{lll}
l_n(t)\leq R(t) &\mbox{for}& t<\ti{T}_n,\\
l_n(\ti{T}_n)=R(\ti{T}_n) &\mbox{if}& \ti{T}_n<\f.
\end{array}\right.
\ee
\begin{eqnarray}\h{336}
\left\{\begin{array}{llll}
w_n(x,t)&\in& [\al_1,\al_2], \ \mbox{if}\ x<l_n(t),\\
w_n(x,t)&=&\left\{\begin{array}{llll}
\al_0 \ \mbox{if}\ x>R(t),\\
\mbox{rarefaction from}\ v_s^n \ \mbox{to}\ v_k^n\ \mbox{if}\ l_n(t)<x<R(t),
\end{array}\right.\end{array}\right.
\end{eqnarray}
where $s$ depends on $t$ and is a non decreasing function of $t$, with
\be\h{337}
y_-(l_n(t),t)\leq A \leq y_+(l_n(t),t),\\
y_-(R(t),t)\leq B \leq y_+(R(t),t),\\
\left|\frac{d l_n}{dt}\right| \leq J.
\ee
Hence for a subsequence, let $\ti{T}_n\rr T_0, l_n\rr l$ in $C_{\mbox{loc}}^0(\R)$, $w_n\rr w$ in $L^1_{\mbox{loc}}(\R^n\times(0,\f))$, where $T_0\in [0,\f]$. Then from (\ref{335}), for $\lim\limits_{n\rr \f} T_{n_{k_0}}\leq t <T_0$, $w$ satisfies 
\begin{eqnarray*}
w(x,t)\in [\al_1,\al_2] &\mbox{if}& x<l(t),\\
w(x,t)=\al_0 &\mbox{if}& x>R(t).
\end{eqnarray*}
From (5), (6), (7) of Lemma \ref{lemma23} and from (\ref{336}), 
\be\h{338}
y_-(l(t),t)\leq A \leq y_+(l(t),t),\\
y_-(R(t),t)\leq B \leq y_+(R(t),t),
\ee
where $y_\pm$ are the extreme characteristic points  corresponds to $g$ and $w_0$. Hence by definition 
\be\h{339}
\left\{\begin{array}{lll}
R_-(t,A)\leq l(t) \leq R_+(t,A),\\
R_-(t,B)\leq R(t) \leq R_+(t,B).
\end{array}\right.
\ee
Since $f^\p(\al_0)<f^\p(\al_1)$ for any $\al\in [\al_1,\al_2]$,  from (1) of structure Theorem,
 $R_-(\cdot, A)$ and $R_+(\cdot, B)$ meet at $(\ti{x}_0,\ti{T}_0)$ with $\ti{x}_0=R_-(\ti{T}_0, A)=R_+(\ti{T}_0, B)$. 
\be\h{340}
\begin{array}{llll}
&\mbox{Hence } l(\cdot) \mbox{ and } R(\cdot) \mbox{ meets at }\\
 &T_0\leq \ti{T}_0\leq \g |A-B| \mbox{ and for } x_0=l(T_0)=R(T_0),\\
&w(x,\ti{T}_0)\in\left\{\begin{array}{lll}
\ [\al_1,\al_2] \ \mbox{if}\ x<x_0,\\
\ \al_0 \ \mbox{if} \ x>x_0.
\end{array}\right.
\end{array}
\ee
Therefore from (\ref{335}), choose $n_0>0$ such that for $n\geq n_0$, 
\be\h{341}
\ti{T}_n\leq (\ti{\g}+1)|A-B|.
\ee
Now coming  back to $u_n$, let 
\begin{eqnarray*}
\T_{n}&=&\frac{f_n(u_0(B+))-f_n(m)}{u_0(B+)-m},\\
L_n(t)&=&B+\T_{n}t.
\end{eqnarray*}
Then from (v) to (viii), 
$$\T_{n}\leq f^\p( \al_0),\ |\T_{n_1}|\leq J.$$
Hence $L_{n}(t)\leq R(t).$ Let $T_1^n$ be the first time of interaction of the waves of $u_n$. Let $l_n$ be as in the previous case. Then for $0<t<T^n_1$,
\begin{eqnarray*}
\left\{\begin{array}{llll}
u_n(x,t)&\in& [\al_1,\al_2], \ \mbox{ if } x<l_n(t),\\
u_n(x,t)&\in& [\B_2,\B_1],\ \mbox{ if } x>L_n(t),\\
u_n(x,t)&=&\left\{\begin{array}{llll}
m \ \mbox{if}\ l_n(t)<x<r(t),\\
\mbox{rarefaction from}\ m \ \mbox{to}\ v_s^n\ \mbox{for some}\ s\ \mbox{depending on}\ t \\ \mbox{if}\ r(t)<x<L_n(t),
\end{array}\right.\\
\displaystyle\left|\frac{d L_n}{dt}\right|&\leq& J,\ L_n(t)\leq R(t).
\end{array}\right.
\end{eqnarray*}
From the front tracking Lemma and (\ref{341}), this process can be continued till a time $T_0^n\leq \ti{T}_n\leq (\ti{\g}+1)|A-B|$
\begin{eqnarray*}
x_n=L_n(T_0^n)=l_n(T_0^n),\\
R(t)\geq L_n(t)>l_n(t) \ \mbox{for}\ t<T_0^n.
\end{eqnarray*}
Then $u_n$ satisfies 
\begin{eqnarray*}
u_n(x,T_0^n)\in\begin{cases}
 [\al_1,\al_2] \mbox{ if } x<x_n,\\
 [\B_2,\B_1] \mbox { if } x>x_n.
\end{cases}
\end{eqnarray*}
Now letting a subsequence $n\rr \f$ and from Lemma \ref{lemma21} with $\g=\ti{\g}+1$, $T_0^n\rr T_0$, $x_n\rr x_0$, $T_0$ satisfies (\ref{326}) and 
\begin{eqnarray*}
u(x,T_0)\in\begin{cases}
 [\al_1,\al_2] \mbox{ if } x<x_0,\\
 [\B_2,\B_1] \mbox{ if } x>x_0.
\end{cases}
\end{eqnarray*}
Now from Lemma (\ref{lemma47}), $r(\cdot)$ exist satisfying (\ref{326}) and (\ref{327}).
\par For  a general $u_0$, approximate $u_0$ by piecewise constant function in $L^1_{\mbox{loc}}$ norm and by $L^1_{\mbox{loc}}$ contraction, (\ref{326}) and (\ref{327}) follows. This proves the Lemma. 
\end{proof}
 \end{lemma}

\begin{proof}[Proof of the Theorem for convex-concave type]
Let $$m=\min\{\al_1,\inf \bar{u}_0\},\ M=\max\{\B_1, \sup \bar{u}_0\}.$$
\begin{eqnarray*}
u_{0,m}(x)=\left\{\begin{array}{llll}
u_0(x) &\mbox{if}& x\notin(A,B),\\
m &\mbox{if}& x\in(A,B),
\end{array}\right.
\end{eqnarray*}
\begin{eqnarray*}
u_{0,M}(x)=\left\{\begin{array}{llll}
u_0(x) &\mbox{if}& x\notin(A,B),\\
M &\mbox{if}& x\in(A,B),
\end{array}\right.
\end{eqnarray*}
and $u_m, u_M$ be the solutions of (\ref{11}), (\ref{12}) with respective initial data $u_{0,m}$ and $u_{0,M}$. Then from Lemma \ref{lemma31}, Lemma \ref{lemma47}, Lemma \ref{lemma48} there exist $\g>0$, $T_1\leq \g|A-B|$, $A_1, B_1$ such that 
\begin{eqnarray*}
u_m(x,T_1)\in\begin{cases}
[\al_1,\al_2] \mbox{ if } x<A_1,\\
 [\B_2,\B_1] \mbox{ if } x>A_1.
\end{cases}
\end{eqnarray*}
\begin{eqnarray*}
u_M(x,T_1)\in\begin{cases}
 [\al_1,\al_2] \mbox{ if } x<B_1,\\
[\B_2,\B_1] \mbox{ if } x>B_1.
\end{cases}
\end{eqnarray*}
Since $u_{0,m}\leq u_0\leq u_{0,M}$, thus $u_m\leq u \leq u_M$. Therefore $A_1\leq B_1$ and 
\begin{eqnarray}\h{1327}
u(x,T_1)\in\begin{cases}
 [\al_1,\al_2] \mbox{ if } x<A_1,\\
 [\al_1,\B_1] \mbox{ if } A_1<x<B_1,\\
 [\B_2,\B_1] \mbox{ if } x>B_1.
\end{cases}
\end{eqnarray}
Therefore, without loss of generality we can assume that 
\be\h{1328}
\bar{u}_0(x)\in [\al_1,\B_1] \ \mbox{for}\ x\in (A,B).
\ee
From (\ref{118}), (\ref{125}), (\ref{126}), choose $\e>0$ such that 
\be\h{1329}
\left\{\begin{array}{lllllll}
(\B_1-\B_2)+(\al_2-\al_1)<\e,\\
f(\T)>L_{\al_1+\e, \B_1-\e}(\T)\ \mbox{for}\ \T\in [C,D],\\
L_{\al_1}(\B_1-\e)>f(\B_1-\e).
\end{array}\right.
\ee
Define 
\begin{eqnarray*}
u_{0,1}=\left\{\begin{array}{llll}
u_0(x) &\mbox{if}& x\notin(A,B), \ \mbox{or}\ u_0(x)\leq \al_1+\frac{\e}{2},\\
\al_1+\frac{\e}{2} &\mbox{if}& x\in(A,B), \ \mbox{and}\ u_0(x)\geq \al_1+\frac{\e}{2},\\
\end{array}\right.
\end{eqnarray*}
\begin{eqnarray*}
u_{0,2}=\left\{\begin{array}{llll}
u_0(x) &\mbox{if}& x\notin(A,B), \ \mbox{or}\ u_0(x)> \B_1-\frac{\e}{2},\\
\B_1-\frac{\e}{2} &\mbox{if}& x\in(A,B), \ \mbox{and}\ u_0(x)\leq \B_1-\frac{\e}{2}.\\
\end{array}\right.
\end{eqnarray*}
Let $u_1$ and $u_2$ be the solutions of (\ref{11}), (\ref{12}) with respective initial data $u_{0,1}$ and $u_{0,2}$. Then from (\ref{1329}), Lemma \ref{lemma48}, there exist $\g>0, T_1>0, A_1, B_1$ such that 
$$T_1\leq \g|A-B|,$$
\begin{eqnarray}\h{1330}
u_1(x,T_1)\in\begin{cases}
 [\al_1,\al_2] \mbox{ if } x<A_1,\\
 [\B_2,\B_1] \mbox{ if } x>A_1.
\end{cases}
\end{eqnarray}
\begin{eqnarray}\h{1331}
u_2(x,T_1)\in\begin{cases}
 [\al_1,\al_2] \mbox{ if } x<B_1,\\
 [\B_2,\B_1] \mbox{ if } x>B_1.
\end{cases}
\end{eqnarray}
Since $u_{0,1}\leq u_0\leq u_{0,2}$, consequently  $u_1\leq u \leq u_2$ and therefore from (\ref{1330}), (\ref{1331}), $A_1\leq B_1$ and 
\begin{eqnarray*}
u(x,T_1)\in\begin{cases}
 [\al_1,\al_2] \mbox{ if } x<A_1,\\
 [\al_1,\B_1] \mbox{ if } A_1<x<B_1,\\
 [\B_2,\B_1] \mbox{ if } x>B_1.
\end{cases}
\end{eqnarray*}
Furthermore, from $L^1$ contraction, 
\begin{eqnarray*}
(\B_2-\al_2)|A_1-B_1|&\leq& \int\limits_{A_1}^{B_1}|u_1(x,T_2)-u_2(x,T_2)|dx\\
&\leq& \int\limits_A^B |u_{0,1}(x)-u_{0,2}(x)|dx\\
&\leq & (\B_1-\al_1-\e)|A-B|.
\end{eqnarray*}
This gives $$|A_1-B_1|\leq \frac{(\B_1-\al_1-\e)}{(\B_2-\al_2)}|A-B|=\delta |A-B|,$$
where $\delta=\frac{(\B_1-\al_1-\e)}{(\B_2-\al_2)}<1$ from (\ref{329}).
\par Choose $$\al_1\leq \al_2<\xi_1<C<D<\xi_2<\B_2\leq \B_1,$$ 
such that 
\begin{eqnarray*}
\left\{\begin{array}{lll}
f(\T)>\max\{L_{\al_1,\xi_2}(\T), L_{\xi_1,\B_1}(\T)\} \ \mbox{for}\ \T\in [C,D],\\
f(\xi_2)<L_{\al_1}(\xi_2), \end{array}\right.
\end{eqnarray*}
and define 
\begin{eqnarray*}
w_{0,1}=\left\{\begin{array}{llll}
\al_1 &\mbox{if}& x<B,\\
\xi_2 &\mbox{if}& x>B.
\end{array}\right.
\end{eqnarray*}
\begin{eqnarray*}
w_{0,2}=\left\{\begin{array}{llll}
\xi_1 &\mbox{if}& x<A,\\
\B_1 &\mbox{if}& x>A.
\end{array}\right.
\end{eqnarray*}
Let $w_1,w_2$ be the solutions  of (\ref{11}), (\ref{12}) with respective initial data are given by 
\begin{eqnarray*}
w_1(x,t)=\left\{\begin{array}{lll}
\al_1 &\mbox{if}& x<B+\displaystyle\frac{f(\al_1)-f(\xi_2)}{\al_1-\xi_2}t=\rho_1(t),\\
\xi_1 &\mbox{if}& x>\rho_1(t),
\end{array}\right.
\end{eqnarray*}
\begin{eqnarray*}
w_2(x,t)=\left\{\begin{array}{lll}
\xi_2 &\mbox{if}& x<A+\displaystyle\frac{f(\xi_1)-f(\B_2)}{\xi_1-\B_2}t=\rho_2(t),\\
\B_1 &\mbox{if}& x>\rho_2(t).
\end{array}\right.
\end{eqnarray*}
Since $w_{0,1} \leq u_{0,1} \leq u_0 \leq u_{0,2} \leq w_{0,2}$, thus $w_1\leq u_1 \leq  u \leq u_2 \leq  w_2$. Hence  $\rho_2(T_1)\leq A_1\leq B_1\leq \rho_1(T_2)$.
\par By induction we can find $T_n>T_{n-1}>\cdots >T_1>0$, $\rho_2(T_n)\leq A_n\leq B_n\leq \rho_1(T_n)$ such that 
\begin{eqnarray*}
|A_n-B_n|&\leq&  \delta |A_{n-1}-B_{n-1}|\leq \delta^n|A-B|\\
T_n&\leq & T_{n-1}+\g|A_{n-1}-B_{n-1}|\\
&\leq& \g(1+\delta+\delta^2+\cdots+ \delta^{n-1})|A-B|,
\end{eqnarray*}
\begin{eqnarray*}
u(x,T_n)\in\begin{cases}
 [\al_1,\al_2] \mbox{ if } x<A_n,\\
 [\al_1,\B_1] \mbox{ if } A_n<x<B_n,\\
 [\B_2,\B_1] \mbox{ if } x>B_n.
\end{cases}
\end{eqnarray*}
Hence  $\{T_n\}$ converges to $T_0$, $x_0=\lim A_n =lim B_n$, 
$$T_0\leq \frac{\g}{1-\delta}|A-B|$$
and 
\begin{eqnarray*}
u(x,T_0)\in\begin{cases}
[\al_1,\al_2] \mbox{ if } x<x_0,\\
 [\B_2,\B_1] \mbox{ if } x>x_0.
\end{cases}
\end{eqnarray*}
Hence from Lemma \ref{lemma47}, there exist a Lipschitz curve $r(\cdot)$ with $r(T_0)=x_0$ such that for $t>T_0$
\begin{eqnarray*}
u(x,T_0)\in\begin{cases}
[\al_1,\al_2] \mbox{ if } x<r(t),\\
 [\B_2,\B_1] \mbox{ if } x>r(t).
\end{cases}
\end{eqnarray*}
This proves the Theorem.
\end{proof}

\noindent\textbf{Counter examples:}
\begin{itemize}
\item [1.] Let $f$ be a super linear function with two inflection points $C<D$. Then $(f,C,D)$ is a convex-convex type triplet. Let $x_0\in (C,D)$ such that $f(x_0)=\max\{f(\T): \T\in [C,D]\}.$ Let 
$\al_0<C<x_0<D<\B_0$ such that 
$f(\al_0)=f(x_0)=f(\B_0)$ and 
\begin{eqnarray*}
u_0(x)=\left\{\begin{array}{lll}
\B_0 &\mbox{if}& x<A,\\
x_0 &\mbox{if}& x\in(A,B),\\
\al_0 &\mbox{if}& x>B.\end{array}\right.
\end{eqnarray*}
Then the solution $u$ to (\ref{11}), (\ref{12}) is given by $u(x,t)=u_0(x)$ for all $(x,t)\in \R\times(0,\f)$, which is not a single shock solution.
\item [2.] Let $f$ be a superlinear function with one inflection point $x_0$. Let $C<x_0<D$, then $(f,C,D)$ is a convex-concave flux triplet. Let $\al_1<\al_2<C<D<\B$ is such that 
$$f(\B)=L_{\al_2}(\B), f^\p(\al_1)<f^\p(\al_2).$$
Let 
\begin{eqnarray*}
u_0(x)=\left\{\begin{array}{lll}
\al_1 &\mbox{if}& x<A,\\
\al_2 &\mbox{if}& x\in(A,B),\\
\B &\mbox{if}& x>B.\end{array}\right.
\end{eqnarray*}
Let $$l_1(t)=A+f^\p(\al_1)t,\ l_2(t)=A+f^\p(\al_2)t, l_3(t)=B+f^\p(\al_2)t.$$
Then the solution $u$ of (\ref{11}), (\ref{12}) is given by 
\begin{eqnarray*}
u(x,t)=\left\{\begin{array}{llll}
\al_1 &\mbox{if}& x<l_1(t),\\
(f^\p)^{-1}\left(\frac{x-A}{t}\right) &\mbox{if}& l_1(t)<x<l_2(t),\\
\al_2 &\mbox{if}& l_2(t)<x<l_3(t),\\
\B &\mbox{if}& x>l_3(t),
\end{array}\right.
\end{eqnarray*}
which is not a shock solution.
\end{itemize}
\section{Appendix}
\setcounter{equation}{0}
For the sake of completeness, we will prove some of the Lemmas stated earlier.
\begin{proof}[Proof of Lemma \ref{lemma21}] First assume that $u_0\in BV(\R).$ Then for any $0\leq s <t,$ $TV(u_k(\cdot, t))\leq TV(u_0(\cdot))$ and 
$$\int\limits_{\R}|u_k(x,s)-u_k(x,t)|dx \leq C |s-t|TV(u_0).$$
Hence from Helly's Theorem, there exists a subsequence still denoted by $\{u_k\}$ such that $u_k\rr u$ in $L^1(\R\times[0,T])$ for any $T>0$.
\par Let $u_0\in L^\f(\R)$ and $u_{0,n}\in BV (\R)$ such that $u_{0,n}\rr u_0$ in $L^1_{\mbox{loc}}(\R)$. Let $u_k^n$ be the solution of (\ref{11}) with flux $f_k$ and initial data $u_{0,n}$. Let $m=\sup\limits_n||u_{0,n}||_\f$ and $K=[-m,m]$, $M=\sup\limits_{K}Lip(f_k,K)$. Then for $t>0$ and $L^1_{\mbox{loc}}$ contraction, we have for $T>0$, $L>0$,
\begin{eqnarray*}
\int\limits_{-L}^{L}|u_k(x,t)-u_m(x,t)|dx &\leq& \int\limits_{-L}^{L}|u_k(x,t)-u_k^n(x,t)|dx+
\int\limits_{-L}^{L}|u_k^n(x,t)-u_m^n(x,t)|dx\\
&+& \int\limits_{-L}^{L}|u_m^n(x,t)-u_m(x,t)|dx\\
&\leq & 2\int\limits_{-L-Mt}^{L+Mt}|u_0(x)-u_{0,n}(x)|dx + \int\limits_{-L}^{L}|u_k^n(x,t)-u_m^n(x,t)|dx,
\end{eqnarray*}
  hence 
  \begin{eqnarray*}
\int\limits_{0}^T  \int\limits_{-L}^{L}|u_k(x,t)-u_m(x,t)|dxdt &\leq& 2T\int\limits_{-L-MT}^{L+MT}|u_0(x)-u_{0,n}(x)|dx\\
&+& \int\limits_{0}^T  \int\limits_{-L}^{L}|u_k^n(x,t)-u_m^n(x,t)|dxdt.
  \end{eqnarray*}
Now letting $k,m\rr \f$ and $n\rr \f$ to obtain 
$$\lim\limits_{k,m\rr \f} \int\limits_{0}^T  \int\limits_{-L}^{L}|u_k(x,t)-u_m(x,t)|dxdt=0.$$
Let $u_k\rr w$ in $L^1_{\mbox{loc}}(\R\times(0,\f)).$ Since $u_k$ are uniformly bounded, thus by Dominated convergence Theorem, $w$ is the solution of (\ref{11}), (\ref{12}) and hence $w=u$. This proves the Lemma.
\end{proof}
\noindent {\it Properties of convex functions (see \cite{12}):}\\
Let $0\leq \rho\in C^\f_0((-1,1)), \e>0, \rho_\e(x)=\frac{1}{\e}\rho(x/\e)$ be a mollifying sequence. For $g\in L^1_{\mbox{loc}}(\R)$, let 
\be\h{42}
g_\e(p)=(\rho_\e*g)(p)+\e|p|^2,
\ee
then $g_\e\in C^\f(\R)$ and satisfies the following
\begin{itemize}
\item [i.] $g_\e\rr g$ in $C^k_{\mbox{loc}}(\R)$ if $g$ is in $C^k(\R).$
\item [ii.] Let $g$ be convex, then  $g_\e$ is uniformly convex and 
\be\h{43}\sup\limits_{0<\e<1}Lip (g_\e, K)<\f,\ \mbox{for any compact set}\ K\subset \R.\ee
Furthermore if $g$ is of superlinear growth, then for $0<\e\leq 1, |p|>2\e$,
\begin{eqnarray*}
\frac{g_\e(p)}{|p|}&=& \int\limits_{|q|\leq 1} \rho(q)\frac{g(p-\e q)}{|p-\e q|}\frac{|p-\e q|}{|p|}dq\\
&\geq& \left(\inf\limits_{|z-p|\leq \e} \frac{g(z)}{|z|}\right) \inf\limits_{|q|\leq 1}\left|1-\e \frac{q}{|p|}\right|\\
&\geq& \frac{1}{2} \inf\limits_{|z-p|\leq \e} \frac{g(z)}{|z|}.
\end{eqnarray*}
Hence 
\be\h{44}
\lim\limits_{|p|\rr \f}\inf\limits_{0<\e\leq 1} \frac{g_\e(p)}{|p|}=\f.
\ee
\end{itemize}
\begin{definition}\h{definition41} Let $g$ be a convex function. Then 
\begin{itemize}
\item [i.] $g$ is said to be degenerate on an interval $I=(a,b)$ if $g$ is affine on $I$. That is there exist $\al,m\in\R$ such that for all $p\in I$
\be\h{45}
g(p)=mp+\al.
\ee
\item [ii.] $g$ is said to have finite number of degeneracies if there is a finite number $L$ of disjoint intervals $I_i=(a_i,b_i)$ such that $g$ is affine on each $I_i$ and $g$ is a strictly increasing function on $\R\setminus \cup_{i=1}^L I_i$.
\item [iii.] $g$ is said to have locally finite degeneracies if there exists $a_i<a_{i+1}$, $m_{i+1}>m_i$, $\al_i\in \R$ such that 
\be\h{46}
\lim\limits_{i\rr\f} (a_i,m_i)=(\f,\f), \ \lim\limits_{i\rr -\f}(a_i,m_i)=(-\f,-\f),
\ee
and for $p\in (a_i,a_{i+1})$
\be\h{47}
g(p) = m_ip+\al_i.
\ee
\item [iv.] Corner points: Let $g$ be a convex function. The collection of end points of maximal interval  on which $g$ is affine is called set of corner points of $g$.
\par For a convex function $g$, define the Fenchel's dual $g^*$ by 
\be\h{48}
g^*(p) =\sup\limits_q\{pq-g(q)\}.
\ee 
\end{itemize}
\end{definition}
\par From now on we assume that functions under consideration are convex and of superlinear growth. Then we have the following 
\begin{lemma}\h{lemma41} Let $g,h,\{g_k\}$ are convex functions having superlinear growth. Then 
\begin{itemize}
\item [1.] 
\be\h{49}
\lim\limits_{|p|\rr \f} \frac{g^*(p)}{|p|}=\f.
\ee
\item [2.] Let $g_k\rr g$ in $C^0_{\mbox{loc}}(\R)$ and 
\be\h{410}
\lim\limits_{|p|\rr \f}\inf\limits_{k} \frac{g_k(p)}{|p|}=\f, \mbox{ then } g^*_k\rr g^* \mbox{ in } C^0_{\mbox{loc}}(\R).
\ee
 Furthermore for any $C\geq 0$, there exists a $p_0\geq 1$ such that for $|p|>p_0$ and for all $k$
\be\h{411}
\frac{g^*_k(p)}{|p|}\geq C+1.
\ee
\item [3.] Let $g$ be $C^1$  and strictly monotone in $(a,b)$, then $g^*$ is differentiable in $(g^\p(a), g^\p(b))$ and for $p\in (a,b)$,
\be\h{412}
(g^*)^\p(g^\p(p))=p.
\ee 
\item [4.] Let $g$ be $C^1$ and having finite number of degeneracies.  $J_i=(a_i,b_i)$ for $1\leq i \leq L$, with 
\be\h{413}
g(p)=m_ip+\al_i \ \mbox{for} \ p\in J_i.
\ee
Then $g^*$ is strictly convex and $C^1(\R\setminus\{m_1,\cdots, m_L\})$ such that 
\be
(g^*)^\p(g^\p(p))&=&p,\ \mbox{if}\ p\in \R\setminus\{m_1,\cdots, m_L\}, \h{414}\\
g^*(g^\p(p))&=& pg^\p(p)-g(p), \ \mbox{for}\ p\in\R\h{415}.
\ee
\item [5.] Let $g$ be a locally finite degenerate convex function with $a_i<a_{i+1},$ $m_{i+1}>m_i$, $\al_i\in\R$ such that (\ref{46}) and (\ref{47}) holds. Then 
\be\h{416}
g^*(p)=a_ip-g(a_i)\ \mbox{for}\ p\in[m_{i-1},m_i].
\ee 
\item [6.] Assume that $g$ is a convex function with finite number of degeneracies. Let $E\subset \R$ be a finite set. Then there exists a sequence $\{g_k\}$ of convex functions having locally finite degeneracies  such that $g_k\rr g$ in $C^0_{\mbox{loc}}(\R),$ and for all $k$,
\be
E\cup \{\mbox{corner points of}\ g\}\subset \{\mbox{corner points of }\ g_k\}, \h{417}\\
\lim\limits_{|p|\rr \f}\inf\limits_{k}\frac{g_k(p)}{|p|}=\f.\h{418}
\ee
\end{itemize}
\begin{proof} For $q$ fixed, we have 
\begin{itemize}
\item [1.] 
\begin{eqnarray*}
g^*(p) &=& \sup\limits_{r}\{pr-g(r)\}\\
&\geq & pq -g(p).
\end{eqnarray*}
Hence $$\varliminf\limits_{|p|\rr \f} \frac{g^*(p)}{|p|}\geq \pm q.$$
Now letting $\pm q \rr \f$ to obtain (\ref{410}).
\item [2.] Let $|p|\leq p_0$, then 
\be\h{419}
g^*_k(p)=\sup\limits_{q}\{pq-g_k(q)\}\geq -g_k(0). 
\ee
From (\ref{410}) choose $q_0$ such that for all $|q|\geq q_0$, $|p|\leq p_0$, for all $k$, $pq-g_k(q)<-\min\limits_k g_k(0)$, $pq-g(q)<-g(0)$. Then from (\ref{419}), there exist $|q_k|\leq q_0$, $|\ti{q}|\leq q_0$ such that 
\begin{eqnarray*}
g^*(p)&=&\sup\limits_{|q|\leq q_0}\{pq-g(q)\}=p\ti{q}-g(\ti{q})\\
g_k^*(p)&=&\sup\limits_{|q|\leq q_0}\{pq-g_k(q)\}=pq_k-g(q_k).
\end{eqnarray*}
Hence 
\begin{eqnarray*}
g^*(p)-g^*_k(p)&\geq& pq_k-g_k(q_k)-pq_k+g(q_k)\\
&=& g_k(q_k)-g(q_k)\geq -\sup\limits_{|q|\leq q_0}|g(q)-g_k(q)| 
\end{eqnarray*}
and similarly 
$$g^*_k(p)-g^*(p)\geq -\sup\limits_{|q|\leq q_0} |g_k(q)-g(q)|.$$
Therefore $$\sup\limits_{|p|\leq p_0}|g_k^*(p)-g^*(p)|\leq \sup\limits_{|q|\leq q_0}|g_k(q)-g(q)|\rr 0, \ \mbox{as}\ k\rr \f.$$
Thus $g_k^*\rr g^*$ in $C^0_{\mbox{loc}}(\R).$ Let $q=(C+2)sign \frac{p}{|p|},$ then 
$$\frac{g^*_k(p)}{|p|}\geq C + 2 -\frac{g_k(sign \frac{p}{|p|}(C+2))}{|p|}.$$
Let $p_0>0$ such that $\left|\displaystyle\frac{g_k(sign \frac{p}{|p|}(C+2))}{|p|}\right|<1$ for all $|p|>p_0,$ then for $|p|>p_0,$
$$\frac{g^*_k(p)}{|p|}\geq C+2-1=C+1.$$
This proves (\ref{411}) and hence (2).
\item [3.] Denote $g_\pm^{*^{\p}}(p)$ the right and the left derivatives of $g^*$. Let $g_\e$ be as in (\ref{42}), then $g_\e$ are uniformly convex, $C^2$ function converging to $g$ in $C^0_{\mbox{loc}}(\R).$ From (\ref{44}) and (\ref{410}), $g^*_\e\rr g^*$ in $C^0_{\mbox{loc}}(\R)$. Let $q\in (a,b)$ and $a<q_1<r_1<q<r_2<q_2<b.$ Since $g_\e^{*^{\p}}(g_\e^\p(\T))=\T$
for all $\T\in \R$, we have by convexity
\begin{eqnarray*}
q_1=g_\e^{*^{\p}}(g_\e^\p(q_1))&\leq& \frac{g_\e^*(g_\e^{\p}(r_1))-g^*_\e(g_\e^{\p}(q))}{g_\e^\p(r_1)-g_\e^\p(q)}\\&\leq& \frac{g_\e^*(g_\e^{\p}(r_2))-g^*_\e(g_\e^{\p}(q))}{g_\e^\p(r_2)-g_\e^\p(q)}\\
&\leq& g_\e^{*^{\p}}(g_\e^\p(q_2))=q_2.
\end{eqnarray*}
Now for $q\in (a,b)$ and letting $\e\rr 0, r_1\uparrow q, r_2\downarrow q$ to obtain 
$$q_1\leq g_-^{*^{\p}}(g^\p(q))\leq g_+^{*^{\p}}(g^\p(q))\leq q_2.$$
Letting $q_1,q_2 \rr q$ to obtain, 
$$g^{*^{\p}}(g^\p(q))=q.$$
This proves (3).
\item [4.] From (\ref{412}), $g^*$ is in $C^1(\R\setminus\{m_1,\cdots,m_L\})$ and satisfies 
$$g^{*^{\p}}(g^\p(p))=p,\ \mbox{for}\ p\in \R\setminus\{m_1,\cdots, m_L\}.$$
Hence $g^*$ is strictly convex. Let $g_\e$ be as in (\ref{42}), then $g_\e$ satisfies 
$$g^*_\e(g^\p_\e(p))=pg_\e^\p(p)-g_\e(p).$$
Then from (2), letting $\e\rr 0$ to obtain 
$$g^*(g^\p(p))=pg^\p(p)-g(p).$$
This proves (4).
\item [5.] From direct calculations, (\ref{416}) follows.
\item [6.] Let $F=E\cup \{\mbox{corner points of} \ g\}$. For $k>1$, define sequences $\{a_i^k\},\ \{m_i^k\},\ \{g_k\}$ such that 
$$a_i^k<a_{i+1}^k,\ m_i^k\geq m_{i-1}^k \lim\limits_{i\rr \pm\f} a_i^k=\pm\f,\ F\subset \{a_i^k\}_{i=-\f}^{\f}.$$
If $a_i^k$ and $a_{i+1}^k$ are not corner points of $g$, then choose
\begin{eqnarray*}
|a_i^k-a_{i+1}^k|&<&1/k.\\
m_i^k&=& \frac{g(a_i^k)-g(a_{i-1}^k)}{a_i^k-a_{i-1}^k}.\\
g_k(p)&=&g(a_i^k)+m_i^k(p-a_i^k)\ \mbox{if}\ p\in (a_{i-1}^k, a_i^k).
\end{eqnarray*}
Then $\{g_k\}$ is the required sequence. This proves (6) and hence the Lemma.
\end{itemize}
\end{proof}
\end{lemma}

\begin{proof}[Proof of Lemma \ref{lemma23}]
(1) and (2) follows from (1) and (2) of Lemma \ref{lemma41}. 
\par Let $v(x,t)=v(x,t,f)$. Taking $y=x$ to obtain 
\be\h{435}
v(x,t)\leq v_0+tf^*(0).
\ee 
\begin{eqnarray*}
v_0(y)+tf^*\left(\frac{x-y}{t}\right)&=& v_0(x)+v_0(y)-v_0(x)+tf^*\left(\frac{x-y}{t}\right)\\
&\geq& v_0(x) -M|x-y|+tf^*\left(\frac{x-y}{t}\right)\\
&=& v_0(x) +t\left[-M\left|\frac{x-y}{t}\right|+f^*\left(\frac{x-y}{t}\right)\right].
\end{eqnarray*}
Hence for $\left|\frac{x-y}{t}\right|>p_0$, 
\be\h{436}
v_0(y)+tf^*\left(\frac{x-y}{t}\right)>v_0(x)+tf^*(0).
\ee
Hence from (\ref{435}) and (\ref{436})
\begin{eqnarray*}
v(x,t)&=& \inf\limits_{\left|\frac{x-y}{t}\right|\leq p_0}\left\{v_0(y)+tf^*\left(\frac{x-y}{t}\right)\right\}\\
&=& \min\limits_{\left|\frac{x-y}{t}\right|\leq p_0}\left\{v_0(y)+tf^*\left(\frac{x-y}{t}\right)\right\}.
\end{eqnarray*}
Therefore $ch(x,t,f)\neq \phi$. For $0< s<t$, (\ref{215}), (\ref{216}) follows from the Dynamic programming principle. This proves (3).
\par Let $y\in ch(x,t,f)$ and $v(x,t)=v(x,t,f).$ Let $0<s<t$, $r(\T)=r(\T, x,y,t),$ $r(s)=\xi$. Then $\xi\in ch(x,s,t,f)$ and $y\in ch(\xi,s, f)$. Since $\frac{x-\xi}{t-s}=\frac{x-y}{t}=\frac{\xi-y}{s}$, thus
\begin{eqnarray*}
v(x,t)&=& v_0(y)+ tf^*\left(\frac{x-y}{t}\right)\\
&=& v_0(y)+sf^*\left(\frac{\xi-y}{s}\right)+(t-s)f^*\left(\frac{x-\xi}{t-s}\right)\\
&\geq& v(\xi,s)+(t-s)f^*\left(\frac{x-\xi}{t-s}\right)\\
&\geq& v(x,t).
\end{eqnarray*}
Hence $\xi\in ch(x,s,t,f)$ and $y\in ch(\xi,s,t)$.
\par Let $\xi\in ch(x,s,t,f)$ and $y\in ch(\xi,s,t)$. Then $y\in ch(x,t,f)$. Moreover if $f^*$ is strictly convex, then $(x,t), (\xi,s)$ and $(y,0)$ lies on the same straight line. For 
\begin{eqnarray*}
v(x,t,f)&=&v(\xi,s)+(t-s)f^*\left(\frac{x-\xi}{t-s}\right)\\
&=& v_0(y)+sf^*\left(\frac{\xi-y}{s}\right)+(t-s)f^*\left(\frac{x-\xi}{t-s}\right)\\
&\geq& v_0(y) +tf^*\left(\frac{x-y}{t}\right)\\
&\geq& v(x,t,f).
\end{eqnarray*}
Hence $y\in ch(x,t,f)$ and if $f^*$ is strictly convex, then $(x,t), (\xi,s), (y,0)$ lie on the same straight line.
\par Let $x_1<x_2$ and for $i=1,2$, $y_i\in ch(x_i,t,f)$ and $\g_i(\T)=\g(\T,x_i,0,t,y_i)$. Suppose $y_2<y_1$, then there exist $0<s<t$ such that $\xi=\g_1(s)=\g_2(s).$ Then from the above analysis, $\xi\in ch(\xi,s,f)$. Hence $y_1\in ch(x_2,t,f)$. Furthermore if $f^*$ is strictly convex, then $y_1\leq y_2$ and hence $\g_1$ and $\g_2$ never intersect in $(0,t)$.
\par As a consequence of this, $y_+(x_1,t,f)\leq y_+(x_2,t,f)$. Similarly $y_-(x_1,t,f)\leq y_-(x_2,t,f).$ This proves (\ref{219}) to (\ref{223}). Similar proof follows for $0<s<t$ and this proves (4).
\par Let $p_1=\sup\limits_n\{f_n(C+2), f_n(-(C+2))\},$ then for $|p|>p_1$, $q=\frac{p}{|p|}(C+2),$ we have 
$$\frac{f^*_n(p)}{|p|}\geq (C+2) -\frac{f_n\left(\frac{p}{|p|}(C+2)\right)}{|p|}\geq C+1.$$
This proves (\ref{224}).
\par Let $s=0$, $y_n\in ch(x,t,f_n)$, $v_n(x,t)=v(x,t,f_n),$ $v(x,t)=v(x,t,f)$. Since $\left|\frac{x-y_n}{t}\right|\leq p_0$, consequently  for a subsequence  let $y_n\rr y_0$. Then for $y\in \R$, we have 
\begin{eqnarray*}
v_0(y_0)+tf^*\left(\frac{x-y_0}{t}\right)&=&\lim\limits_{n\rr \f}\left\{ v_0(y_n)+tf^*_n\left(\frac{x-y_n}{t}\right)\right\}\\
&\leq&\lim\limits_{n\rr \f}\left\{ v_0(y)+tf^*_{n}\left(\frac{x-y}{t}\right)\right\}\\
&\leq&\left\{ v_0(y)+tf^*\left(\frac{x-y}{t}\right)\right\}.
\end{eqnarray*}
Hence $y_0\in ch(x,t,f)$ and $v_n(x,t)\rr v(x,t)$ as $n\rr \f$. Since $\sup\limits_n Lip(v_n, \R)\leq ||u_0||_\f$, thus $\{v_n\}$ is an equicontinuous family. Hence from Arzela-Ascoli, $v_n\rr v$ in $C^0_{\mbox{loc}}(\R).$ This proves (\ref{224}) to (\ref{226}) when $s=0$. Similar proof follows from the Dynamic programming principle. This proves (5).
\par Let $s=0$ and $(x_n,t_n)\rr (x,t)$ and $y_n\in ch(x,t_n,f_n)$ and $y_n\rr y$ as $n\rr \f$. Then from (\ref{225})
\begin{eqnarray*}
v(x,t)&=&\lim\limits_{n\rr \f} v_n(x_n,y_n)\\
&=&\lim\limits_{n\rr \f}\left\{ v_0(y_n)+t_nf^*_{n}\left(\frac{x-y_n}{t_n}\right)\right\}\\
&=&\left\{ v_0(y)+tf^*\left(\frac{x-y}{t}\right)\right\}.
\end{eqnarray*}
Therefore $y\in ch(x,t,f).$ Similarly for $s>0$ and this proves (6).
\par Since $u_{0,n}\rightharpoonup u_0$ in $L^\f \mbox{ weak }^*$ topology, therefore $v_{0,n}(x)\rr v_0(x)$ for all $x\in \R$. Since $\sup\limits_{n}Lip (v_{0,n},\R)\leq \sup\limits_{n}||u_{0,n}||_{\f}<\f$, thus $\{v_{0,n}\}$ is an equicontinuous sequence and hence converges in $C^0_{\mbox{loc}}(\R).$ Let $y_n\in ch(x,t,f_n)$ such that $y_n\rr y$. Then for any $\xi\in \R$, we have 
\begin{eqnarray*}
v_n(x,t)&=& v_{0,n}(y_n)+tf^*_{n}\left(\frac{x-y_n}{t}\right)\\
&\leq&v_{0,n}(\xi)+tf_n^*\left(\frac{x-\xi}{t}\right).
\end{eqnarray*}
 Letting $n\rr \f$ to obtain 
 $$ v_0(y)+tf^*\left(\frac{x-y}{t}\right)\leq v_0(\xi)+tf^*\left(\frac{x-\xi}{t}\right).$$
Thus $y\in ch(x,t,f)$. This proves (7) and hence the Lemma.
\end{proof}
\noindent {\it Riemann problem for piecewise convex flux:}
Let $a_i<a_{i+1},$ $m_i<m_{i+1}, \al_i\in \R$ such that 
\be
\lim\limits_{i\rr \pm \f}(a_i,m_i)&=&(\pm \f, \pm \f)\h{437}\\
m_{i-1}a_i+\al_{i-1}&=& m_ia_i+\al_i\h{438}.
\ee
Define 
\be\h{439}
g(p)=m_ip+\al_i, \ \mbox{for}\ p\in [a_i,a_{i+1}].
\ee
Then from (\ref{437}), (\ref{438}), $g$ defines a super linear piecewise affine convex function with 
\be\h{440}
g^*(p)=a_ip-g(a_i) \ \mbox{for}\ p\in[m_{i-1},m_i].
\ee
\noindent {\it Riemann problem:} Let $x_0,a,b,\in \R$ and 
\be\h{441}
u_0(x)=\left\{\begin{array}{lll}
a &\mbox{if}& x<x_0,\\
b &\mbox{if}& x>x_0.\end{array}\right.
\ee
\be\h{442}
v_0(x)=\left\{\begin{array}{lll}
ax-ax_0 &\mbox{if}& x<x_0,\\
bx-bx_0 &\mbox{if}& x>x_0.\end{array}\right.
\ee
Let $u$ be the solution of (\ref{11}), (\ref{12}) with flux $g$ and initial data $u_0$ and $v$ be the corresponding value function defined in (\ref{25}). Then $u$ is given by\\
\noindent Case (i): Let $a,b\in [a_i,a_{i+1}]$. Define 
$$L(t,g)=R(t,g)=x_0+m_it,$$ 
then 
\be\h{443}
u(x,t)=\left\{\begin{array}{lll}
a &\mbox{if}& x<L(t,g),\\
b &\mbox{if}& x>L(t,g).\end{array}\right.
\ee
\be
&&ch(x,t,g)=\{x-m_it\} \h{444}\\
&& y_-(x_0+m_it,t)=y_+(x_0+m_it,t)=x_0\h{445}.
\ee
\noindent Case (ii): Let $a_j\leq b<a_{j+1}<a_i\leq a\leq a_{i+1}.$ Define 
\be\h{446}
m=\frac{f(a)-f(b)}{a-b}, L(t,g)=R(t,g)=x_0+mt
\ee
Then 
\be\h{447}
u(x,t)=\left\{\begin{array}{lll}
a &\mbox{if}& x<L(t,g),\\
b &\mbox{if}& x>R(t,g).\end{array}\right.
\ee
\be\h{448}
ch(x,t,g)=\left\{\begin{array}{lll}
x-m_it &\mbox{if}& x<L(t,g),\\
x-m_jt &\mbox{if}& x>R(t,g).\end{array}\right.
\ee
\be
&& y_-(x_0+mt,t)=x_0+(m-m_i)t<x_0\h{449}.\\
&& y_+(x_0+mt,t)=x_0+(m-m_j)t>x_0\h{450}.
\ee
\noindent Case (iii): Let $a_i\leq a <a_{i+1}<a_j\leq b \leq a_{j+1}.$ Define 
\be
&&L(t,g)=x_0+m_it, R(t,g)=x_0+m_jt\h{451}\\
&& l_k(t)=x_0+m_kt \ \mbox{for}\ i\leq k \leq j \h{452}.
\ee
Then 
\be\h{453}
u(x,t)=\left\{\begin{array}{lll}
a &\mbox{if}& x<L(t,g),\\
a_k &\mbox{if}& l_{k-1}(t)<x<l_{k+1}(t),\\
b &\mbox{if}& x>R(t,g).\end{array}\right.
\ee
\be\h{454}
y_-(x_0+m_it,t)=y_+(x_0+m_jt,t)=x_0.
\ee
All the above properties follows easily from the direct calculation.

\begin{lemma}\h{lemma44} Let $(f,C,D)$ be a convex-convex type triplet and $\al<C<D<\B$ such that 
\be\h{426}
f(C)<L_{\al,\B}(C), f(D)<L_{\al,\B}(D).
\ee
Then
\be\h{427}
f^\p(\al)<\frac{f(\B)-f(\al)}{\B-\al}<f^\p(\B).
\ee
\begin{proof}
Suppose $\frac{f(\B)-f(\al)}{\B-\al}\leq f^\p(\al),$ then by convexity of $f$ in $(-\f,C]$, for $x\in [\al,C]$.
$$\frac{f(\B)-f(\al)}{\B-\al}\leq f^\p(\al)\leq f^\p(x).$$
Integrating from $\al$ to $C$ to obtain 
$$L_{\al,\B}(C)=f(\al)+\frac{f(\B)-f(\al)}{\B-\al}(C-\al)\leq f(C),$$
contradicting (\ref{426}).
\par Suppose $f^\p(\B)\leq \frac{f(\B)-f(\al)}{\B-\al},$ then by convexity of $f$ in $[D,\f)$, for all $x\in[D,\B]$, 
$$f^\p(x)\leq f^\p(\B)\leq \frac{f(\B)-f(\al)}{\B-\al}.$$
Integrating from $D$ to $\B$ to obtain 
$$f(D)\geq f(\B)+\frac{f(\B)-f(\al)}{\B-\al}(D-\B)=L_{\al,\B}(D).$$
Contradicting (\ref{426}). This proves the Lemma.
\end{proof}
\end{lemma}
\begin{figure}[ht]        \centering
        \def\svgwidth{0.8\textwidth}
        \begingroup
    \setlength{\unitlength}{\svgwidth}
  \begin{picture}(1.4,0.76363752)%
    \put(0,0.1){\includegraphics[width=0.95\textwidth]{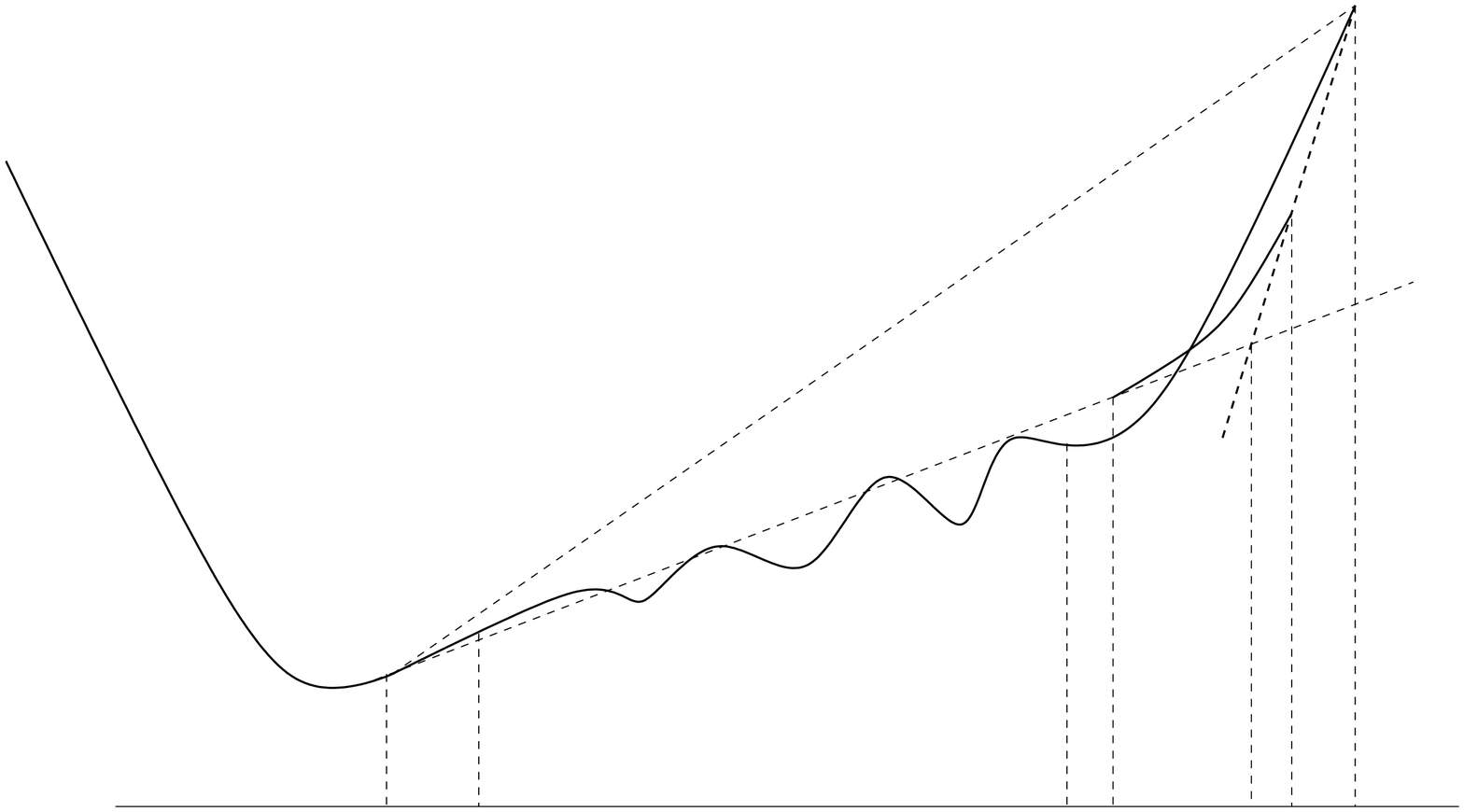}}%
       \put(0.3,0.075){\color[rgb]{0,0,0}\makebox(0,0)[lb]{\smash{$\al$}}}%
              \put(0.91,0.5){\color[rgb]{0,0,0}\makebox(0,0)[lb]{\smash{$Q(x)$}}}%
            \put(0.3729,0.075){\color[rgb]{0,0,0}\makebox(0,0)[lb]{\smash{$C$}}}%
                \put(0.84263,0.075){\color[rgb]{0,0,0}\makebox(0,0)[lb]{\smash{$D$}}}%
                    \put(.892468751,0.075){\color[rgb]{0,0,0}\makebox(0,0)[lb]{\smash{$x_1$}}}%
                                        \put(.9992468751,0.075){\color[rgb]{0,0,0}\makebox(0,0)[lb]{\smash{$d$}}}%
   \put(1.0299192468751,0.075){\color[rgb]{0,0,0}\makebox(0,0)[lb]{\smash{$x_2$}}}
    \put(1.10124,0.075){\color[rgb]{0,0,0}\makebox(0,0)[lb]{\smash{$\B$}}}%
        \put(1.16,0.517){\color[rgb]{0,0,0}\makebox(0,0)[lb]{\smash{$L_{\al}$}}}%
                \put(0.95,0.379517){\color[rgb]{0,0,0}\makebox(0,0)[lb]{\smash{$L_{\B}$}}}%
                \put(0.639,0.506097){\color[rgb]{0,0,0}\makebox(0,0)[lb]{\smash{$L_{\al,\B}$}}}%
            \put(0.0159,0.4922){\color[rgb]{0,0,0}\makebox(0,0)[lb]{\smash{$f$}}}%
  \end{picture}%
\endgroup
        \caption{Illustration for the convex modification of $f$}
        \label{Fig4}
\end{figure}
\begin{figure}[ht]        \centering
        \def\svgwidth{0.8\textwidth}
        \begingroup
    \setlength{\unitlength}{\svgwidth}
  \begin{picture}(1.4,0.6855363752)%
    \put(0,0.1){\includegraphics[width=0.95\textwidth]{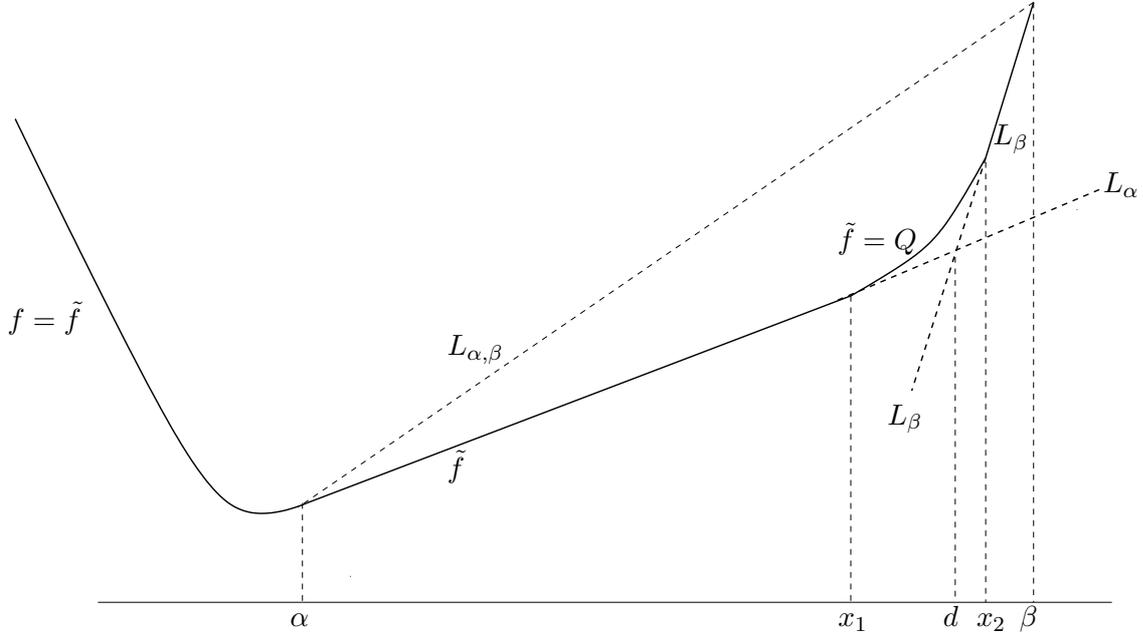}}%
                    \put(.889892468751,0.075){\color[rgb]{0,0,0}\makebox(0,0)[lb]{\smash{$x_1$}}}%
                                                    \put(.29892468751,0.075){\color[rgb]{0,0,0}\makebox(0,0)[lb]{\smash{$\al$}}}%
                                        \put(1.00299192468751,0.075){\color[rgb]{0,0,0}\makebox(0,0)[lb]{\smash{$d$}}}
                                                                        \put(.8890719943,0.485199517){\color[rgb]{0,0,0}\makebox(0,0)[lb]{\smash{$\ti{f}=Q$}}}%
    \put(1.0390910124,0.075){\color[rgb]{0,0,0}\makebox(0,0)[lb]{\smash{$x_2$}}}%
        \put(1.085910124,0.075){\color[rgb]{0,0,0}\makebox(0,0)[lb]{\smash{$\B$}}}%
        \put(1.176,0.547){\color[rgb]{0,0,0}\makebox(0,0)[lb]{\smash{$L_{\al}$}}}%
                \put(0.943,0.29517){\color[rgb]{0,0,0}\makebox(0,0)[lb]{\smash{$L_{\B}$}}}%
                                \put(1.05719943,0.599517){\color[rgb]{0,0,0}\makebox(0,0)[lb]{\smash{$L_{\B}$}}}%
                \put(0.46840639,0.370566097){\color[rgb]{0,0,0}\makebox(0,0)[lb]{\smash{$L_{\al,\B}$}}}%
                                \put(0.46840639,0.2370566097){\color[rgb]{0,0,0}\makebox(0,0)[lb]{\smash{$\ti{f}$}}}%
            \put(-0.006,0.39922){\color[rgb]{0,0,0}\makebox(0,0)[lb]{\smash{$f=\ti{f}$}}}%

  \end{picture}%
\endgroup
        \caption{Convex modification $\ti{f}$}
\label{Fig5}
\end{figure}
From Lemma  \ref{lemma44} we make a convex modification of $(f,C,D)$ as follows. \\
\par  For $i=1,2$, let $(a_i,b_i)\in \R^2$ such that 
\be\h{420}
a_1=a_2\ \mbox{whenever } b_1=b_2.
\ee
For $b_1+b_2\neq 0, x_1\in \R$, define 
\be
x_2&=& x_1+\frac{2(a_2-a_1)}{b_2+b_1}, \h{421}\\
Q(x)&=&\frac{b_2-b_1}{2(x_2-x_1)}(x-x_1)^2 +b_1(x-x_1)+a_1.\h{422}
\ee
Then 
\begin{eqnarray*}
(Q(x_1), Q^\p(x_1))&=&(a_1,b_1)\\
Q(x_2)&=& \frac{(b_2-b_1)}{2}(x_2-x_1) + b_1(x_2-x_1) + a_1\\
&=& \frac{1}{2} (b_2+b_1)(x_2-x_1) +a_1\\
&=& a_2-a_1+a_1=a_2.\\
Q^\p(x_2)&=&\frac{b_2-b_1}{x_2-x_1}+b_1=b_2.
\end{eqnarray*}
Let $\al<C<D<\B$ satisfying (\ref{426}). Let 
\begin{eqnarray*}
&&f^\p(\al)=b_1, f^\p(\B)=b_2\\
&& l_1(x)=f(\al)+f^\p(\al)(x-\al)\\
&& l_2(x)=f(\B)+f^\p(\B)(x-\B)\\
&& d=\frac{f(\al)-\al f^\p(\al)-f(\B)+\B f^\p(\B)}{f^\p(\B)-f^\p(\al)}
\end{eqnarray*}
is the point of interaction of the tangent lines $l_1$ and $l_2$. Observe that from (\ref{427}) we have 
\begin{eqnarray*}
\begin{array}{lll}
\al<d<\B,
\min(f(\al),f(\B))<l_1(d)<\max(f(\al),f(\B)).
\end{array}
\end{eqnarray*}
\par Let $\al<x_1<d<x_2<\B$ and $a_1=l_1(x_1), a_2=l_2(x_2)$ and define $Q$ as in (\ref{422}). \\
\noindent Case (i): Suppose $f^\p(\al)+f^\p(\B)=b_1+b_2=0$. Then choose $l_1(x_1)=l_2(x_2)$.
Then $(Q(x_1), Q^\p(x_1))=(a_1,b_1)$ and 
\begin{eqnarray*}
Q(x_2)&=&\frac{1}{2}(b_2-b_1)(x_2-x_1)+b_1(x_2-x_1)+l_1(x_1)\\
&=& (b_1+b_2)(x_2-x_1)+l_1(x_1)\\
&=& l_1(x_1)\\
&=& l_2(x_2).\\
Q^\p(x_2)&=& (b_2-b_1)+b_1=b_2.
\end{eqnarray*}
Hence $(Q(x_1), Q^\p(x_1), Q(x_2), Q^\p(x_2))=(a_1,b_1,a_2,b_2)$  and uniformly convex function.\\
\noindent Case (ii): Let $f^\p(\al)+f^\p(\B)=b_1+b_2\neq 0$. Then 
\begin{eqnarray*}
a_2-a_1&=& l_2(x_2)-l_1(x_1)\\
&=& f(\B)-f(\al)-\B f^\p(\B)-\al f^\p(\al)+f^\p(\B)x_2-f^\p(\al)x_1\\
&=& -d(f^\p(\al)+f^\p(\B))+f^\p(\B)x_2-f^\p(\al)x_1\\
&=& (f^\p(\B)+f^\p(\al))(x_2-d)+f^\p(\al)(2d-x_1-x_2)\\
&=& f^\p(\B)(x_1+x_2-2d) +(f^\p(\al)+f^\p(\B))(d-x_1).
\end{eqnarray*}
Suppose $f^\p(\al)+f^\p(\B)>0$, then from (\ref{427}), $f^\p(\B)>0$. Choose $x_1$ and $x_2$ such that $x_1+x_2>2d$. Then $a_2-a_1>0.$ Suppose $f^\p(\al)+f^\p(\B)<0$, then from (\ref{427}), $f^\p(\al)<0$. Choose $x_1$ and $x_2$ such that $x_1+x_2\leq 2d$. Then $a_2-a_1>0$. Hence $Q(x)$ as in (\ref{422}) is a uniformly convex function satisfying $(Q(x_1),Q^\p(x_1), Q(x_2),Q^\p(x_2))=(a_1, b_1,a_2,b_2)$.

\noindent Case (iii): $f^\p(\B)\leq 0.$ Then from (\ref{427}), $f(\al)>f(\B)$, $f^\p(\al)<0$ and $l_2(x)<l_1(x)$ for $x>d$. From Lemma \ref{lemma44}, choose $\al<x_1<d<x_2<\B$, $a_1=l_1(x_1)$, $a_2=l_2(x_2)$ such that $Q$ is a smooth convex function satisfying $(Q(x_1),Q^\p(x_1), Q(x_2),Q^\p(x_2))=(a_1, b_1,a_2,b_2)$. Define $\ti{f}$ by
\be\h{428}
\ti{f}(x)=\left\{\begin{array}{llllll}
f(x) &\mbox{if}& x\notin(\al,\B),\\
l_1(x) &\mbox{if}& \al\leq x \leq  x_1,\\
Q(x) &\mbox{if}&  x_1 \leq  x \leq  x_2,\\
l_2(x) &\mbox{if}& x_2 \leq x \leq \B. \end{array}\right.
\ee
Then $\ti{f}$ is in $C^1(\R)$, convex and satisfy $\ti{f}(x)=f(x)$, for $x\notin (\al,\B)$.
 
 \noindent \textit{Convex modification (see figures \ref{Fig4} and \ref{Fig5}):} Given $(f,C,D)$ of convex-convex type triplet and $\al<C<D<\B$ satisfying (\ref{426}). Then $\ti{f}$ constructed above is called a convex modification of $f$ with respect to $\al,\B$.

\noindent \textbf{Acknowledgements:} Both the authors would like to thank Gran Sasso Science Institute, L'Aquila, Italy, as  the work has been initiated  from there. First author acknowledge the support from Rajaramanna fellowship.  The second author would like to thank Inspire research grant for the support.

 \end{document}